\documentclass[12pt]{article}
\usepackage{amsfonts,amsthm,amssymb,mathrsfs,setspace,pstricks,booktabs,xcolor,mathtools,amsmath,geometry,graphicx,subcaption}
\usepackage{tikz}
\usetikzlibrary{arrows.meta, positioning}
\usepackage{scalerel}
\usepackage{hyperref}
\hypersetup{
	colorlinks=true,
	linkcolor=blue,
	citecolor=red,
	filecolor=red,      
	urlcolor=magenta,
}

\theoremstyle{plain}
\newtheorem{theorem}{Theorem}[section]

\newtheorem{proposition}[theorem]{Proposition}

\theoremstyle{definition}
\newtheorem{definition}[theorem]{Definition}

\theoremstyle{remark}
\newtheorem{remark}[theorem]{Remark}

\begin{document}
	
	\title{Minimal Graph Transformations and their Classification}
	\author{Sam K Mathew}
	\date{}
	\maketitle
	
\begin{abstract}
	This paper presents a complete classification of minimal graph surfaces that admit graphical transformations into other minimal surfaces. These transformations are functions that map the height function of a minimal graph surface to another height function, which also describes a minimal graph surface. While trivial maps such as translations and reflections exist, we formulate and solve the Non‑Trivial Minimal Graph Transformation Problem, governed by a coupled system of partial differential equations. A central result establishes the rigorous equivalence of this original system to a modified problem for a harmonic function. Through a complex variable approach and a weakening technique, the analysis is reduced to solving a fundamental ordinary differential equation parameterized by a real constant \(k\). The explicit integration of this ODE involves various elliptic integrals and identities of elliptic functions. Solving the ODE for the three cases \(k = 0\), \(k > 0\), and \(k < 0\) yields the full classification of all admissible surfaces and their associated transformations. This process yeilds several classes of minimal surfaces that, to the best of the author’s knowledge, constitute new families of minimal surfaces.
\end{abstract}
	
	\section{Introduction}
	
	 The study of transformations of minimal surfaces has been a central theme in differential geometry of minimal surface for a long time, beginning with the deformation family of helicoid to catenoid, where helicoid can be deformed to a catenoid by including an \(e^{i\theta}\) inside the Weiersterass Enneper integral representation of these minimal surfaces. Infact for any minimal surface there is a such a deformation family of minimal surface, which is precisely its associated family. Over time, this field has expanded to encompass a rich variety of geometric and analytic methods.
	 
	 Thybaut \cite{Thy} established that minimal surfaces admit transformations analogous to the Bäcklund transformations of pseudospherical surfaces, where the surface and its transform serve as focal sheets of a $W$-congruence. Eisenhart \cite{Eis} established a transformation for the broader class of surfaces characterized by an isothermal spherical representation of their lines of curvature. Notably, this transformation is minimality preserving, when the initial surface is minimal, its transform remains minimal. More recently, Corro, Ferreira, and Tenenblat \cite{CFT} have utilized Ribaucour transformations to expand the family of known minimal surfaces by relating these transformations to the generation of planar embedded ends . By applying these methods to classical examples like Enneper’s surface and the catenoid, they obtained new families of complete, genus-zero minimal surfaces in $\mathbb{R}^3$ characterized by a varying number of embedded planar ends.
	 
	 In a more recent study, B{\"a}ck~\cite{PB} examined B{\"a}cklund transformations for minimal surfaces, showing how such transformations act on a given surface; for instance, the catenoid to produce an infinite family of new minimal surfaces. These transformations function as mappings between solutions of the elliptic Liouville equation. By combining classical geometric ideas with novel algebraic techniques, this work provides a systematic framework for generating new solutions that had not been identified previously in the literature.
	 
	 In this paper, we address a question that is fundamentally similar to earlier works mentioned above: Given a minimal graph surface \(f\), we answer the question of when does there exist a smooth real‑valued function \(g\) such that the composition \(g\circ f\) is again a minimal graph surface. To put this in more formal terms, we define the following terminologies:
	 
	 \begin{definition}
	 	Let \(\Omega\) be an open connected set in \(\mathbb{R}^2\) and \(f: \Omega \subseteq \mathbb{R}^2 \longrightarrow \mathbb{R}\) be a Smooth Function. Then \(f\) is called as a \textbf{Minimal Graph Surface} if the graph of \(f\) forms a Minimal Surface in the Euclidean Space \(\mathbb{E}^3\). 
	 	
	 	Alternatively a Minimal Graph Surface can also be defined using the following Minimal Surface Equation:
	 	\begin{equation}
	 		\label{eqn: MSE}
	 		(1 + f_y^2) f_{xx} - 2 f_{xy} f_x f_y + (1 + f_x^2) f_{yy} = 0
	 	\end{equation}
	 \end{definition}
	 
	 \begin{definition}
	 	A Smooth Map \(g: U \subseteq \mathbb{R} \longrightarrow \mathbb{R}\) is called a \textbf{Minimal Graph Transformation} of the Minimal Graph Surface \(f: \Omega \longrightarrow \mathbb{R} \) if \(f(\Omega)\subseteq U\) and the composition
	 	\[
	 	g \circ f:\Omega\longrightarrow \mathbb{R}
	 	\]
	 	is again a Minimal Graph Surface.
	 \end{definition}
	 
	 In other words, the function $g$ transforms the height function $f(x,y)$ at each point $(x,y)$ of a Minimal Graph Surface $(x,y,f(x,y))$ into another height function $(x,y, (g \circ f)(x,y))$, which also describes a Minimal Graph Surface. We  are now ready to state the Problem Statement.
	 \subsection{Problem Statement:}
	 \begin{enumerate}
	 	\item Find all maps $g:\mathbb{R}\longrightarrow\mathbb{R}$ such that $g$ is a Minimal Graph Transformation for \emph{every} Minimal Graph Surface $f: \Omega \longrightarrow \mathbb{R}$. We call these type of maps \(g\) as \textbf{Trivial Minimal Graph Transformations (TMG Transformations)}. Ofcourse such maps do exist, for instance the affine maps of the form $g(x) = \pm x + C$ for any constant $C \in \mathbb{R}$. So the precise problem is to determine whether these are the only TMG Transformations?  Note that this universal requirement is highly restrictive, suggesting that there might be other maps $g$ that act as a Minimal Graph Transformations for a particular Minimal Graph Surface $f$ but not for all of them. Hence, we ask the following:
	 	
	 	\item Find all Minimal Graph Surfaces \(f: \Omega \subseteq \mathbb{R}^2 \longrightarrow \mathbb{R}\) that admit a minimal graph transformation other than the Trivial Minimal Graph Transformations. We call such a map a \textbf{Nontrivial Minimal Graph Transformations (NMG Transformations)} of the Minimal Graph Surface \(f\). In fact, such Minimal Graph Surfaces exist. For example, the flat Minimal Graph Surface \(f(x,y)=0\) admits NMG Transformations, since for any function \(g:\mathbb{R}\longrightarrow\mathbb{R}\), the composition \(g \circ f\) becomes the constant function \(g(0)\), and a constant function always defines a Minimal Graph Surface. However, in general, \(g \circ f\) is not a Minimal Graph Surface for an arbitrary map \(g\) and a given Minimal Graph Surface \(f\). The correct question to ask is: Do all Minimal Graph Surface admit at least one NMG Transformations? Furthermore, for a  Minimal Graph Surface \(f\) that admits an NMG Transformation, find all of its NMG Transformations.
	 \end{enumerate}
	 
	 Note that, since the action of a minimal graph transformation is to transform the height function of a minimal graph surface into the height function of another minimal graph surface, a minimal graph transformation does not change the contour lines of the surface. Conversely, if two minimal surfaces share the same contour lines, then there exists a function that transforms one surface into the other by changing the height at which each contour line lies. Hence, in a certain sense, this paper is concerned with finding all families of minimal graph surfaces that share the same contour lines.  
	 
	 \subsection{Methodology}
	 The idea is first to derive a mathematical condition for a minimal graph surface \(f\) to admit a minimal graph transformation \(g\), and to determine the equation that \(g\) must satisfy; this is accomplished in Proposition~\ref{thm: Alt. def. MGT}. From this, we seek those functions \(g\) that satisfy the condition for every minimal graph surface; these are the trivial minimal graph transformations and are the subject of Theorem~\ref{thm: Trivial MGT}. In Theorem~\ref{thm: MGT of Const MGS} we observe that constant minimal graph surfaces can be transformed to another constant minimal graph surface by any function.
	 
	 Next we consider nontrivial minimal graph transformations of non‑constant minimal graph surfaces. For a non‑constant minimal graph surface to admit a NMG transformation it must obey Equation~\ref{eqn: Complex Non-Trivial MGT: 2}. Given the function \(h\) appearing in that equation, we provide an expression for the NMG transformation in terms of \(h\), valid for every minimal graph surface satisfying \ref{eqn: Complex Non-Trivial MGT: 2}. We then examine whether any minimal graph surface admits a NMG transformation with \(h=0\); it turns out that such surfaces do exist, they are precisely helicoids and planes, and the transformed minimal graph surface is again a helicoid or a plane, but with different parameters.
	 
	 For the subsequent analysis we assume \(h\) is a non‑zero function. Under this assumption we reduce system~\ref{eqn: Complex Non-Trivial MGT: 2} to the simpler system~\ref{eqn: Modified Complex Non-Trivial MGT: 2}, and prove in Theorem~\ref{thm: Equivalence thm} that the two systems are equivalent. We further reduce this system to a single complex ordinary differential equation, \ref{eqn: weakened eqn 2}, which is weaker than \ref{eqn: Modified Complex Non-Trivial MGT: 2} in the sense that it may have more solutions than the original one. The strategy is then to solve this equation explicitly and afterwards verify which of its solutions also satisfy system~\ref{eqn: Modified Complex Non-Trivial MGT: 2}. 
	 
	Three cases emerge from this equation, depending on a real parameter \(k\): \(k=0\), \(k>0\), and \(k<0\).  
	For \(k=0\) the admissible minimal graph surface is Scherk’s surface; the NMG transformation essentially changes a parameter and yields another Scherk’s surface with a different parameter.  
	For \(k>0\) we obtain the catenoid together with another surface given by Equation~\ref{eqn: soln k>0, c_0 neq 0, f}, and their NMG transformations are likewise parameter changes.  
	For \(k<0\) there are three distinct minimal graph surfaces, described precisely by expressions \ref{eqn: soln k<0, c_0 neq 0, f: case 1}, \ref{eqn: soln k<0, c_0 neq 0, f: case 2} and \ref{eqn: soln k<0, c_0 neq 0, f: case 3}. In contrast to the previous cases, the NMG transformations here do not merely alter the internal parameters; they map any one of these three surfaces onto any other of the three. One of the main tools that proves essential is the theory of elliptic integrals and elliptic functions. In particular, various identities involving Jacobi elliptic functions are repeatedly employed to obtain the complete list of solutions for this ODE.
	
	We provide a flowchart that summarises the techniques used and the results obtained in this paper. 
	 
	 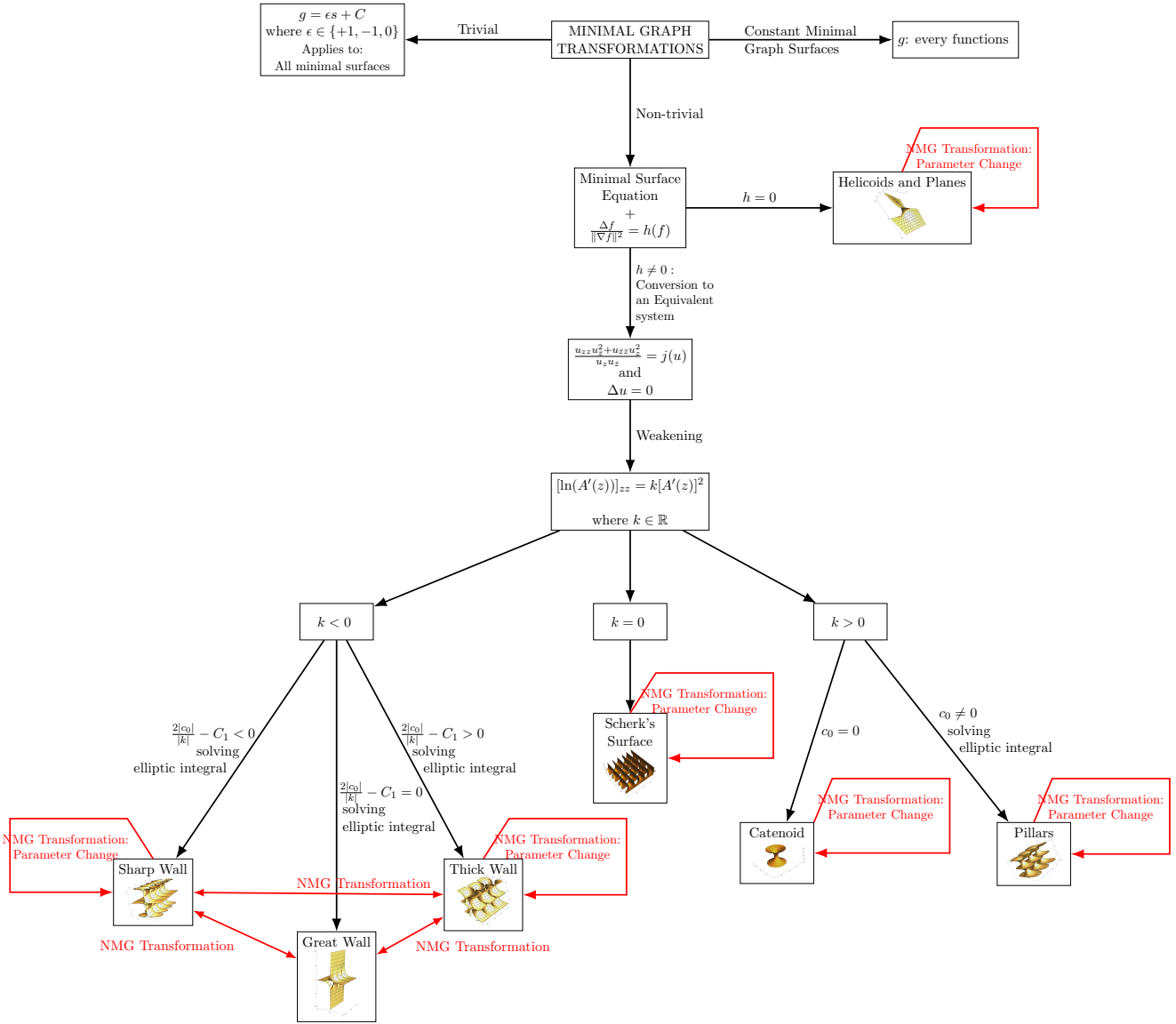
\begin{figure}[h]
	 	\centering
	 	\makebox[470pt][r]{
	 		\resizebox{1.1\textwidth}{!}{
	 			\begin{tikzpicture}[
	 				block/.style={
	 					rectangle, 
	 					draw, 
	 					minimum width=2cm, 
	 					minimum height=1cm, 
	 					align=center, 
	 					font=\small
	 				},
	 				arrow/.style={-{Latex[scale=1.2]}, very thick},
	 				dual_arrow/.style={Latex-Latex, very thick, draw=red}
	 				]
	 				
	 				\node (MGT) [block] {MINIMAL GRAPH \\ TRANSFORMATIONS};	
	 				
	 				\node (Constant) [block, right=5cm of MGT] {
	 					\(g\): every functions
	 				};		
	 				
	 				\node (TMGT) [block, left=4cm of MGT] {
	 					$g = \epsilon s + C$ \\ 
	 					where \(\epsilon\in\{+1,-1,0\}\)\\
	 					\footnotesize Applies to: \\ 
	 					\footnotesize All minimal surfaces
	 				};			
	 				
	 				\node (MSE) [block, below=3cm of MGT] {
	 					Minimal Surface \\ Equation \\ + \\ 
	 					$\tfrac{\Delta f}{\|\nabla f\|^2} = h(f)$
	 				};			
	 				
	 				\node (helicoids) [block, right=4cm of MSE] {
	 					Helicoids and Planes\\ \includegraphics[width=1.5cm]{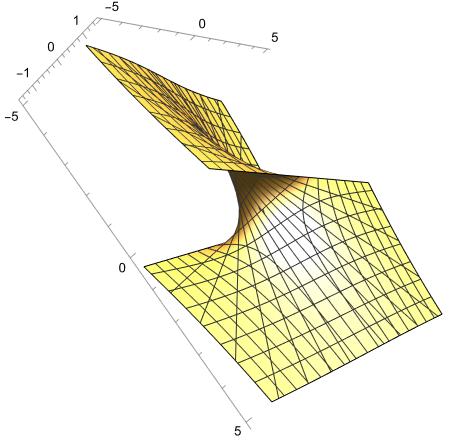}
	 				};
	 				
	 				\node (Equivalent) [block, below=2.5cm of MSE] {
	 					$\displaystyle \tfrac{u_{zz}u_{\bar{z}}^2 + u_{\bar{z}\bar{z}}u_{z}^2}{u_z u_{\bar{z}}} = j(u)$ \\
	 					and \\
	 					$\Delta u = 0$
	 				};
	 				
	 				\node (Weak) [block, below=2cm of Equivalent] {
	 					$[\ln(A'(z))]_{zz}=k[A'(z)]^2$\\\\
	 					where $k\in \mathbb{R}$
	 					
	 				};
	 				
	 				\node (k-neg) [block, below=2cm of Weak, xshift=-8cm] {
	 					$k<0$
	 				};
	 				
	 				\node (k-zero) [block, below=2cm of Weak] {
	 					$k=0$
	 				};
	 				
	 				\node (k-pos) [block, below=2cm of Weak, xshift=6cm] {
	 					$k>0$
	 					
	 				};
	 				
	 				\node (Scherk) [block, below=2cm of k-zero] {
	 					Scherk's\\
	 					Surface \\ \includegraphics[width=1.5cm]{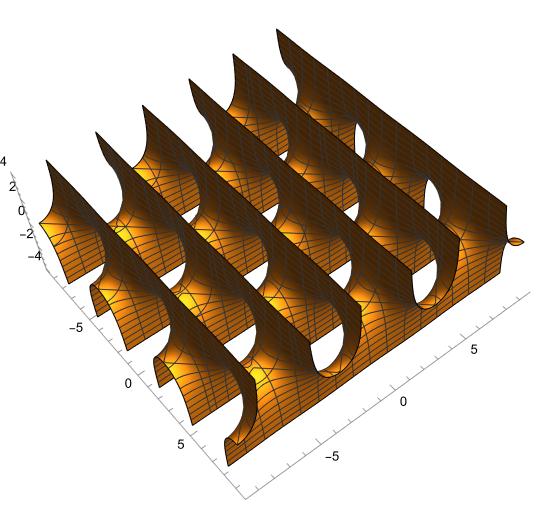}
	 				};
	 				
	 				\node (Catenoid) [block, below=5cm of k-pos, xshift=-2cm] {Catenoid \\ \includegraphics[width=1.5cm]{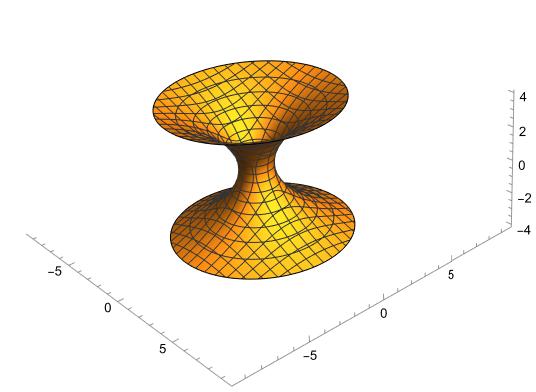}
	 				};
	 				
	 				\node (Pillars) [block, below=5cm of k-pos, xshift=5cm] {Pillars\\ \includegraphics[width=1.5cm]{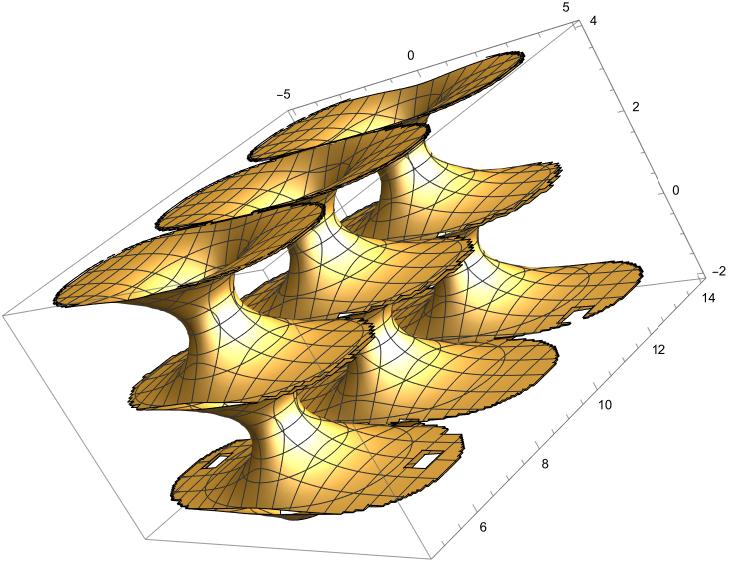}
	 				};
	 				
	 				\node (Great Wall) [block, below=8cm of k-neg, xshift=0cm] {Great Wall\\ \includegraphics[width=1.2cm]{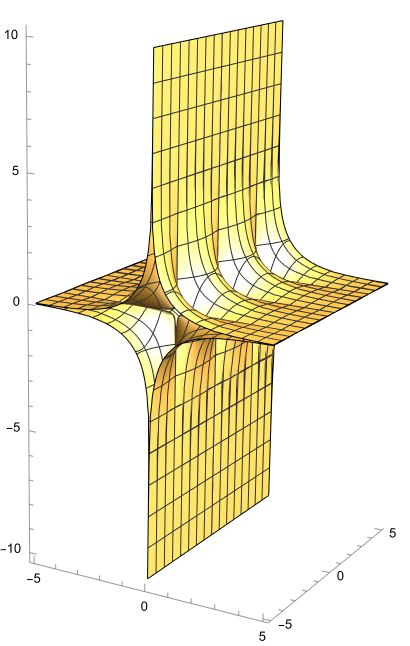}
	 				};
	 				
	 				\node (Thick Wall) [block, below=6cm of k-neg, xshift=4cm] {Thick Wall\\ \includegraphics[width=1.5cm]{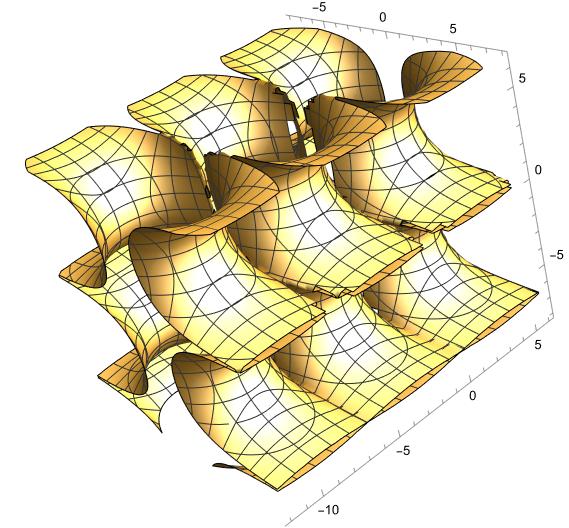} 
	 					
	 				};
	 				
	 				\node (Sharp Wall) [block, below=6cm of k-neg, xshift=-5cm] {Sharp Wall\\ \includegraphics[width=1.5cm]{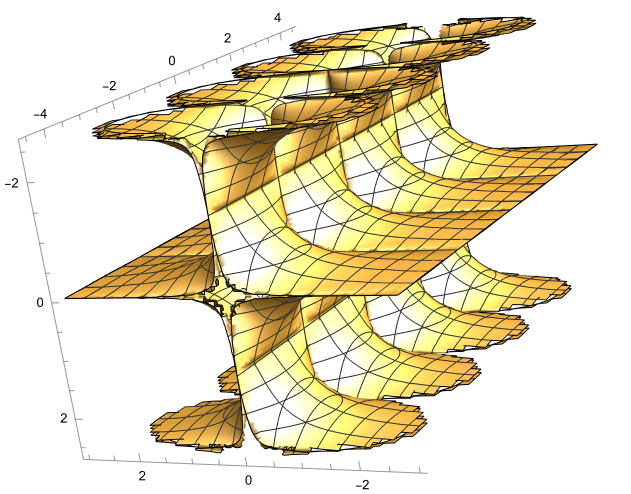}
	 				};
	 				
	 				\draw [arrow] (MGT) -- node[above, font=\small] {Trivial} (TMGT);
	 				\draw [arrow] (MGT) -- node[above, yshift=-0.6cm, align=left, font=\small] {Constant Minimal\\ Graph Surfaces} (Constant);
	 				\draw [arrow] (MGT) -- node[right, font=\small] {Non-trivial} (MSE);
	 				\draw [arrow] (MSE) -- node[above, font=\small] {$h=0$} (helicoids);
	 				\draw [arrow] (MSE) -- node[right, align=left, font=\footnotesize] {
	 					\(h\neq 0:\)\\Conversion to \\ an Equivalent \\ system
	 				}(Equivalent);
	 				\draw [arrow, draw=red] (helicoids.north) -- ++(0.5, 1.2) -- ++(3.2, 0) 
	 				node[left, midway, xshift=1.6cm, yshift=-0.8cm, align=center, font=\footnotesize, text=red] {
	 					NMG Transformation: \\ 
	 					Parameter Change
	 				} 
	 				|- (helicoids.east);
	 				\draw [arrow] (Equivalent) -- node[right, font=\small] {Weakening} (Weak);
	 				\draw [arrow] (Weak) -- node[above, font=\small] {} (k-neg);
	 				\draw [arrow] (Weak) -- node[above, font=\small] {} (k-zero);
	 				\draw [arrow] (Weak) -- node[above, font=\small] {} (k-pos);
	 				\draw [arrow] (k-pos) -- node[right, xshift=-0.1cm, align=left, font=\small] {
	 					$c_0\neq 0$ \\ \(\;\) solving \\ \(\;\;\;\;\) elliptic integral
	 				} (Pillars);
	 				\draw [arrow] (k-pos) -- node[right, font=\small] {
	 					\(c_0=0\) 
	 				} (Catenoid);
	 				\draw [arrow] (k-zero) -- node[above, font=\small] {} (Scherk);
	 				\draw [arrow, draw=red] (Catenoid.north east) -- ++(0.5, 1.2) -- ++(3.2, 0) 
	 				node[left, midway, xshift=1.6cm, yshift=-0.8cm, align=center, font=\footnotesize, text=red] {
	 					NMG Transformation: \\ 
	 					Parameter Change
	 				} 
	 				|- (Catenoid.east);
	 				\draw [arrow, draw=red] (Pillars.north) -- ++(0.5, 1.2) -- ++(3.2, 0) 
	 				node[left, midway, xshift=1.6cm, yshift=-0.8cm, align=center, font=\footnotesize, text=red] {
	 					NMG Transformation: \\ 
	 					Parameter Change
	 				} 
	 				|- (Pillars.east);
	 				\draw [arrow, draw=red] (Scherk.north) -- ++(0.7, 1.1) -- ++(3.2, 0) 
	 				node[left, midway, xshift=1.6cm, yshift=-0.8cm, align=center, font=\footnotesize, text=red] {
	 					NMG Transformation: \\ 
	 					Parameter Change
	 				} 
	 				|- (Scherk.east);
	 				\draw [arrow] (k-neg) -- node[right, xshift=-0.1cm, yshift=-0.6cm, align=left, font=\small] {$\frac{2|c_0|}{|k|}-C_1=0$ \\ \(\;\)solving \\\(\;\)elliptic integral} (Great Wall);
	 				\draw [arrow] (k-neg) -- node[right, xshift=-0.3cm, align=left,  yshift=0cm, font=\small] {$\frac{2|c_0|}{|k|}-C_1>0$ \\ \(\;\;\) solving \\\(\;\;\;\;\;\;\)elliptic integral} (Thick Wall);
	 				\draw [arrow] (k-neg) -- node[left, xshift=0.3cm, align=right,  yshift=0cm, font=\small] {$\frac{2|c_0|}{|k|}-C_1<0$ \\ solving\(\;\;\;\;\) \\elliptic integral \(\;\;\;\;\;\;\)} (Sharp Wall);
	 				\draw [dual_arrow] (Sharp Wall) -- node[above, align=center, xshift=1.3cm, font=\small, text=red] {
	 					NMG Transformation
	 				} (Thick Wall);
	 				\draw [dual_arrow] (Thick Wall) -- node[right, align=center, yshift=-0.3cm, font=\small, text=red] {
	 					NMG Transformation
	 				} (Great Wall);
	 				\draw [dual_arrow] (Sharp Wall) -- node[left, align=center, yshift=-0.3cm, font=\small, text=red] {
	 					NMG Transformation
	 				} (Great Wall);
	 				\draw [arrow, draw=red] (Sharp Wall.north) -- ++(-0.7, 1.1) -- ++(-3.2, 0) 
	 				node[left, midway, xshift=1.8cm, yshift=-0.8cm, align=center, font=\footnotesize, text=red] {
	 					NMG Transformation: \\ 
	 					Parameter Change
	 				} 
	 				|- (Sharp Wall.west);
	 				\draw [arrow, draw=red] (Thick Wall.north) -- ++(0.7, 1.1) -- ++(3.2, 0) 
	 				node[left, midway, xshift=1.6cm, yshift=-0.8cm, align=center, font=\footnotesize, text=red] {
	 					NMG Transformation: \\ 
	 					Parameter Change
	 				} 
	 				|- (Thick Wall.east);
	 			\end{tikzpicture}
	 	}}
	 	\caption{Flow chart illustrating the logical progression and structural organization of this paper.}
	 	\label{pic: flow chart}
	 \end{figure}	 
	 
	 In the final sections we provide a brief discussion of the singular‑point analysis of several minimal surfaces constructed using the techniques developed in this paper. A section on immediate applications of these techniques is also included. Two appendices have been added: one on analytic functions and their properties (Appendix: \ref{section: Appendix analytic function}), many of whose theorems have been used consistently throughout the paper; and a second appendix on elliptic integrals and Jacobi elliptic functions (Appendix: \ref{section: Elliptic funnctions}), an important class of functions whose identities and properties have also been employed consistently in the paper.
	
	\section{Trivial Minimal Graph Transformation}
    In this section we will find and explicitly characterise all the TMG Transformations. Before doing that lets give a useful alternative definition for Minimal Graph Transformation in the form of the following Proposition.
    
    \begin{proposition}
    	\label{thm: Alt. def. MGT}
        Let \(g: U \subseteq \mathbb{R} \longrightarrow \mathbb{R}\) be a smooth map and \(f: \Omega \longrightarrow \mathbb{R} \) be a Minimal Graph Surface such that \(f(\Omega) \subseteq U\), then
        \(g\) is the Minimal Graph Transformation of the Minimal Graph Surface \(f\) if and only if they obeys the equation:
        \begin{equation}
        	\label{eqn: Alt. def. MGT}
        	g\prime\prime(f)\cdot \|\nabla f\|^2= [{g\prime}^3-g\prime](f) \cdot \Delta f
        \end{equation}
    \end{proposition}
    \begin{proof}
        We know by definition that \(g: U\longrightarrow\mathbb{R}\) is a Minimal Graph Transformation and \(f: \Omega \longrightarrow \mathbb{R} \) is a Minimal Graph Surface if and only if \(h = g \circ f : \Omega \longrightarrow \mathbb{R}\) is a Minimal graph Surface.  So \(h\) must satisfy the Minimal Surface Equation:
        \[
        (1 + h_y^2) h_{xx} - 2 h_{xy} h_x h_y + (1 + h_x^2) h_{yy} = 0.
        \]  
        Now substituting the fact that \(h=g\circ f\), gives the following equation:
        \[
        (1 + g\prime^2 f_x^2)(g\prime\prime f_y^2 + g\prime f_{yy}) - 2 g\prime^2 f_x f_y (g\prime\prime f_x f_y + g\prime f_{xy}) + (1 + g\prime^2 f_y^2)(g\prime\prime f_x^2 + g\prime f_{xx}) = 0
        \]
        Expand and simplify this, then add \(g\prime^3\Delta f\) on both sides, simplify again using the fact that \(f\) obeys Minimal Graph Surface Equation, we get:
        \[
        g\prime\prime(f)\cdot\|\nabla f\|^2= (g\prime^3-g\prime)(f)\cdot\Delta f
        \]
        This proves the Proposition.  
    \end{proof}

    Now suppose that \(g:\mathbb{R}\longrightarrow\mathbb{R}\) is a TMG Transformation. By the Proposition \ref{thm: Alt. def. MGT}, this is equivalent to the statement that for every Minimal Graph Surface \(f\) the Equation \ref{eqn: Alt. def. MGT} holds. 
    
    If \(g\) is a function such that \(g\prime^3-g\prime \equiv 0\) then either \(g\prime \equiv 0\) or \(g\prime \equiv 1\) or \(g\prime \equiv -1\), in all three cases \(g\prime\prime \equiv 0 \), hence for all the Minimal Graph Surface \(f\) the Equation \ref{eqn: Alt. def. MGT} holds. So by Proposition \ref{thm: Alt. def. MGT} every \(g\) given by one of the following forms :
    \begin{align}
    	&g\prime \equiv 0  &\text{or}\,\,\, &g\equiv C \,\, &\text{for some constant}\, C\\
    	&g\prime \equiv 1  &\text{or}\,\,\, &g(t)=t+C \,\,  &\text{for some constant}\, C\\
    	&g\prime \equiv -1 &\text{or}\,\,\, &g(t)=-t+C \,\, &\text{for some constant}\, C
    \end{align}
	is a TMG Transformation.\\
	
	Infact it is also true that these are the only TMG Transformations. We will prove this by contradiction, suppose that \(g\) be a TMG Transformation such that \(g\prime^3-g\prime \neq 0\). Since \(g\) is continuous, there exist a domain \(\tilde U\subseteq \mathbb{R}\) such that \(g\prime^3-g\prime \not\equiv 0\) in that domain. Consider the Scherk surface given by \(\tilde f(x,y)=\ln\big(\tfrac{cos(x)}{cos(y)}\big)\), it is an unbounded surface and its range is \(\mathbb{R}\), hence contains \(\tilde U\), let \(\tilde \Omega\subseteq\mathbb{R}^2\) be such that \(\tilde f(\tilde \Omega)\subseteq\tilde U\). We can safely assume that in \(\tilde \Omega\), \(\|\nabla \tilde f\|^2 \not\equiv 0\) (for that shrink \(\tilde \Omega\) if necessary). These facts along with the fact that \(g\) is a TMG Transformation gives that \(\tfrac{g\prime\prime(\tilde f)}{[g\prime^3-g\prime](\tilde f)}=\tfrac{\Delta \tilde f}{\|\nabla \tilde f\|^2}\) in this domain \(\tilde \Omega\). This implies that \(\tfrac{\Delta \tilde f}{\|\nabla \tilde f\|^2}\) is a function of \(\tilde f\), call it \(h(\tilde f)=\tfrac{g\prime\prime(\tilde f)}{[g\prime^3-g\prime](\tilde f)}\). Hence for the Scherk surface \(\tilde f\) we have:  \(\tfrac{\Delta \tilde f}{\|\nabla \tilde f\|^2}=h(\tilde f)\). We differentiate this equation with respect to \(x\) and multiply \(f_y\) in both LHS and RHS, Similarly we also differentiate this equation with respect to \(y\) and multiply \(f_x\) in both LHS and RHS, we get the following:
	\begin{equation}
		\Big[\tfrac{\Delta \tilde f}{\|\nabla \tilde f\|^2}\Big]_x.\tilde f_y=h'(\tilde f) \tilde f_x \tilde f_y=\Big[\tfrac{\Delta \tilde f}{\|\nabla \tilde f\|^2}\Big]_y.\tilde f_x
	\end{equation}
	But we know that for the Scherk surface the equality \(\Big[\tfrac{\Delta \tilde f}{\|\nabla \tilde f\|^2}\Big]_x.\tilde f_y=\Big[\tfrac{\Delta \tilde f}{\|\nabla \tilde f\|^2}\Big]_y.\tilde f_x\) doesn't holds.\\
	Contradition! It is due to our assumption that a TMG Transformation \(g\) exist with \(g\prime^3-g\prime \neq 0\). Hence we have the following Theorem: 
	\begin{theorem}
		\label{thm: Trivial MGT}
		The only smooth maps \(g:\mathbb{R}\longrightarrow\mathbb{R}\) which are also Trivial Minimal Graph Transformations are precisely the maps of the form: \(g(t)\equiv C\), \(g(t)=t+C\) and \(g(t)=-t+C\) for some arbitrary constant \(C\)
	\end{theorem}

    \section{Constant Minimal Graph Surfaces}
    In this section we give some properties of Constant Minimal Graph Surfaces and give their Minimal Graph Transformations.
    
    First, suppose that, \(f: \Omega \longrightarrow \mathbb{R}\) be the Constant Minimal Graph Surface \(f\equiv C\). Then for every smooth map \(g: U \subseteq \mathbb{R} \longrightarrow \mathbb{R}\) with  \(f(\Omega)\subseteq U\), we have \(g \circ f: \Omega \longrightarrow \mathbb{R}\) is also a constant function given by \(g \circ f\equiv g(C)\). Hence we have the following Proposition.
	\begin{proposition}
		\label{thm: MGT of Const MGS}
	    Let \(f: \Omega \longrightarrow \mathbb{R}\) be the Constant Minimal Graph Surface, then every smooth map \(g: U \subseteq \mathbb{R} \longrightarrow \mathbb{R}\) with  \(f(\Omega)\subseteq U\) is a  Minimal Graph Transformations of the Constant Minimal Graph Surface \(f\) 	
	\end{proposition}

    We gave all the Minimal Graph Transformation of the Constant Minimal Graph Surface. A natural question to ask is: What about the Non-Constant Minimal Graph Surfaces with 'significant part' of it being Constant? Does there exist such a Minimal Surface? More formally, what we are asking is: Does there exist a Non-Constant Minimal Graph Surface \(f: \Omega \longrightarrow \mathbb{R}\) with a open subset \(V\subseteq\Omega\) such that \(f\big|_V\equiv C\) for some constant \(C\). If such Minimal Graph Surfaces exists, what are the Minimal Graph Transformations for such surfaces. The Answer is that such Minimal Graph Surfaces donot exist. We state and prove this fact in the following Proposition.

   \begin{proposition}
   	   \label{thm: MGS is non-constant everywhere}
	   Let \( f : \Omega \longrightarrow \mathbb{R} \) be a minimal graph surface, and let \( \Omega_0 \subseteq \Omega \) be a nonempty open subset such that \( f|_{\Omega_0} \equiv C \) for some constant \( C \in \mathbb{R} \). Then \( f \equiv C \) on the entire domain \( \Omega \).
   \end{proposition}

   \begin{proof}
	  Define the open subset \( \tilde{\Omega_0} \subseteq \Omega \) by
	  \[
	    \tilde{\Omega_0} = \bigcup \left\{ \Omega_{\alpha} \subseteq \Omega \;\middle|\; \Omega_0 \subseteq \Omega_{\alpha},\; \Omega_{\alpha} \text{ is open, and } f|_{\Omega_{\alpha}} \equiv C \right\}.
	  \]
	  By construction, \( \tilde{\Omega_0} \) is the largest open subset of \( \Omega \) on which \( f \) is identically equal to \( C \).
	
	  Consider the boundary \( \delta \tilde{\Omega_0} \) of \( \tilde{\Omega_0} \) relative to \( \Omega \). If \( \delta \tilde{\Omega_0} = \varnothing \), then \( \tilde{\Omega_0} = \Omega \), since \( \Omega \) is connected. Hence \( f \equiv C \) on \( \Omega \), and the result follows.
	
      Now assume, for the sake of contradiction, that \( \delta \tilde{\Omega_0} \neq \varnothing \). Then there exists a point \( (x_0, y_0) \in \delta \tilde{\Omega_0} \subseteq \Omega \). Since the surface \( (x, y, f(x, y)) \) is minimal, it admits a local isothermal parametrization
	  \[
	    (\phi_1(u,v), \phi_2(u,v), \phi_3(u,v)) : V_0 \longrightarrow \mathbb{R}^3
	  \]
	  in a neighborhood \( V_0 \) of the point  \( (x_0, y_0, f(x_0, y_0)) \), where each \( \phi_i \) is a harmonic function (see \cite{Os}, Lemma 4.2 and Lemma 4.4).

	  Because $(x_0, y_0)$ lies on the boundary of $\tilde{\Omega_0}$ and $f \equiv C$ on $\tilde{\Omega_0}$, the function $\phi_3$ attains the constant value $C$ on an open set inside $V_0$. Since $\phi_3$ is harmonic, the function $\Phi_3 = \phi_3 + i\tilde{\phi}_3$ is holomorphic in $V_0$, where $\tilde{\phi}_3$ is the harmonic conjugate of $\phi_3$. The holomorphic map $\Phi_3$ equals $C + i\tilde{\phi}_3$ on an open subset of $V_0$; by the identity theorem, this equality holds throughout $V_0$. Consequently, $\phi_3$ is constant throughout the neighborhood $V_0$.
	
	  Hence \( f \) is locally constant around \( (x_0, y_0) \), which contradicts the maximality of \( \tilde{\Omega_0} \). Therefore, our assumption that \( \delta \tilde{\Omega_0} \neq \varnothing \) must be false. It follows that \( \tilde{\Omega_0} = \Omega \), and consequently \( f \equiv C \) on \( \Omega \).
  \end{proof}

   We now pose a question stronger than the previous one: Does there exist a non-constant minimal graph surface \(f: \Omega \longrightarrow \mathbb{R}\) and a closed subset \(V \subseteq \Omega\) of nonzero measure such that \(f\big|_V \equiv C\) for some constant \(C\)? The answer is that such minimal graph surfaces do not exist. We state and prove this fact in the following proposition.
  
  \begin{proposition}
      \label{thm: MGS is non-constant almost everywhere}
      Let \( f : \Omega \longrightarrow \mathbb{R} \) be a minimal graph surface, and let \( V \subseteq \Omega \) be a nonempty closed subset of nonzero measure such that \( f|_V \equiv C \) for some constant \( C \in \mathbb{R} \). Then \( f \equiv C \) on the entire domain \( \Omega \).
  \end{proposition}
  
  \begin{proof}
  	 It is a well known fact that a minimal graph surface must be real analytic on its domain (see Appendix~\ref{section: Appendix analytic function}, Theorem~\ref{thm: MGS is real analytic}). Consequently, \(f - C\) is also a real analytic function. By the hypothesis of this proposition, this function vanishes on a set \(V\) of non‑zero measure. Therefore, by the fact that a real analytic function must have zero set of measure zero (see Appendix \ref{section: Appendix analytic function}, Theorem~\ref{thm: zero set of real analytic fns}), it follows that \(f - C \equiv 0\), or equivalently, \(f \equiv C\) on \(\Omega\).
  \end{proof}
  
  We now prove a proposition that is crucial for justifying Choice of Domain \ref{section: assumptions}, which will be introduced later in this section. Specifically, we show that for a nonconstant minimal graph surface \(f\), the function \(\|\nabla f\|^2\) is nonzero almost everywhere in its domain.
  
  \begin{proposition}
  	\label{thm: f_x^2+f_y^2 is nonzero almost everywhere}
  	Let \( f : \Omega \longrightarrow \mathbb{R} \) be a nonconstant minimal graph surface. Then \(\|\nabla f\|^2\) vanishes on a set of measure zero.
  \end{proposition}
  
  \begin{proof}
  	Since \(f\) is a real analytic function, it follows that \(\|\nabla f\|^2 = f_x^2 + f_y^2\) is also real analytic (see Appendix: \ref{section: Appendix analytic function}, Theorems \ref{thm: algebra of real analytic functions} and \ref{thm: calculus of real analytic functions}). As \(f\) is nonconstant, \(\|\nabla f\|^2\) is not identically zero on \(\Omega\). Therefore, by Theorem \ref{thm: zero set of real analytic fns}, its zero set must be of measure zero.
  \end{proof}

  The preceding theorems establish properties of minimal graph surfaces will be useful in making our subsequent choice of domain. We now present a corresponding result for the minimal graph transformation \(g\). Since \(f\) is real analytic on the connected domain \(\Omega\), its image is an interval \(I \subset \mathbb{R}\). Given a minimal graph transformation \(g : U \longrightarrow \mathbb{R}\) with \(I \subseteq U\), we address the following question: can \(g\prime^3 - g\prime\) vanish on any open subset of \(I\)? We demonstrate that this is not possible, we state and prove this fact in the following proposition.
  
\begin{proposition}
	\label{thm: NTMGT is not locally trivial}
	Let \(f : \Omega \longrightarrow \mathbb{R}\) be a minimal graph surface with range \(I\). Suppose \(g : U \longrightarrow \mathbb{R}\), where \(I \subseteq U\), is a NMG Transformation of \(f\); that is, \(g\prime^3 - g\prime\) is not identically zero on \(I\). Then there does not exist any open subinterval \(I_0 \subseteq I\) such that \(g\prime^3 - g\prime|_{I_0} \equiv 0\). (Note that \(g\prime^3 - g\prime\) can possibly vanish at few points in \(I\); the claim is that these zeros cannot form an open subinterval \(I_0 \subseteq I\).).
\end{proposition}

  \begin{proof}
  	Suppose, for contradiction, that such an open set \(I_0 \subseteq I\) exists. Then, by definition, \(g\) acts as a TMG Transformation on \(I_0\). Let \(\Omega_0 = f^{-1}(I_0)\). On \(\Omega_0\), the transformed surface \(\tilde{f} = g \circ f\) satisfies \(\tilde{f}|_{\Omega_0} = g(f|_{\Omega_0}) = \epsilon f|_{\Omega_0} + C\) for some constant \(C\) and \(\epsilon = \pm 1\).
  	
  	Both \(\epsilon f + C\) and \(\tilde{f}\) are minimal graph surfaces and are real analytic. Since they agree on the set \(\Omega_0\), which has positive measure, their difference \(\tilde{f} - (\epsilon f + C)\) is a real analytic function that vanishes on a set of positive measure. By Theorem~\ref{thm: zero set of real analytic fns}, it must be identically zero on its domain, implying \(\tilde{f} \equiv \epsilon f + C\) on \(\Omega\). Consequently, \(g\) is a TMG Transformation on the entire interval \(I\), meaning \(g\prime^3 - g\prime|_I \equiv 0\). This contradicts the initial hypothesis that \(g\) is nontrivial.
  \end{proof}

   \begin{remark}
   	   This author believes that a stronger result holds: if \(g\) is a NMG Transformation of a minimal graph surface \(f\), then the subset of the range \(I\) of \(f\) on which \(g\prime^3 - g\prime\) vanishes is of measure zero. A proof of this claim would be easily acheivable by establishing that \(f\) satisfies the following property: "for any subset \(I_0 \subseteq I\) of positive measure, its preimage \(f^{-1}(I_0)\) also has positive measure". Given that \(f\) is a real analytic function which is a very strong condition the author thinks that this property is true.
   \end{remark}

    Based on all the results established, we can now introduce the following choice of domains for our subsequent analysis.
   
    \subsection{Convenient Choice of Domain}\label{section: assumptions}
    Proposition~\ref{thm: f_x^2+f_y^2 is nonzero almost everywhere} establishes that for a non-constant minimal graph surface \( f : \Omega \longrightarrow \mathbb{R} \), the function \(\|\nabla f\|^2\) vanishes at most on a set of measure zero. Similarly, Proposition~\ref{thm: NTMGT is not locally trivial} states that if a minimal graph transformation \(g: U \longrightarrow \mathbb{R}\) is trivial on any open subset of the range of \(f\), then it is trivial globally. Consequently, for \(g\) to be nontrivial, \(g\prime^3 - g\prime\) cannot vanish on any open subset of \(f(\Omega)\). These results characterize the size and nature of the sets where \(\|\nabla f\|^2\) and \(g\prime^3 - g\prime\) vanish.
    
    Since all minimal graph surfaces are real analytic, it is sufficient to solve the problem on a smaller, well-behaved domain. Specifically, given a non-constant minimal graph surface \(f\) on a domain \(\Omega\) that admits a NMG Transformation \(g\), there exists a subdomain \(\Omega_0 \subseteq \Omega\) on which \(\|\nabla f\|^2\) and \((g\prime^3 - g\prime)(f)\) are non-vanishing. This presents two related problems:
    \begin{enumerate}
    	\item Find \(f\) and \(g\) on the original domains \(\Omega\) and \(I = f(\Omega)\).
    	\item Find the restrictions \(f|_{\Omega_0}\) and \(g|_{I_0}\) on the subdomains \(\Omega_0\) and \(I_0 = f(\Omega_0)\).
    \end{enumerate}
    Solving either problem is sufficient. A solution on \(\Omega\) directly restricts to a solution on \(\Omega_0\). Conversely, a solution on \(\Omega_0\) has a unique analytic continuation to \(\Omega\) by Theorem~\ref{thm: zero set of real analytic fns}, using this the minimal graph transformation \(g\) on \(I\) can be determined. Our strategy is to first find solutions on a suitable subdomain \(\Omega_0\) and then extend them globally via analytic continuation. So we see that the choice of domain is quite flexible and we can choose a domain of our convenience
    
    Therefore, our subsequent analysis is conducted under the assumption that \(\|\nabla f\|^2\) and \(g\prime^3 - g\prime\) are non-vanishing on \(\Omega\) and \(U\) (with \(f(\Omega) \subseteq U\)), respectively. The existence of such domains \(\Omega\) and \(U\) is justified by the following reasoning: Since \(f\) is a nonconstant minimal graph surface, there exists a point where \(\|\nabla f\|^2\) is nonzero. By the continuity of \(\|\nabla f\|^2\), there exists a region \(\Omega'\) around this point where \(\|\nabla f\|^2\) is never zero. Now consider the interval \(U' = f(\Omega')\). The same argument also works for the function \(g\prime^3(s) - g\prime(s)\), as \(g\) is a NMG Transformation, there exists a subinterval \(U \subseteq U'\) where the function \(g\prime^3(s) - g\prime(s)\) is never vanishing. Finally, take \(\Omega\) to be any connected open subset of \(f^{-1}(U) \cap \Omega'\). For this choice of \(\Omega\) and \(U\), the functions \(\|\nabla f\|^2\) and \(g\prime^3 - g\prime\) are non-vanishing on \(\Omega\) and \(U\), respectively, and \(f(\Omega) \subseteq U\).
    
    The subsequent steps will involve identifying the maximal domain where this condition holds, locating the singular points where either function vanishes, and finally extending the solution to the entire plane via analytic continuation, excluding these singular points.

  \section{Non-Trivial Minimal Graph Transformations of Non-Constant Minimal Graph Surfaces}
   We begin by supposing that \( f : \Omega \longrightarrow \mathbb{R} \) is a Minimal Graph Surface and \(g: U \subseteq \mathbb{R} \longrightarrow \mathbb{R}\) (with \(f(\Omega)\subseteq U\)) is a NMG Transformation of \(f\). We further assume that in the region \(\Omega\), the functions \(\|\nabla f\|^2\) and \(g\prime^3(s) - g\prime(s)\) are never zero. Then, the condition that \(g\) is a NMG Transformation of \(f\) is equivalent to the following pair of equations:
   \[
     \boxed{\frac{\Delta f}{\|\nabla f\|^2}\Bigg|_{(x,y)}=\frac{g\prime\prime}{g\prime^3 - g\prime}\Bigg|_{f(x,y)}} \quad \textbf{and} \quad \boxed{(1 + f_y^2) f_{xx} - 2 f_{xy} f_x f_y + (1 + f_x^2) f_{yy} = 0}.
   \]
    
   Let us denote \(h = \frac{g\prime\prime}{g\prime^3 - g\prime}\) and \(G = g\prime\). We are going to show that if a function \(h\) satisfies the condition \(\frac{\Delta f}{\|\nabla f\|^2} = h(f)\) for a minimal graph surface \(f\), then \(f\) admits a NMG Transformation \(g\). This transformation \(g\) can be explicitly found from the function \(h\) under the assumption that \(g\) and \(h\) satisfy \(\frac{g\prime\prime}{g\prime^3 - g\prime} \big|_f = \frac{\Delta f}{\|\nabla f\|^2} = h(f)\), via the following calculation:
     
   We are given the relation \(\frac{g\prime\prime}{g\prime^3 - g\prime} = \frac{G'}{G^3 - G} = h\), where \(G = g\prime\). A partial fraction decomposition gives \(-\frac{G'}{G} + \frac{G'}{2(G-1)} + \frac{G'}{2(G+1)} = h\), and integrating both sides yields  \(\ln\big[\frac{\sqrt{|G^2 - 1|}}{|G|}\big]dt = \int h dt+C_0\) for some constant \(C_0\). Exponentiating and squaring gives \(e^{2C_0} e^{2\int h\, dt} = \frac{|G^2 - 1|}{G^2}\), which leads to the two cases \(G^2 > 1\) and \(0 < G^2 < 1\). Since \(G^3 - G \neq 0\) by assumption, \(G^2\) cannot be 0 or 1. Furthermore, because \(G^2\) is continuous on the connected set \(U\), its image is connected; consequently, \(G^2\) must lie entirely in one of the two intervals \((0,1)\) or \((1,\infty)\), as a transition between them would require \(G^2\) to attain the value 1 due to Intermediate Value Theorem for Continous Functions, contradicting the assumption. So lets consider the following two cases:
   
   \textbf{Case 1:} Let \(G^2 > 1\), then \( |G^2 - 1| = G^2 - 1 \). Hence \( e^{2C_0}e^{2\int h\, dt} = 1 - \tfrac{1}{G^2}\), which immediately implies that \( G = \pm \frac{1}{\sqrt{1 -e^{2C_0} e^{2\int h\, dt}}}\). Since \(G = g\prime\), it follows that: \( g = \pm \int \frac{1}{\sqrt{1 - e^{2C_0}e^{2\int h\, dt}}}+C_1\) for some constant \(C_1\)

   \textbf{Case 2:} Similarly if we let \(0<G^2<1\), then \( |G^2 - 1| = 1-G^2\). Hence \( e^{2C_0}e^{2\int h\, dt} = \tfrac{1}{G^2}-1\), which immediately implies that \( G = \pm \frac{1}{\sqrt{1 +e^{2C_0} e^{2\int h\, dt}}}\). Since \(G = g\prime\), it follows that: \( g = \pm \int \frac{1}{\sqrt{1 + e^{2C_0}e^{2\int h\, dt}}}+C_1\) for some constant \(C_1\)\\
   \textbf{Combining both cases}, we write \(e^{2C_0}\) as \(C^2\) (since \(e^{2C_0} > 0\)) and \(C_1\) as \(D\). This leads us to the general expression for a minimal graph transformation \(g\) of the minimal graph \(f\), under the condition that \(\tfrac{\Delta f}{\|\nabla f\|^2} = h(f)\). The solution is given by the formula:
   \begin{equation}
   	\label{eqn: conv h to g}
   	\boxed{ g = \pm \int \frac{1}{\sqrt{1 \pm C^2e^{2\int h\, dt}}}\, dt + D},
   \end{equation}
   where the specific choice of constant \(C>0\) is determined by the requirement that the function \(\frac{1}{\sqrt{1 \pm C^2e^{2\int h\, dt}}}\) is real-valued and well-defined.  This calculation implies that if we are able to find a pair \((f,h)\) satisfying the system of partial differential equations:
   \begin{equation}
   	\label{eqn: Non-Trivial MGT: 1}
   	\boxed{(1 + f_y^2) f_{xx} - 2 f_{xy} f_x f_y + (1 + f_x^2) f_{yy} = 0}\quad \text{and} \quad  \boxed{\tfrac{\Delta f}{\|\nabla f\|^2} = h(f)},
   \end{equation}
   then we obtain a minimal graph surface \(f\) and a NMG Transformation given by Formula~\ref{eqn: conv h to g}. Therefore, the main aim of the rest of this paper is to exhaustively classify all possible pairs \((f,h)\) that satisfy Equation~\ref{eqn: Non-Trivial MGT: 1}. We call System~\ref{eqn: Non-Trivial MGT: 1} the \textbf{Nontrivial Minimal Graph Transformation Equations (NMG Transformation Equation)} and the function \(h\) in~\ref{eqn: Non-Trivial MGT: 1} the \textbf{First Characteristic Function} of the Nontrivial Minimal Graph Transformation Problem.
   
   In the next section, we will check for the existence of a minimal graph surface \(f\) for which \(h\equiv0\) is the first characteristic function that solves the NMG Transformation Equation given in \ref{eqn: Non-Trivial MGT: 1}.

   \section{Zero First Characteristic function} \label{section: h=0}

   The aim of this section is the following: Find and characterise all the Minimal Graph Surfaces \(f\) such that First Characteristic Function \(h\) of its NMG Transformation Problem is zero, under the choice of domain as given in \ref{section: assumptions}
   
   It will be very useful for our calculations to convert from the real variables \(x, y\) and the real partial derivatives \(\tfrac{\partial}{\partial x}, \tfrac{\partial}{\partial y}\) to the complex variables \(z, \bar{z}\) and the complex partial derivatives \(\tfrac{\partial}{\partial z}, \tfrac{\partial}{\partial \bar{z}}\). Let us therefore express the NMG Transformation Equation \ref{eqn: Non-Trivial MGT: 1} in terms of these complex variables and derivatives:
   
   \begin{equation}
   	\label{eqn: Complex Non-Trivial MGT: 2}
   	\boxed{f_{zz}f_{\bar{z}}^2 + f_{\bar{z}\bar{z}}f_{z}^2 - f_{z\bar{z}}(1 + 2f_z f_{\bar{z}}) = 0} \quad \text{and} \quad \boxed{\frac{f_{z\bar{z}}}{f_z f_{\bar{z}}} = h(f)}
   \end{equation}
   
   Here, the first equation is the complex form of the minimal surface equation, and the second is the complex form of \(\tfrac{\Delta f}{\|\nabla f\|^2} = h(f)\). Now, since \(h \equiv 0\), the second equation implies \(f_{z\bar{z}} = 0\); that is, \(f\) is harmonic. Substituting this into the complex minimal surface equation simplifies it to \(f_{zz} f_{\bar{z}}^2 + f_{\bar{z}\bar{z}} f_{z}^2 = 0\). Hence, the system we need to solve is:
   
   \begin{equation}
   	\label{eqn: Complex Non-Trivial MGT: h=0}
   	\boxed{f_{zz} f_{\bar{z}}^2 + f_{\bar{z}\bar{z}} f_{z}^2 = 0} \quad \text{and} \quad \boxed{f_{z\bar{z}} = 0}
   \end{equation}

   By the choice of domain as in \ref{section: assumptions}, we have that \(\|\nabla f\|^2 \not\equiv 0\), so \(f_z\) and \(f_{\bar{z}}\) are never zero in the domain of \(f\). Consequently, the first equation of \ref{eqn: Complex Non-Trivial MGT: h=0} becomes \(\frac{f_{zz}}{f_z^2} = -\frac{f_{\bar{z}\bar{z}}}{f_{\bar{z}}^2}\), which can be rewritten as \(\big(-\frac{1}{f_z}\big)_z = \big( \frac{1}{f_{\bar{z}}} \big)_{\bar{z}}\). Since \(f_{z\bar{z}} = 0\), it follows that \(f_z\) is a function of \(z\) alone and \(f_{\bar{z}}\) is a function of \(\bar{z}\) alone. This implies that \(\big( -\frac{1}{f_z} \big)_z\) is a function only of \(z\), while \(\big( \frac{1}{f_{\bar{z}}} \big)_{\bar{z}}\) is a function only of \(\bar{z}\). Their equality forces them to be equal to a common complex constant \(C_0 \in \mathbb{C}\), i.e.,
   \[
   \left( -\frac{1}{f_z} \right)_z = C_0 = \left( \frac{1}{f_{\bar{z}}} \right)_{\bar{z}}.
   \]
   This leads us to two cases:
   
   \textbf{Case 1: } \(C_0 = 0\). This implies that \(\frac{1}{f_z}\) and \(\frac{1}{f_{\bar{z}}}\) are constant, so \(f_z = C_1\) and \(f_{\bar{z}} = C_2\) for some constants \(C_1, C_2 \in \mathbb{C}\). Therefore, \(f(z, \bar{z}) = C_1 z + C_2 \bar{z} + C_3\). The condition that \(f\) is real-valued implies \(f(0,0) = C_3 \in \mathbb{R}\) and, from \(\bar{f} = f\), that \(C_1 z + C_2 \bar{z} + C_3 = \bar{C_1} \bar{z} + \bar{C_2} z + C_3\). Comparing coefficients yields \(C_2 = \bar{C_1}\). Hence, \(f(z, \bar{z}) = C_1 z + \overline{C_1 z} + C_3\). Letting \(C_1 = \tfrac{1}{2}(\alpha - i\beta)\) and \(C_3 = \gamma\), we obtain:
   
   \begin{equation}
   	\label{eqn: case h=0 Plane}
   	\boxed{f(x,y) = \alpha x + \beta y + \gamma} \quad \text{for arbitrary real constants } \alpha, \beta, \gamma.
   \end{equation}
   
   This represents a plane. Thus, a \textbf{plane} is a minimal graph surface that admits a NMG Transformation. Before finding its explicit transformation, let us consider the second case.
   
   \textbf{Case 2: } \(C_0 \neq 0\). We have \(\big( -\frac{1}{f_z} \big)_z = C_0 = \big( \frac{1}{f_{\bar{z}}} \big)_{\bar{z}}\). Integrating, and using the fact that \(f_z\) is a function of \(z\) alone and \(f_{\bar{z}}\) is a function of \(\bar z\) alone, we obtain \(-\frac{1}{f_z} = C_0 z + D_1\) and \(\frac{1}{f_{\bar{z}}} = C_0 \bar{z} + D_2\) for some constants \(C_0, D_1, D_2 \in \mathbb{C}\). Thus, \(f_z = \frac{-1}{C_0 z + D_1}\) and \(f_{\bar{z}} = \frac{1}{C_0 \bar{z} + D_2}\). Since \(f\) is real-valued, we have the relation \(\overline{f_z} = f_{\bar{z}}\), which implies  \(\overline{\big(\frac{-1}{C_0z+D_1}\big)}=\frac{-1}{\overline{C_0}\bar{z}+\overline{D_1}} =\frac{1}{C_0\bar{z}+D_2}\) Consequently, \(-\overline{C_0} \bar{z} - \overline{D_1} = C_0 \bar{z} + D_2\). Comparing coefficients gives \(-\overline{D_1} = D_2\) and \(-\overline{C_0} = C_0\). The latter implies that \(C_0\) is purely imaginary, so we write \(C_0 = i\alpha\) for some real \(\alpha \neq 0\). Substituting this yields \(f_z = \frac{-1}{i \alpha z + D_1}\) and \(f_{\bar{z}} = \frac{1}{i \alpha \bar{z} - \overline{D_1}}\). Integrating these expressions, we find \(f = -\frac{1}{i\alpha} \ln \left( \frac{i\alpha z + D_1}{i\alpha} \right) + A(\bar{z})\) and \(f = \frac{1}{i\alpha} \ln \left( \frac{i\alpha \bar{z} - \overline{D_1}}{i\alpha} \right) + B(z)\), where \(A\) and \(B\) are functions of \(\bar{z}\) and \(z\), respectively. Comparing both expressions for \(f\), we deduce that 
   \[
    f(z, \bar{z}) = \frac{1}{i\alpha} \ln \left( \frac{i\alpha \bar{z} - \overline{D_1}}{i\alpha z + D_1} \right) + C_1
   \]
   for some constant \(C_1 \in \mathbb{C}\). Now, let us rewrite the constants. Set \(\frac{D_1}{i\alpha} = \beta + i\gamma\) and \(C_1 = \delta \in \mathbb{C}\). Note that \(-\frac{\overline{D_1}}{i\alpha} = \overline{\left( \frac{D_1}{i\alpha} \right)} = \beta - i\gamma\). Substituting \(z = x + iy\), we obtain
   \[
     f(x,y) = \frac{1}{i\alpha} \ln \left( \frac{x + \beta - i(y + \gamma)}{x + \beta + i(y + \gamma)} \right) + \delta = \frac{-2}{2i\alpha} \ln \left( \frac{1 + i\frac{y+\gamma}{x+\beta}}{1 - i\frac{y+\gamma}{x+\beta}} \right) + \delta
   \]
   Referring to \cite{JD}, Page 142, Section 2.7.2.2, Equation 3, we get that 
   
   \begin{equation}
   	\label{eqn: case h=0 Helicoid}
   	\boxed{f(x,y) = \frac{-2}{\alpha} \tan^{-1} \left( \frac{y + \gamma}{x + \beta} \right) + \delta}
   \end{equation}

   for real constants \(\alpha \neq 0, \beta, \gamma, \delta\). Here, \(\delta\) must be real, as all other terms are real. This function represents a helicoid. Thus, a \textbf{helicoid} is also a minimal graph surface that admits a NMG Transformation.
   
   We conclude that planes and helicoids are the only minimal graph surfaces admitting a NMG Transformation with a first characteristic function \(h \equiv 0\). We now proceed to find the explicit form of these transformations.

   \textbf{Non-Trivial Minimal Graph Transformations:} We know that \(\frac{g\prime\prime}{g\prime^3 - g\prime} = h\) and that \(h=0\). This implies that \(g\prime\prime=0\), which has the solution \(g(t)=at+b\). Hence, the transformations are affine maps (scaling and translation in the \(z\)-direction). For the transformation to be nontrivial, the scaling constant \(a\) must not be \(1\), \(-1\), or \(0\).

   \textbf{Transformed Minimal Graph Surface:} The transformed minimal surfaces are \(g \circ f\), where \(f\) is a minimal graph surface admitting the transformation \(g\). They are given by:
   \[
     a\alpha x + a\beta y + a\gamma + b \quad \text{for arbitrary real constants } a, b, \alpha, \beta, \gamma,
   \]
   which is again a plane when \(f\) is a plane. In the case where \(f\) is a helicoid, the transformed surface is:
   \[
    \frac{-2a}{\alpha}\tan^{-1}\left(\frac{y+\gamma}{x+\beta}\right) + a\delta + b \quad \text{for real constants } a, b, \alpha, \beta, \gamma, \delta,
  \]
  which is again a helicoid.

  \textbf{Maximal Domain and Singular Points of the Solution} We know that the NMG Transformations here are precisely \(g(t)=at+b\), where \(a \notin \{0, 1, -1\}\). For such \(g\), the expression \(g\prime^3 - g\prime = a^3 - a\) is defined and nonzero on all of \(\mathbb{R}\). Now, consider the gradient condition. For the plane given by Equation~\ref{eqn: case h=0 Plane}, we have \(\|\nabla f\|^2 = \alpha^2 + \beta^2\). This is zero if and only if \(\alpha = \beta = 0\), which would make \(f = \gamma\) a constant function (a case already dealt with). Therefore, for a non-constant plane, \(\|\nabla f\|^2\) is a nonzero constant and well-defined everywhere in \(\mathbb{R}^2\). For the helicoid given by \ref{eqn: case h=0 Helicoid}, we find that \(f_z = \frac{-1}{i\alpha(z + \beta + i\gamma)}\) is never zero, but it is undefined singular at the point \(z = -\beta - i\gamma\), i.e., at \((x, y) = (-\beta, -\gamma)\). Consequently, \(\|\nabla f\|^2\) is well-defined on the maximal domain \(\mathbb{R}^2 \setminus \{(-\beta, -\gamma)\}\). And note that the point \((-\beta, -\gamma)\) is where the axis of the helicoid lies.
  
  \section{Non-Zero First Characteristic function} 
  
  In the last section, we showed that helicoids and planes are precisely the minimal graph surfaces that admit a NMG Transformation when the first characteristic function is identically zero. In this and the next section, we will develop a method for solving the NMG Transformation problem when the first characteristic function is nonzero.
  
  A natural question arises: what happens if there exists some open set \(U_0 \subseteq U\) and \(\Omega_0 \subseteq \Omega\) (where \(U\) is the domain of \(h\) and \(\Omega\) is the domain of \(f\)) such that \(f(\Omega_0) \subseteq U_0\) and the first characteristic function \(h\) is zero on \(U_0\), while being nonzero on the rest of \(U\)? A straightforward answer is that, on the subdomain \(\Omega_0\), we would apply the analysis from Section~\ref{section: h=0}. However, a more interesting question is: do such domains actually exist within the domain of a minimal surface? The answer is no; such minimal graph surfaces do not exist. We will try to prove this more formally in the following proposition:
  
  \begin{proposition}
  	\label{thm: h is non zero every where}
  	Let \(f:\Omega\longrightarrow\mathbb{R}\) be a minimal graph surface. Suppose there exist a subdomain \(\Omega_0\subseteq\Omega\) such that \(\Delta f\) is zero on \(\Omega_0\). Then \(\Delta f\equiv 0\) on whole of \(\Omega\)
  \end{proposition}
  
  \begin{proof}
  	The function \(f\) is real analytic, and consequently, so is \(\Delta f\). If \(\Delta f\) vanishes on a set of positive measure, then by Theorem~\ref{thm: zero set of real analytic fns}, it must be identically zero on \(\Omega\).
  \end{proof}
  
  This result implies that the range \(I\) of \(f\) cannot contain an open subset \(U_0\) on which \(h\) vanishes. Suppose, for contradiction, that such a \(U_0\) exists. Then its preimage \(\Omega_0 = f^{-1}(U_0)\) is an open subset of \(\Omega\) on which \(\Delta f\) vanishes. By the preceding theorem, this would force \(\Delta f \equiv 0\) on all of \(\Omega\).
  
  \subsection{Additional Conditions on the Choice of Domain \ref{section: assumptions}}\label{section: assumption 2}
  
  In Section~\ref{section: h=0}, we explicitly characterized all minimal graph surfaces for which the first characteristic function \(h\) is identically zero. In this section, we operate under the assumption that \(h\) is not the zero function. Proposition~\ref{thm: h is non zero every where} provides insight into the size of the set where \(h\) may vanish. Accordingly, we now restrict our domain \(\Omega\) to a sufficiently small, simply connected region such that \(h(f)\) is non-vanishing on \(\Omega\).
  
  Note that the domain restrictions introduced here, as well as in \ref{section: assumptions}, do not affect the global nature of our analysis. Since solutions to the minimal surface equation are real analytic, any solution defined on a small open set \(\Omega\) possesses a unique analytic continuation to a maximal domain due to Theorem \ref{thm: zero set of real analytic fns}. Therefore, any solution found on a restricted domain inherently defines a unique minimal graph surface on a potentially larger region. Henceforth, we will proceed under this assumption.
  
   \section{Strategy for solving the system \ref{eqn: Complex Non-Trivial MGT: 2}}
   In this section we will present a strategy for solving the NMG Transformation Problem. As a first step we will reduce the complexity of the system by converting it into a simpler one.

   \subsection{Step 1: Converting the system \ref{eqn: Complex Non-Trivial MGT: 2} to a simpler one}
   \label{section: Step 1}
    We begin with the minimal surface equation in complex coordinates: \(f_{zz}f_{\bar{z}}^2 + f_{\bar{z}\bar{z}}f_z^2 - f_{z\bar{z}}(1 + 2f_z f_{\bar{z}}) = 0\), together with the relation \(\frac{f_{z\bar{z}}}{f_z f_{\bar{z}}} = h(f)\). The latter equation can be rewritten as \( \frac{f_{z\bar{z}}}{f_z} = f_{\bar{z}} \cdot h(f)\). 

   Let \(H_1 = \int h\,d\tau\). Then the equation becomes \(\partial_{\bar{z}}\ln(f_z) = \partial_{\bar{z}}(H_1 \circ f)\), which implies \(\partial_{\bar{z}}\left(\ln(f_z) - H_1 \circ f\right) = 0\). Consequently, \(\ln(f_z) - H_1 \circ f = a(z)\) for some function \(a(z)\) depending only on \(z\). This yields \(f_z = e^{a(z) + H_1 \circ f}\), or equivalently, \(e^{- H_1 \circ f}f_z=e^{a(z)}\). 

   Now define \(H = \int e^{-H_1} \, d\tau\). Then \(\partial_z(H \circ f) = e^{a(z)}\). Introducing \(u = H \circ f\), we see that \(u_z = e^{a(z)}\), from which it follows that \(u\) is harmonic:
   \[
     \Delta u=u_{z\bar{z}} = 0 \quad \text{or equivalently} \quad u(z,\bar{z}) = A(z) + \overline{A(z)},
   \]
   for some holomorphic function \(A(z)\), since every harmonic function is the real part of some holomorphic function. We now present some important remarks.

   \begin{remark}
	  \begin{enumerate}
		 \item The function \(H\) defined in the calculation above is a diffeomorphism from the domain of \(h\) to the range of \(H\). This follows because \(H\) is a smooth function which is the integral of the strictly positive function \(e^{-H_1}\), and we know that the integral of a strictly positive function is strictly increasing and hence bijective. Moreover, its derivative \(H' = e^{-H_1} > 0\) is never zero in its domain, so by the inverse function theorem there exists a smooth inverse \(K\) of \(H\).
		
		 \item Through the above calculation, we have converted the equation \(\frac{f_{z\bar{z}}}{f_z f_{\bar{z}}} = h(f)\) into \(\Delta u = 0\), which is a much simpler equation to handle.
		 
		 \item Multiple diffeomorphisms \(H\) can correspond to the same first characteristic function \(h\), due to the arbitrary integration constant in the definition of \(H\).
		
	   	 \item Note that \(u_z\) and \(u_{\bar{z}}\) are never zero. This is because \(u_z = H'(f)f_z\), and since \(H\) is a diffeomorphism, \(H'\) is never zero. Furthermore, \(f_z\) is never zero since we assumed in \ref{section: assumptions} that \(\|\nabla f\|^2 = 4f_z f_{\bar{z}}\) is never zero. Thus \(u_z\) is never zero. Similarly, \(u_{\bar{z}}\) is never zero in the domain. Therefore, in future calculations we may safely place these terms in denominators when necessary.
	 \end{enumerate}
 \end{remark}

  We now will rewrite the Minimal Surface Equation in terms of this \(u\). For that we use the relation \(K(u)=f\) and we substitute this in the Minimal Surface Equation. We use the following substitution:
  \begin{align}
  	\label{eqn: partial derivatives of f}
  	&\boxed{f_z=K'(u)u_z} \quad\quad\quad \boxed{f_{\bar{z}}=K'(u) u_{\bar{z}}} \quad\quad\quad \boxed{f_{zz}=K''(u) u_z^2+K'(u)u_{zz}}\\
  	\label{eqn: second partial derivatives of f}
  	&\boxed{f_{\bar{z}\bar{z}}=K''(u)u_{\bar{z}}^2+K'(u)u_{\bar{z}\bar{z}}} \quad\quad \boxed{f_{z\bar{z}}=K''(u)u_{\bar{z}}u_z+K'(u)u_{z\bar{z}}=K''(u)u_{\bar{z}}u_z}    
  \end{align}
  
  Now since we have that \(f_{z\bar{z}}=f_z\cdot f_{\bar{z}}\cdot h(f)\), we get that: $\boxed{f_{z\bar{z}}={K'}^2(u) h(K(u)) u_z u_{\bar{z}}}$\\

  We now substitute all the boxed facts given above in the Minimal Surface Equation in complex coordinate as given in the equation \ref{eqn: Complex Non-Trivial MGT: 2} and then simplifying by factoring out a \({K'}^2(u)\) from all terms to get: \(K'(u) u_z^2 u_{\bar{z}\bar{z}}+K'(u) u_{\bar{z}}^2 u_{zz}-(h\circ K)(u)u_z u_{\bar{z}}=0\).
  
  Again simplifying using the fact that \(K'(u)\) is never zero (as \(K\) is a diffeomorphism), we have:
  \[
  \frac{u_z^2\cdot u_{\bar{z}\bar{z}}+ u_{\bar{z}}^2\cdot u_{zz}}{u_z \cdot u_{\bar{z}}}=\frac{h(K(u))}{K'(u)} \quad\quad\text{and}\quad\quad u_{z\bar{z}}=0
  \]
  
  Now call the function \(\frac{h(K(s))}{K'(s)}=j(s)\). We have:
  \begin{equation}
  	\label{eqn: Modified Complex Non-Trivial MGT: 2}
  	\boxed{\frac{u_z^2\cdot u_{\bar{z}\bar{z}}+ u_{\bar{z}}^2\cdot u_{zz}}{u_z \cdot u_{\bar{z}}}=j(u)} \quad\quad\text{and}\quad\quad \boxed{u_{z\bar{z}}=0}    
  \end{equation}
  
  We call the system given above as the \textbf{Modified Non-Trivial Minimal Graph Transformation Problem (Modified NMG Transformation Problem )} and the function \(j\) is called the \textbf{Second Characteristic Function of the Non-Trivial Minimal Graph Transformation Problem}.\\
  
Before going further, let us understand the relations between the various functions we have discussed above in the following remarks.

\begin{remark}\label{remark: rln bw 1st and 2nd Char Fns}
	\begin{enumerate}
		\item \(h(s) = \frac{-H''(s)}{H'(s)}\): We know that \(H(s) = \int e^{-H_1}\), so \(H' = e^{-H_1}\) and \(H'' = -H_1' e^{-H_1} = -h e^{-H_1}\). This implies that \(h(s) = \frac{-H''(s)}{H'(s)}\), as required.
				
		\item \(j(H(s)) = -H''(s)\): We have \(j(H(s)) = \frac{h(K(H(s)))}{K'(H(s))}\). Noting that \(K\) is the inverse of \(H\) and by the chain rule we have \(\frac{1}{K'(H(s))} = H'(s)\), we get \(j(H(s)) = h(s) H'(s)\). Substituting the previous relation for \(h(s)\), we obtain \(j(H(s)) = -H''(s)\), as required.
		
		\item We have that \(u = H(f)\) and the domain of \(f\) is \(\Omega\), hence the domain of \(u\) is also \(\Omega\). Similarly, \(\frac{h(K(s))}{K'(s)} = j(s)\), so the domain of \(j\) is the domain of \(K\), which is the interval \(J = H(I)\) where \(I = \operatorname{Range}(f)\).
		
		\item The function \(j\) never vanishes in its domain, and its values have the same sign as the function \(h\) in its domain: This follows from the definition \(j(s) = \frac{h(K(s))}{K'(s)}\) and the fact that the derivative \(K'(s) = \frac{1}{H'(K(s))}\) is strictly positive, since \(H'(s) = e^{-H_1}\) is strictly positive in its domain. Hence \(j\) has to be entirely positive or entirely negative in its domain.
		
		\item Given a diffeomorphism \(H\), there exists a unique second characteristic function \(j\). This follows from the relation \(j(H(s)) = -H''(s)\), which implies \(j(t) = j(H(K(t))) = -H''(K(t))\). However, as noted in the previous remark, multiple diffeomorphisms \(H\) can correspond to the same first characteristic function \(h\). Therefore, given a first characteristic function \(h\), the corresponding second characteristic function \(j\) is not uniquely determined.
	\end{enumerate} 
\end{remark}

\subsection{Step 2: Equivalence of System \ref{eqn: Complex Non-Trivial MGT: 2} and \ref{eqn: Modified Complex Non-Trivial MGT: 2}}
\label{section: Step 2}

In this section, we will address the equivalence of systems \ref{eqn: Complex Non-Trivial MGT: 2} and \ref{eqn: Modified Complex Non-Trivial MGT: 2}. Specifically, we know from the calculations in Section \ref{section: Step 1} that for every solution \((f,h)\) of the NMG Transformation Problem \ref{eqn: Complex Non-Trivial MGT: 2}—where the first characteristic function \(h\) is non-vanishing on the connected interval \(I = \operatorname{Range}(f)\) (i.e., either entirely positive or entirely negative on its domain \(I\))—there exists a solution \((u,j)\) of the Modified NMG Transformation Problem. This solution is explicitly given by \(u = H \circ f\), where \(H(t) = \int e^{-\int h}\), and \(j(s) = \frac{h(K(s))}{K'(s)}\), where \(K = H^{-1}\). The second characteristic function \(j\) is also non-vanishing on its domain \(J\) (i.e., either entirely positive or entirely negative), and it have the same sign as the first characteristic function \(h\).

To establish equivalence, we must consider the converse: suppose \((u,j)\) is a solution of the Modified NMG Transformation Problem \ref{eqn: Modified Complex Non-Trivial MGT: 2}, with the second characteristic function \(j\) being entirely positive or entirely negative on its domain. Does there exist a pair \((f,h)\) that solves the NMG Transformation Problem \ref{eqn: Complex Non-Trivial MGT: 2}, such that \(u_0 = H \circ f\) (where \(H(t) = \int e^{-\int h}\)) and \(j_0(s) = \frac{h(K(s))}{K'(s)}\) (where \(K = H^{-1}\)), with \(u_0\) and \(j_0\) being the restrictions of \(u\) and \(j\) to a possibly smaller domain and the first characteristic function \(h\) being either entirely positive or entirely negative, with the same sign as the second characteristic function \(j\) on its domain?

Thankfully, the answer is yes. We prove this fact in the following two propositions.
  
\begin{proposition}
	Let \(u:\Omega\longrightarrow\mathbb{R}\) and \(j:J\longrightarrow\mathbb{R}\) be two smooth function respectively defined on a domain \(\Omega\) and interval \(J\) with \(u(\Omega)=J\). Now suppose  \((u,j)\) be a solution of the Modified Non-Trivial Minimal Graph Transformation Problem \ref{eqn: Modified Complex Non-Trivial MGT: 2}, with the second characteristic function \(j\) and the functions \(u_z, u_{\bar z}\) being entirely positive or entirely negative on their respective domains. Suppose there exists a function \(h: I\longrightarrow\mathbb{R}\) which is entirely positive or entirely negative and has the same sign as \(j\) on its domain, satisfying the condition:
	\[
	j(s) = \frac{h(K(s))}{K'(s)} \quad \text{where} \quad K = H^{-1} \quad \text{and} \quad H = \int e^{-\int h} \quad \text{and} \quad H(I)=J.
	\]
	Then for \(f:\Omega\longrightarrow\mathbb{R}\) defined by \(f = K(u)\), the pair \((f,h)\) satisfies the Non-Trivial Minimal Graph Transformation Problem \ref{eqn: Complex Non-Trivial MGT: 2}, and \(f_z,f_{\bar z}\) are nonvanishing in its domain.
\end{proposition}

\begin{proof}
	The proof is basically the reversing of the derivation that led to the Modified NMG Transformation Problem \ref{eqn: Modified Complex Non-Trivial MGT: 2}. Given that \(u\) and \(j\) satisfy system \ref{eqn: Modified Complex Non-Trivial MGT: 2}, and that there exists an \(h\) such that \(j(s) = \frac{h(K(s))}{K'(s)}\) where \(K = H^{-1}\) and \(H = \int e^{-\int h}\), we substitute and simplify to obtain:
	\[
	K'(u) u_z^2 u_{\bar{z}\bar{z}} + K'(u) u_{\bar{z}}^2 u_{zz} - (h \circ K)(u) u_z u_{\bar{z}} = 0 \quad \text{and} \quad u_{z\bar{z}} = 0.
	\]
	
	Adding and subtracting \(2 {K'}^2(u) (h \circ K)(u) u_z^2 u_{\bar{z}}^2\) and multiplying through by \({K'}^2(u)\), we get:
	\begin{align*}
		&{K'}^2(u) u_z^2 \left[{K'}^2(u) (h \circ K)(u) u_{\bar{z}}^2 + K'(u) u_{\bar{z}\bar{z}}\right] \\
		&\quad\quad + {K'}^2(u) u_{\bar{z}}^2 \left[{K'}^2(u) (h \circ K)(u) u_z^2 + K'(u) u_{zz}\right] \\
		&\quad\quad\quad\quad\quad- {K'}^2(u) (h \circ K)(u) u_z u_{\bar{z}} \left[1 + 2 {K'}^2(u) u_z u_{\bar{z}}\right] = 0.
	\end{align*}
	
	Now assume temporarily that \(\frac{\Delta f}{\|\nabla f\|^2} = h(f)\). (This will be proved later.) With \(f = K(u)\), the derivatives are given by \ref{eqn: partial derivatives of f} and \ref{eqn: second partial derivatives of f}. From \(f_{z\bar{z}} = f_z f_{\bar{z}} h(f)\) and substituting \(f_z\) and \(f_{\bar{z}}\), we have \(f_{z\bar{z}} = {K'}^2(u) h(K(u)) u_z u_{\bar{z}}\). Substituting all these into the above equation yields the Minimal Surface Equation:
	\[
	f_{zz} f_{\bar{z}}^2 + f_{\bar{z}\bar{z}} f_z^2 - f_{z\bar{z}}(1 + 2 f_z f_{\bar{z}}) = 0,
	\]
	confirming that \(f\) is a minimal surface.
	
	It remains to prove that \(\frac{\Delta f}{\|\nabla f\|^2} = h(f)\). Since \(u_{z\bar{z}} = 0\) and \(u = H(f)\), we compute: \(u_z = H'(f) f_z\), \(u_{\bar{z}} = H'(f) f_{\bar{z}}\), and \(u_{z\bar{z}} = H''(f) f_z f_{\bar{z}} + H'(f) f_{z\bar{z}}\). Using \(u_{z\bar{z}} = 0\), we obtain \(H'(f) f_{z\bar{z}} + H''(f) f_z f_{\bar{z}} = 0\). Note from equation \ref{eqn: partial derivatives of f} that \(f_z, f_{\bar{z}} \neq 0\) since \(u_z, u_{\bar{z}} \neq 0\) in their domain. Also, \(H'(f)\) is never zero by definition. Dividing the equation gives: \( \frac{f_{z\bar{z}}}{f_z f_{\bar{z}}} = \frac{-H''(f)}{H'(f)}\).	Substituting \(H''\) and \(H'\) from the relation \(H = \int e^{-\int h}\), we finally obtain 
	\[
	  \frac{\Delta f}{\|\nabla f\|^2} = \frac{f_{z\bar{z}}}{f_z f_{\bar{z}}} = h(f)
	\]
\end{proof}

 What we have proved in the last proposition is that, given a solution to the Modified NMG Transformation Problem and a function \(h\) that satisfies a specific relation with the second characteristic function \(j\)—a relation that must hold if \(h\) were indeed the first characteristic function corresponding to the given  \(j\)—then there exists a solution to the original NMG Transformation Problem for which this \(h\) serves as the first characteristic function. 
 
 Next we prove that given a second characteristic function \(j\) which is never vanishing in its domain, then there exist an \(h\) satisfying the condition described in the last proposition.
 
 So suppose that we are given a function \(j: J \longrightarrow \mathbb{R}\) that is never vanishing on its domain \(J\). If there exists a function \(h: I \longrightarrow \mathbb{R}\) such that \(j(s) = \frac{h(K(s))}{K'(s)}\), where \(K = H^{-1}\) and \(H = \int e^{-\int h}\), then the functions \(j\) and \(H\) must satisfy the relations \(j(H(s)) = -H''(s)\) and \(h(s) = \frac{-H''(s)}{H'(s)}\).\\
 
 Let us now consider the equation \(j(H(s)) = -H''(s)\). Define \(j_1 = -2\int j+C_1\). Then we obtain \(j_1(H(s))= {H'}^2(s)\), which implies \(\sqrt{j_1(H(s))} = H'(s)\). Since \(H'(s) > 0\), we take the positive square root. Rearranging gives \(1 = \frac{H'(s)}{\sqrt{j_1(H(s))}}\). Now define \(j_2 = \int \frac{1}{\sqrt{j_1 + C_1}}\). Note that the integrand is positive, so \(j_2\) is strictly increasing and hence diffeomorphism. Integrating the previous equation yields \(s + C_2 = j_2(H(s))\). Since \(j_2\) is a diffeomorphism, we can invert it to obtain \(H(s) = j_3(s + C_2)\), where \(j_3 = j_2^{-1}\) is also a diffeomorphism.\\
 
Note that the function \(j_2\) is defined, by its definition, on the subinterval \(J_0 \subseteq J\) where \(\sqrt{j_1}\) is real. This subinterval depends on the choice of the integration constant \(C_1\), which should be chosen so that \(j_2\) is defined at least somewhere. For any \(s_0 \in J\), by an appropriate choice of \(C_1\), there exists a subinterval \(J_0\) of \(J\) where \(j_2\) is well-defined. Now, the domain of \(j_3\) is \(I_0 = j_2(J_0)\). Since \(H(s) = j_3(s + C_2)\), the domain of \(H\) is \(I_1 = I_0 - C_2\). Hence, the size and position of this interval \(I_1\) are determined by the integration constants \(C_1\) and \(C_2\).  Now using this \(H\), we define the first characteristic function, \(h(s)=\frac{-H''(s)}{H'(s)}\). \\

We have shown that if there exists a function \(h\) such that for a given \(j\) the relation \(j(s) = \frac{h(K(s))}{K'(s)}\) holds (at least locally), where \(K = H^{-1}\) and \(H = \int e^{-\int h}\), then \(h\) must be of the form \(h(s) = \frac{-H''(s)}{H'(s)}\), with \(H(s) = j_3(s + C_2)\), \(j_3 = j_2^{-1}\), \(j_2 = \int \frac{1}{\sqrt{j_1}}\), and \(j_1 = -2\int j + C_1\).

Now we consider the converse: given the \(h\) constructed above, Remark \ref{remark: rln bw 1st and 2nd Char Fns} tells us that many functions \(j_{\alpha}\) can correspond to this \(h\), due to the integration constants \(C_3\) and \(C_4\) that arise when recovering \(H\) from \(h\). The question is: for which choice of \(C_3\) and \(C_4\) do we recover our original \(j\) (restricted to \(J_0\))?

The answer is \(C_3 = 0\) and \(C_4 = 0\). Here is the explanation: since \(h(s) = \frac{-H''(s)}{H'(s)} = -\frac{d}{ds}[\ln(H'(s))]\), we have \(-\int h = \ln(H'(s)) + C_3\), which implies \(H(s) = \int e^{-\int h} = e^{C_3} H(s) + C_4\). For this equality to hold, \(C_3\) and \(C_4\) must be zero. With this \(H\), we can then recover \(j\) uniquely via the equation \(j(H(s)) = -H''(s)\), or equivalently, \(j(s) = \frac{h(K(s))}{K'(s)}\).
 
We summarise this into the following proposition:
\begin{proposition}
	Let \( j: J \to \mathbb{R} \) be a smooth function that is nowhere vanishing on the interval \( J \). Then there exist constants \( C_1, C_2 \in \mathbb{R} \), a subinterval \( J_0 \subseteq J \), and a function \( h: I_1 \to \mathbb{R} \) defined on an interval \( I_1 \), such that the following holds:
	
	Define successively:
	\begin{align*}
		&j_1(s) = -2 \int j(s) ds + C_1, \quad	j_2(s)= \int \frac{1}{\sqrt{j_1(s)}} ds \quad \text{(on \( J_0: \) \;\; \(j_1(s)>0\))}\\
		j_3 &= j_2^{-1},\quad\quad H(s) = j_3(s + C_2), \quad\quad h(s)= -\frac{H''(s)}{H'(s)} \quad\quad I_1 = j_2(J_0)- C_2
	\end{align*}
	
	Then with this definition the functions \(j_2\) is a diffeomorphism and hence so does \(H\), so with \( K = H^{-1} \), the function \( h \) satisfies
	\[
	j(s) = \frac{h(K(s))}{K'(s)} \quad \text{for all } s \in H(I_1) \subseteq J_0 \;\; \text{and} \quad H = \int e^{-\int h}
	\]
	Moreover, the function \( h \) constructed in this way is the unique first characteristic function (up to the choice of \( C_1, C_2 \) determining the domain) corresponding to the given second characteristic function \( j \) on the specified domain.
\end{proposition}

We combine all the results in this subsection in the following theorem:\\
 
\begin{theorem}[Equivalence of \ref{eqn: Modified Complex Non-Trivial MGT: 2} and \ref{eqn: Complex Non-Trivial MGT: 2}]\label{thm: Equivalence thm}
	The following systems are equivalent through explicit transformations:
	
	\textbf{System 1} The pair \((f,h)\) solves the Non-Trivial Minimal Graph Transformation Problem \ref{eqn: Complex Non-Trivial MGT: 2}, where:
	\begin{itemize}
		\item \(f:\Omega\to\mathbb{R}\) is a minimal graph surface with \(f_z,f_{\bar{z}}\neq 0\)
		\item \(h:I\to\mathbb{R}\) is the first characteristic function, non-vanishing and sign-definite on \(I=\operatorname{Range}(f)\)
	\end{itemize}
	
	\textbf{System 2} The pair \((u,j)\) solves the Modified Non-Trivial Minimal Graph Transformation Problem \ref{eqn: Modified Complex Non-Trivial MGT: 2}, where:
	\begin{itemize}
		\item \(u:\Omega\to\mathbb{R}\) is harmonic with \(u_z,u_{\bar{z}}\neq 0\)
		\item \(j:J\to\mathbb{R}\) is the second characteristic function, non-vanishing and sign-definite on \(J=u(\Omega)\)
	\end{itemize}
	
	The equivalence is established through:
	
	\textbf{Forward Direction (1 $\Rightarrow$ 2):} Given $(f,h)$, define
	\[
	H(t)=\int e^{-\int h},\quad K=H^{-1},\quad u=H\circ f,\quad j(s)=\frac{h(K(s))}{K'(s)}
	\]
	Then $(u,j)$ solves System 2 with $j$ having the same sign as $h$.
	
	\textbf{Converse Direction (2 $\Rightarrow$ 1):} Given $(u,j)$, construct $h$ as follows. Choose constants $C_1,C_2\in\mathbb{R}$ and subinterval $J_0\subseteq J$ such that the following construction is valid:
	\begin{align*}
		&j_1(s)=-2\int j(s)ds+C_1,\quad j_2(s)=\int\frac{1}{\sqrt{j_1(s)}}ds \quad\text{(on $J_0$ where $j_1(s)>0$)}\\
		&j_3=j_2^{-1},\quad H(s)=j_3(s+C_2),\quad h(s)=-\frac{H''(s)}{H'(s)}, \quad \text{where}\;j_2,\,j_3,\,H,\,\text{are diffeomorphisms}
	\end{align*}
	Then for $f=K(u)$ with $K=H^{-1}$ and domain of definition of \(h\) being \(I_1 = j_2(J_0)- C_2\), the pair $(f,h)$ solves System 1 with $h$ having the same sign as $j$ on the domain.\\
	
	Moreover, the constants $C_1,C_2$ can be chosen so that the solution \(f,h\) of \ref{eqn: Complex Non-Trivial MGT: 2} corresponds to a solution \(u_0,j_0\) of \ref{eqn: Modified Complex Non-Trivial MGT: 2} that is defined on a subinterval \(J_0 \subseteq J\) of desirable size, centered around any  desired point \(s_0\in J\), and this correspondence is unique up to this choice of constants determining the domain of definition.
\end{theorem}  

\begin{remark}
	The algorithm given in this theorem is exhaustive: any solution \((u,j)\) of System~\ref{eqn: Modified Complex Non-Trivial MGT: 2} can be obtained via the forward direction (as described in the theorem) from some solution \((f,h)\) of System~\ref{eqn: Complex Non-Trivial MGT: 2}. Similarly, any solution \((f,h)\) of System~\ref{eqn: Complex Non-Trivial MGT: 2} can be obtained via the converse direction from some solution \((u,j)\) of System~\ref{eqn: Modified Complex Non-Trivial MGT: 2}.
\end{remark}

Due to Theorem \ref{thm: Equivalence thm}, it is sufficient to solve the Modified NMG Transformation Problem \ref{eqn: Modified Complex Non-Trivial MGT: 2} in order to find all minimal graph surfaces that admit a minimal graph transformation. We now apply a technique, which we will call \textbf{Weakening}, to obtain a system that is weaker than the Modified NMG Transformation Problem \ref{eqn: Modified Complex Non-Trivial MGT: 2}. This means the new system may contain more solutions than the original one. We will solve this weaker system and subsequently determine which of its solutions satisfy our original system.

\subsection{Step 3: Weakening of the system \ref{eqn: Modified Complex Non-Trivial MGT: 2}}
Let us begin with the Modified NMG Transformation Problem \ref{eqn: Modified Complex Non-Trivial MGT: 2}:
\[
u_{z\bar{z}} = 0 \quad \text{and} \quad \frac{u_z^2 u_{\bar{z}\bar{z}} + u_{\bar{z}}^2 u_{zz}}{u_z u_{\bar{z}}} = j(u).
\]

We are given that \(u_z\), \(u_{\bar{z}}\), and \(j\) are non-vanishing functions in their domains. Since \(u_{z\bar{z}} = 0\), \(u\) is harmonic and can be written as \(u(z,\bar{z}) = A(z) + \overline{A(z)}\). Let \(\overline{A(z)} = B(\bar{z})\). Then we have:
\[
u_z = A'(z), \quad u_{\bar{z}} = B'(\bar{z}), \quad u_{zz} = A''(z), \quad u_{\bar{z}\bar{z}} = B''(\bar{z}).
\]
Substituting these into the second equation yields:
\[
\frac{A''(z) [B'(\bar{z})]^2 + B''(\bar{z}) [A'(z)]^2}{A'(z) B'(\bar{z})} = j(u).
\]

Differentiating \(j(u)\) with respect to \(z\), we obtain:
\[
[j(u)]_z = j'(u) A'(z) = \frac{[B'(\bar{z})]^3 \left(A'''(z) A'(z) - [A''(z)]^2\right) + [B''(\bar{z}) B'(\bar{z})] \left(A''(z) [A'(z)]^2\right)}{[A'(z) B'(\bar{z})]^2}.
\]
Since \([j(u)]_z B'(\bar{z}) = j'(u) A'(z) B'(\bar{z})\), it follows that:
\begin{equation}
	\label{eqn: weakening equation 1}
	j'(u) A'(z) B'(\bar{z}) = \frac{[B'(\bar{z})]^4 \left(A'''(z) A'(z) - [A''(z)]^2\right) + [B''(\bar{z}) \cdot [B'(\bar{z})]^2] \left(A''(z) [A'(z)]^2\right)}{[A'(z) B'(\bar{z})]^2}.
\end{equation}

A similar calculation can be performed by differentiating \(j(u)\) with respect to \(\bar{z}\):
\[
[j(u)]_{\bar{z}} = j'(u) B'(\bar{z}) = \frac{[A'(z)]^3 \left(B'''(\bar{z}) B'(\bar{z}) - [B''(\bar{z})]^2\right) + [A''(z) A'(z)] \left(B''(\bar{z}) [B'(\bar{z})]^2\right)}{[A'(z) B'(\bar{z})]^2}.
\]
This implies:
\begin{equation}
	\label{eqn: weakening equation 2}
	j'(u) A'(z) B'(\bar{z}) = \frac{[A'(z)]^4 \left(B'''(\bar{z}) B'(\bar{z}) - [B''(\bar{z})]^2\right) + [A''(z) [A'(z)]^2] \left(B''(\bar{z}) [B'(\bar{z})]^2\right)}{[A'(z) B'(\bar{z})]^2}.
\end{equation}

We observe that the left-hand sides of \eqref{eqn: weakening equation 1} and \eqref{eqn: weakening equation 2} are identical. Therefore, we equate the right-hand sides. Since the denominators are the same, we cancel them. The second terms in the numerators are also identical, so we subtract them. This yields:
\[
[A'(z)]^4 \left(B'''(\bar{z}) B'(\bar{z}) - [B''(\bar{z})]^2\right) = [B'(\bar{z})]^4 \left(A'''(z) A'(z) - [A''(z)]^2\right).
\]

Rearranging, we obtain:
\[
\frac{B'''(\bar{z}) B'(\bar{z}) - [B''(\bar{z})]^2}{[B'(\bar{z})]^4} = \frac{A'''(z) A'(z) - [A''(z)]^2}{[A'(z)]^4}.
\]

In this equation, the left-hand side is a function of \(\bar{z}\) alone, and the right-hand side is a function of \(z\) alone. Therefore, both sides must be equal to a constant \(k\). That is,
\begin{equation}
	\label{eqn: weakened eqn 1}
	\boxed{\frac{B'''(\bar{z}) B'(\bar{z}) - [B''(\bar{z})]^2}{[B'(\bar{z})]^4} = k = \frac{A'''(z) A'(z) - [A''(z)]^2}{[A'(z)]^4}}
\end{equation}

Now note that since \(u = A(z) + B(\bar{z})\), where \(B(\bar{z}) = \overline{A(z)}\) and \(A(z)\) is a holomorphic function of \(z\), we have:
\[
A(z) = \sum_{k=0}^{\infty} a_k z^k, \qquad B(\bar{z}) = \overline{A(z)} = \sum_{k=0}^{\infty} \bar{a}_k \bar{z}^k.
\]

From this, we obtain \(B'(\bar{z}) = \overline{A'(z)}\), \(B''(\bar{z}) = \overline{A''(z)}\), and \(B'''(\bar{z}) = \overline{A'''(z)}\). Substituting into Equation \ref{eqn: weakened eqn 1} gives:
\[
\frac{A'''(z) A'(z) - [A''(z)]^2}{[A'(z)]^4} = k,
\qquad
\overline{\left( \frac{A'''(z) A'(z) - [A''(z)]^2}{[A'(z)]^4} \right)} = k.
\]

This shows that \(k\) must be real, since it is equal to its own conjugate. Therefore, we only need to solve the equation for \(A(z)\). Rearranging, we get:
\[
\frac{A'''(z) A'(z) - [A''(z)]^2}{[A'(z)]^2} = k [A'(z)]^2, \quad \text{for a real } k.
\]

Observe that the left-hand side is \(\left( \frac{A''(z)}{A'(z)} \right)_z = [\ln(A'(z))]_{zz}\). Hence, we obtain, what we call, \textbf{Weakened Equation for System \ref{eqn: Modified Complex Non-Trivial MGT: 2}}:
\begin{equation}
	\label{eqn: weakened eqn 2}
	\boxed{[\ln(A'(z))]_{zz} = k  [A'(z)]^2}
	\qquad
	\boxed{\text{for a real } k}.
\end{equation}

The strategy now is to solve Equation \ref{eqn: weakened eqn 2} explicitly for the following three cases: \(k = 0\), \(k < 0\), and \(k > 0\). This will be carried out in the following section.

\section{Solution to Weakened Equation \ref{eqn: weakened eqn 2} : Case \(k=0\)}
Before beginning the main calculation, let us recall the properties of our system: \(u : \Omega \longrightarrow \mathbb{R}\) is a harmonic function satisfying the Modified NMG Transformation Problem \ref{eqn: Modified Complex Non-Trivial MGT: 2}, and its complex derivatives \(u_z\) and \(u_{\bar{z}}\) are non-vanishing in \(\Omega\). The function \(j : J = u(\Omega) \longrightarrow \mathbb{R}\) is also non-vanishing on its domain.\\

Let us begin by substituting \(k = 0\) into Equation \ref{eqn: weakened eqn 2}, which yields \([\ln(A'(z))]_{zz} = 0\). Since \(A'(z)\) is a function of \(z\) alone, this implies \(\ln(A'(z)) = a z + b\) for some complex constants \(a\) and \(b\). Exponentiating gives \(A'(z) = e^{a z + b}\). This leads to the following two cases: \(a = 0\) and \(a \neq 0\).

\subsection{Subcase 1: \(a = 0\)}
In this case we get \(A'(z) = e^{b}\), so \(A(z) = e^{b} z + c\). Hence \(B(\bar{z}) = \overline{A(z)} = e^{\bar{b}} \bar{z} + \bar{c}\), and \(u = e^{b} z + e^{\bar{b}} \bar{z} + \bar{c} + c\). We calculate: \(u_z = e^{b}\), \(u_{\bar{z}} = e^{\bar{b}}\), and \(u_{zz} = 0 = u_{\bar{z}\bar{z}}\). Substituting into Equation \ref{eqn: Modified Complex Non-Trivial MGT: 2} gives \(j(u) = 0\), which contradicts the hypothesis that \(j\) is non-vanishing. Therefore, \(a \neq 0\).

\subsection{Subcase 2: \(a \neq 0\)}

\textbf{Finding solutions \((u,j)\) to Modified Nontrivial Minimal Graph Transformation Problem \ref{eqn: Modified Complex Non-Trivial MGT: 2} : } We have \(A'(z) = e^{a z + b}\), hence \(A(z) = \frac{1}{a} e^{a z + b} + c_0\) for some complex constants \(a\), \(b\), and \(c_0\). This gives \(B(\bar{z}) = \overline{A(z)} = \frac{1}{\bar{a}} e^{\bar{a} \bar{z} + \bar{b}} + \bar{c_0}\). Therefore, the function \(u\) is:
\begin{equation}
	\label{eqn: u, k=0}
	\boxed{
		u = \frac{1}{a} e^{a z + b} + c_0 + \frac{1}{\bar{a}} e^{\bar{a} \bar{z} + \bar{b}} + \bar{c_0} = 2\Re\left(\frac{1}{a} e^{a z + b}\right) + 2\Re(c_0)
	}
\end{equation}
for some complex constants \(a\), \(b\), and \(c_0\).

Now we find the second characteristic function \(j\). We have \(u_z = A'(z) = e^{a z + b}\), \(u_{\bar{z}} = B'(\bar{z}) = e^{\bar{a} \bar{z} + \bar{b}}\), \(u_{zz} = a e^{a z + b}\), and \(u_{\bar{z}\bar{z}} = \bar{a} e^{\bar{a} \bar{z} + \bar{b}}\). Substituting into Equation \ref{eqn: Modified Complex Non-Trivial MGT: 2}:
\[
\frac{u_z^2 u_{\bar{z}\bar{z}} + u_{\bar{z}}^2 u_{zz}}{u_z u_{\bar{z}}} = \bar{a} e^{a z + b} + a e^{\bar{a} \bar{z} + \bar{b}} = a \bar{a} \left[u - (c_0 + \bar{c_0})\right] = |a|^2 \left[u - 2\Re(c_0)\right] = j(u).
\]
Hence, the second characteristic function is:
\begin{equation}
	\label{eqn: j, k=0}
	\boxed{j(s) = |a|^2 \left[s - 2\Re(c_0)\right]}
\end{equation}
Combining both results we get the pair \((u,j)\) which satisfy the Modified NMG Transformation Problem:
\[
\boxed{u = 2\Re\left(\tfrac{1}{a} e^{a z + b}\right) + 2\Re(c_0)} \quad\text{and}\quad \boxed{j(s) = |a|^2 \left[s - 2\Re(c_0)\right]}
\]
We observe that on the open intervals \(J = (\beta, \infty)\) and \(J = (-\infty, \beta)\), the second characteristic function \(j\) is non-vanishing. Furthermore, the functions \(u_z\) and \(u_{\bar{z}}\) are non-vanishing on all of \(\mathbb{C}\). Therefore, for \(\Omega_0 = u^{-1}(\beta, \infty)\) or \(\Omega_0 = u^{-1}(-\infty, \beta)\), we have \(u_z, u_{\bar{z}} \neq 0\) on \(\Omega_0\), and \(u(\Omega_0) = J\). Thus, the pair \((u, j)\) is the solution of the Modified NMG Transformation Problem, with \(u : \Omega_0 \to \mathbb{R}\) and \(j : J \to \mathbb{R}\), satisfies the non-vanishing conditions in their respective domains.

Hence, \((u, j)\) satisfies the hypotheses of Theorem \ref{thm: Equivalence thm}. So there exist a Minimal Graph Surface \(f\) and First Characteristic Function \(h\) associated to it, which solves the NMG Transformation Problem. We now apply the converse direction of Theorem \ref{thm: Equivalence thm} to construct all posssible, open subinterval \(J_0 \subseteq J\), pair \((f, h)\) that solves the NMG Transformation Problem \ref{eqn: Complex Non-Trivial MGT: 2} on the respectively on domains \(\Omega = u^{-1}(J_0)\) and \(I = f(\Omega)\), such that the corresponding pair \((u_0, j_0)\) for the modified problem is precisely \(u_0 = u|_{\Omega}\) and \(j_0 = j|_{J_0}\).

We now proceed to determine the explicit form of the minimal surface \(f\) and the associated first characteristic function \(h\).\\

\textbf{Finding solutions \((f,h)\) to the Nontrivial Minimal Graph Transformation Problem \ref{eqn: Complex Non-Trivial MGT: 2}:} We begin with the construction described in Theorem \ref{thm: Equivalence thm}. We have \(j(s) = |a|^2 [s - 2\Re(c_0)]\). Then, \(j_1 = -2\int j\, ds + C_1= -2|a|^2 \left[\frac{s^2}{2} - 2\Re(c_0)s\right] + C_1\), which can be rewritten as: \(j_1 = -|a|^2 s^2 + 4|a|^2 \Re(c_0) s + C_1\). The constant \(C_1\) must be chosen so that \(j_1(s) > 0\) on some interval. Since the leading coefficient \(-|a|^2\) is negative, \(j_1(s) \to -\infty\) as \(s \to \pm\infty\). Therefore, if such an interval \(J_0\) exists, it must be bounded. Next, we compute \(j_2(s) = \int \frac{1}{\sqrt{-|a|^2 s^2 + 4|a|^2 \Re(c_0) s + C_1}}\, ds + C_2\). Again, the leading coefficient is negative, and the discriminant of this quadratic must be positive for \(J_0\) to exist, which is ensured by an appropriate choice of real \(C_1\). To evaluate this integral, we refer to \cite{JD}, page 172, section 4.3.4.1, equation 1. We get the formula: \(j_2(s)=\frac{-1}{|a|} \sin^{-1}\big[ \frac{4|a|^2 \Re(c_0)-2|a|^2s}{\sqrt{16|a|^4 \Re^2(c_0)+4|a|^2 C_1}}\big]\).

Now substitute the relation \(j_2(H(s)) = s + C_2\), yielding the following equation: \(\frac{-1}{|a|} \sin^{-1}\big[ \frac{4|a|^2 \Re(c_0)-2|a|^2H(s)}{\sqrt{16|a|^4 \Re^2(c_0)+4|a|^2 C_1}}\big]=s+C_2\). This immediately give us the following: \(H(s) = \sin\left(|a|(s + C_2)\right) \tfrac{\sqrt{4|a|^2 \Re^2(c_0) + C_1}}{|a|} + 2\Re(c_0)\). Note that by appropriately choosing the integration constant \(C_1\), we can vary the value of \(\tfrac{\sqrt{4|a|^2 \Re^2(c_0) + C_1}}{|a|}\) to any positive real number. Let us denote this positive constant as \(C_3^2\), giving us \(H(s) = C_3^2 \sin\left(|a|(s + C_2)\right) + 2\Re(c_0)\). Let us find \(H^{-1}\). We have \(s = \frac{1}{|a|} \sin^{-1}\left[\frac{H - 2\Re(c_0)}{C_3^2}\right] - C_2\). Therefore, \(K(s) = H^{-1}(s) = \frac{1}{|a|} \sin^{-1}\left[\frac{s - 2\Re(c_0)}{C_3^2}\right] - C_2\). We now use the formulas \(f=K(u)\) and \(h=\frac{-H''}{H'}\) to finally find the Minimal Graph Surface \(f\) and the first characteristic function \(h\) which satisfies the NMG Transformation Problem.
\[
  f=K(u)=\frac{1}{|a|} \sin^{-1}\Big[\frac{u - 2\Re(c_0)}{C_3^2}\Big]-C_2=\frac{1}{|a|} \sin^{-1}\Big[\frac{2\Re\big(\frac{1}{a} e^{a z + b}\big)}{C_3^2}\Big]-C_2
\]
\[
  h=\frac{-H''}{H'}=\frac{C_3^2|a|^2\sin\big(|a|(s + C_2)\big)}{C_3^2|a|\cos\big(|a|(s + C_2)\big)}=|a|\tan\big(|a|(s + C_2)\big)
\]
So the pair \(f,h\) is precisely:
\begin{equation}
	\label{eqn: soln k=0}
	\boxed{f=\frac{1}{|a|} \sin^{-1}\Big[\frac{2\Re\big(\frac{1}{a} e^{a z + b}\big)}{C_3^2}\Big]-C_2}\quad\text{and}\quad\boxed{h=|a|\tan\big(|a|(s + C_2)\big)}
\end{equation}

The Minimal Graph Surface \(f\) is infact the Scherk’s First Surface, We will prove that now:\\

\textbf{Proving that \(f\) is Scherk’s First Surface:} In this section, we use \(x_1, x_2, x_3\) to represent the \(x, y, z\) components, respectively, to avoid confusion with the complex variable \(z = x_1 + i x_2\). Let \(a = a_1 + i a_2\) and \(b = b_1 + i b_2\). Then \(a z + b = (a_1 x_1 - a_2 x_2 + b_1) + i (a_2 x_1 + a_1 x_2 + b_2)\). Define \(A = a_1 x_1 - a_2 x_2 + b_1\) and \(B = a_2 x_1 + a_1 x_2 + b_2\), so that \(e^{a z + b} = e^{A} (\cos B + i \sin B)\). Then \(2\Re\left(\frac{1}{a} e^{a z + b}\right) = 2\Re\left(\frac{1}{a_1 + i a_2} e^{A} (\cos B + i \sin B)\right) = \frac{2 e^{A}}{a_1^2 + a_2^2} (a_1 \cos B + a_2 \sin B)\). This implies \(f = \frac{1}{\sqrt{a_1^2 + a_2^2}} \sin^{-1}\left[\frac{2 e^{A}}{C_3^2 (a_1^2 + a_2^2)} (a_1 \cos B + a_2 \sin B)\right] - C_2\). Let \(\theta = \tan^{-1}(a_2 / a_1)\). Then \(f = \frac{1}{\sqrt{a_1^2 + a_2^2}} \sin^{-1}\Big[\frac{2 e^{A}}{C_3^2 \sqrt{a_1^2 + a_2^2}} \cos[B - \theta]\Big] - C_2\). Since \(f(x_1, x_2) = x_3\) represents the \(z\)-component of the minimal graph surface, we have \(\sin\left[\sqrt{a_1^2 + a_2^2} (x_3 + C_2)\right] = \frac{2 e^{A}}{C_3^2 \sqrt{a_1^2 + a_2^2}} \cos[B - \theta]\). Now set \(C_4 = C_2 - \frac{\pi}{2\sqrt{a_1^2 + a_2^2}}\) and rearrange, then apply logarithm (\(\ln\)) on both sides to get \(\ln\Big[\frac{\cos\left[\sqrt{a_1^2 + a_2^2} (x_3 + C_4)\right]}{\cos[B - \theta]}\Big] = A + \ln\Big[\frac{2}{C_3^2 \sqrt{a_1^2 + a_2^2}}\Big]\). Substituting back the values of \(A\) and \(B\) yields:
\[
 \boxed{ \ln\Big[\frac{\cos[ x_3\sqrt{a_1^2 + a_2^2} + C_4\sqrt{a_1^2 + a_2^2}\,]}{\cos[a_2 x_1 + a_1 x_2 + b_2 - \theta]}\Big] = a_1 x_1 - a_2 x_2 + b_1 + \ln\Big[\tfrac{2}{C_3^2 \sqrt{a_1^2 + a_2^2}}\Big]}.
\]

Set the constants: \(C_Z = C_4 \sqrt{a_1^2 + a_2^2}\), \(C_Y = b_2 - \theta\), and \(C_X = b_1 + \ln\left[\frac{2}{C_3^2 \sqrt{a_1^2 + a_2^2}}\right]\). Define the variables: \(Z = x_3 \sqrt{a_1^2 + a_2^2}\), \(X = a_2 x_1 + a_1 x_2\), and \(Y = a_1 x_1 - a_2 x_2\). We rewrite the above equation as: \(\boxed{Y + C_Y = \ln\left[\frac{\cos(Z + C_Z)}{\cos(X + C_X)}\right]}\)\\

Clearly, the set of all \((X, Y, Z)\) satisfying this equation forms a Scherk’s First Surface translated by the constant vector \((C_X, C_Y, C_Z)\). The aim of this section is to determine whether the set of all \((x_1, x_2, x_3)\) satisfying this equation also forms a Scherk’s First Surface. To analyze this, consider the following matrix relation.

\[
\begin{pmatrix} 
	X\\
	Y \\
	Z
\end{pmatrix}
=
\begin{pmatrix} 
	a_2 & a_1 & 0 \\
	a_1 & -a_2 & 0 \\
	0 & 0 & \sqrt{a_1^2+a_2^2}
\end{pmatrix}
\begin{pmatrix} 
	x_1 \\
	x_2 \\
	x_3
\end{pmatrix}
= \sqrt{a_1^2+a_2^2}\,
\begin{pmatrix} 
	\dfrac{a_2}{\sqrt{a_1^2+a_2^2}} & \dfrac{a_1}{\sqrt{a_1^2+a_2^2}} & 0 \\
	\dfrac{a_1}{\sqrt{a_1^2+a_2^2}} & -\dfrac{a_2}{\sqrt{a_1^2+a_2^2}} & 0 \\
	0 & 0 & 1
\end{pmatrix}
\begin{pmatrix} 
	x_1 \\
	x_2 \\
	x_3
\end{pmatrix}.
\]

This shows that the coordinates \((X, Y, Z)\) are obtained from \((x_1, x_2, x_3)\) through a transformation involving an \textbf{orthogonal matrix} followed by a \textbf{uniform scaling} by the factor \(\sqrt{a_1^2 + a_2^2}\).

\[
\begin{pmatrix} 
	x_1 \\
	x_2 \\
	x_3
\end{pmatrix}
= \frac{1}{\sqrt{a_1^2+a_2^2}}\,
\begin{pmatrix} 
	\dfrac{a_2}{\sqrt{a_1^2+a_2^2}} & \dfrac{a_1}{\sqrt{a_1^2+a_2^2}} & 0 \\
	\dfrac{a_1}{\sqrt{a_1^2+a_2^2}} & -\dfrac{a_2}{\sqrt{a_1^2+a_2^2}} & 0 \\
	0 & 0 & 1
\end{pmatrix}
\begin{pmatrix} 
	X \\
	Y \\
	Z
\end{pmatrix}.
\]

Thus, \((x_1, x_2, x_3)\) is obtained from \((X, Y, Z)\) by applying an \textbf{orthogonal transformation} (which may correspond to a rotation or reflection) and then \textbf{scaling} by the factor \(\frac{1}{\sqrt{a_1^2 + a_2^2}}\).

Since \((X, Y, Z)\) represents a \textbf{Scherk’s surface}, it follows that \((x_1, x_2, x_3)\) also defines a \textbf{Scherk’s surface}, differing only by a rigid motion and uniform scaling. Hence, the function \(f\) indeed represents a Scherk’s surface.\\

\textbf{Finding Non-Trivial Minimal Graph Transformations.}  
We can directly determine the NMG Transformations \(g\) by applying Formula~\ref{eqn: conv h to g}, we determine \(g\) to be given by  \(g = \pm \int \frac{1}{\sqrt{1 \pm C^2 e^{2\int h\, dt}}}\, dt + D\). First, we calculate \(\int h(s)\, ds\). Using the identity \(\int \tan x \, dx = -\ln|\cos x|\) and noting that \(h = |a|\tan[|a|(s + C_2)]\), we have \(\int h(s)\, ds = \int |a|\tan[|a|(s + C_2)]\, ds = -\ln\!\big|\cos[|a|(s + C_2)]\big|\). Next, we compute \(e^{2\int h(s)\, ds}\) as follows: \(e^{2\int h(s)\, ds} = e^{-2\ln\!\big|\cos[|a|(s + C_2)]\big|} = \frac{1}{\cos^2[|a|(s + C_2)]}\). Now, we compute \(g\) as follows:  
\(g = \pm \scaleobj{1}{\int} \scaleobj{0.8}{\frac{1}{\sqrt{1\, \pm\, \tfrac{C^2}{\cos^2[|a|(s + C_2)]}}}\, ds + D}.\)  This can be rewritten as  \(g = \pm \int \frac{\cos[|a|(s + C_2)]}{\sqrt{\cos^2[|a|(s + C_2)] \pm C^2}}\, ds + D = \pm \int \frac{\cos[|a|(s + C_2)]}{\sqrt{1 \pm C^2 - \sin^2[|a|(s + C_2)]}}\, ds + D.\) Let us make the substitution \(t = \sin[|a|(s + C_2)]\), giving \(dt = |a|\cos[|a|(s + C_2)]\, ds\), or equivalently, \(\frac{dt}{|a|} = \cos[|a|(s + C_2)]\, ds.\)  
Thus, we have  
\(g = \pm \frac{1}{|a|} \int \frac{1}{\sqrt{1 \pm C^2 - t^2}}\, dt,\)  
evaluated at \(t = \sin[|a|(s + C_2)].\) Notice that \(g\) has a real solution only if \(1 \pm C^2 > 0.\) We solve the above integral using the well-known standard result \(\int \frac{1}{\sqrt{c - w^2}}\, dw = \sin^{-1}\!\left(\frac{w}{\sqrt{c}}\right) + \text{constant}.\) Hence, we obtain  
\(g(s) = \pm \tfrac{1}{|a|}\sin^{-1}\big[\tfrac{\sin[|a|(s + C_2)]}{\sqrt{1 \pm C^2}}\big] + D.\)

Since \(\sqrt{1 \pm C^2}\) can be either positive or negative, if it is negative, we factor out the \(-1\) from \(\sin^{-1}\) (as it is an odd function), which does not change the overall sign outside the integral because of the existing \(\pm\) symbol. Therefore, without loss of generality, we may assume that \(\sqrt{1 \pm C^2} > 0\), and denote it by \(C_5^2.\) Hence, we finally obtain  
\[
\boxed{g(s) = \pm \frac{1}{|a|}\sin^{-1}\!\left(\frac{\sin[|a|(s + C_2)]}{C_5^2}\right) + D.}
\]\\

\textbf{Finding the Transformed Minimal Surface.}  
We substitute \(f\) in place of \(s\), giving the expression for the transformed minimal surface:  
\(g(f) = \pm \frac{1}{|a|}\sin^{-1}\!\left(\frac{\sin[|a|(f + C_2)]}{C_5^2}\right) + D 
= \pm \frac{1}{|a|}\sin^{-1}\!\Big(\frac{2\Re\left[\frac{1}{a} e^{az + b}\right]}{(C_3 \cdot C_5)^2}\Big) + D.\) Hence, we obtain  
\[
\boxed{g(f) = \pm \tfrac{1}{|a|}\sin^{-1}\Big(\tfrac{2\Re\big(\frac{1}{a} e^{az + b}\big)}{(C_3 \cdot C_5)^2}\Big) + D.}
\]
This represents another surface within the family of Scherk’s First Surfaces.\\

\textbf{Solution and its Maximal Domain of Definition:} 
The characterisation of the solution to the NMG Transformation Problem \ref{eqn: Complex Non-Trivial MGT: 2} under the Choice of Domain as in \ref{section: assumptions} and \ref{section: assumption 2} and the condition that the real constant $k$ (as in Equation \ref{eqn: weakened eqn 2}) is zero, is the following:

If such a Minimal Graph $f$ with a NMG Transformation $g$ exists under the above criteria, then it must be the \textbf{Scherk's First Surface} given by Formula:
\[
\boxed{f=\tfrac{1}{|a|} \sin^{-1}\Big[\tfrac{2\Re\big[\frac{1}{a} e^{a z + b}\big]}{C_3^2}\Big]-C_2}.
\]
where \(z=x+iy\in\mathbb{C}\). Moreover, the NMG Transformation $g$ for this Minimal Graph Surface \(f\) should be:
\[
\boxed{g(s) = \pm \tfrac{1}{|a|}\sin^{-1}\left(\tfrac{\sin[|a|(s + C_2)]}{C_5^2}\right) + D.}
\]
Finally, the Transformed Minimal Surface under the the NMG Transformation $g$  is also  \textbf{Scherk's First Surface}, and is given by: 
\[
  \boxed{g(f) = \pm \tfrac{1}{|a|}\sin^{-1}\left(\tfrac{2\Re\left(\tfrac{1}{a} e^{az + b}\right)}{(C_3 \cdot C_5)^2}\right) + D.}
\]
We note that for the NMG Transformation $g$ to be defined, we must have \(\tfrac{\sin[|a|(s + C_2)]}{C_5^2} \in [-1,1]\), or equivalently, \(s \in \tfrac{1}{|a|}\sin^{-1}[-C_5^2, C_5^2] - C_2\). Call this set \(U\). Thus, \(U\) is the maximal domain of definition of the NMG Transformation $g$. Now, \(U = \mathbb{R}\) if \(C_5^2 \geq 1\); otherwise, it is a proper subset of \(\mathbb{R}\).

Next, we determine the maximal domain of definition \(\Omega\) for the Minimal Graph Surface \(f\). For \(f\) to be defined, the function \(\frac{2\Re\big[\frac{1}{a} e^{a z + b}\big]}{C_3^2}\) must take values in \([-1,1]\) on some domain \(\Omega'\). Furthermore, since the range of \(f\) must lie within \(U\), the maximal domain of \(g\), we define \(\Omega'' = f^{-1}(U)\). Therefore, the maximal domain of definition of the Minimal Graph Surface \(f\) is given by \(\Omega = \Omega' \cap \Omega''.\) 

\section{Method for Solving the Weakened Equation \ref{eqn: weakened eqn 2} for the Case \(k \neq 0\)}
In this section, we devise a method to solve the Weakened Equation \ref{eqn: weakened eqn 2} for the case where \(k \neq 0\). The approach is to decompose the given equation into two simpler equations, which can then be solved independently. We begin with the Weakened Equation \ref{eqn: weakened eqn 2}: \([\ln(A'(z))]_{zz} = k [A'(z)]^2, \quad \text{for } k \neq 0.\) Let us introduce a holomorphic function \(\mathcal{A}\) defined by the relation \(\frac{\mathcal{A}}{\sqrt{|k|}} = A.\) This substitution transforms the Weakened Equation \ref{eqn: weakened eqn 2} into an equivalent equation in terms of \(\mathcal{A}\): \(\big[\ln\big(\tfrac{\mathcal{A}'(z)}{\sqrt{|k|}}\big)\big]_{zz} 
= \frac{k}{|k|} [\mathcal{A}'(z)]^2\). Notice that \(\big[\ln(\tfrac{\mathcal{A}'(z)}{\sqrt{|k|}})\big]_{zz} = [\ln(\mathcal{A}'(z))]_{zz}\), and \(\tfrac{k}{|k|} = \pm 1\). Therefore, we obtain two distinct equations depending on the sign of \(k\):
\begin{equation}
	\label{eqn: weakened eqn 2 k>0}
	[\ln(\mathcal{A}'(z))]_{zz} = [\mathcal{A}'(z)]^2 
	\qquad \textbf{(Weakened Equation for } k > 0\textbf{)}
\end{equation}
\begin{equation}
	\label{eqn: weakened eqn 2 k<0}
	[\ln(\mathcal{A}'(z))]_{zz} = -[\mathcal{A}'(z)]^2 
	\qquad \textbf{(Weakened Equation for } k < 0\textbf{)}
\end{equation}
We will analyze and solve these two equations in the subsequent sections.

\section{Solution to the Weakened Equation for \(k>0\)}
In this section, we will analytically solve Equation \ref{eqn: weakened eqn 2 k>0}: \([\ln(\mathcal{A}'(z))]_{zz} = [\mathcal{A}'(z)]^2 \) and thereby obtain all minimal surfaces for which the First Characteristic Function $h(s)$ is nonzero and the constant $k$ is positive. Furthermore, we will determine all possible NMG Transformations $g$ associated with these surfaces and corresponding Transformed Minimal Surfaces. We make the following substitution: $\mathcal{B}(z)=\ln[\mathcal{A}'(z)]$. Hence our Equation \ref{eqn: weakened eqn 2 k>0} is transformed to: \(\mathcal{B}_{zz}=e^{2\mathcal{B}}\). Now multiply LHS and RHS with $2\mathcal{B}_z$, we get: \(2\mathcal{B}_z\mathcal{B}_{zz}=[\mathcal{B}_z^2]_z=[e^{2\mathcal{B}}]_z=e^{2\mathcal{B}}2\mathcal{B}_z\). Which means that:
\[
  \mathcal{B}_z^2=e^{2\mathcal{B}}+c_0 \qquad\text{for some complex constant} c_0
\]
then:
\[
  \mathcal{B}_z=\pm \sqrt{e^{2\mathcal{B}}+c_0}
\]
\textbf{This leads us to 2 cases}: $c_0=0$ and $c_0\neq0$

\subsection{Subcase 1: $c_0=0$}
\textbf{Finding solutions \((u,j)\) to the Modified Nontrivial Minimal Graph Transformation Problem \ref{eqn: Modified Complex Non-Trivial MGT: 2}:} We take the complex constant \(c_0 = 0\). This implies that \(\mathcal{B}_z = \pm \sqrt{e^{2\mathcal{B}}} = \pm e^{\mathcal{B}}\). Let us denote \(\pm 1 = \epsilon_0\), so that we have \(dz = \epsilon_0 e^{-\mathcal{B}} d\mathcal{B}\). Integrating both sides gives \(-\epsilon_0 e^{-\mathcal{B}} = z + c_1\) for some complex constant \(c_1\). This further implies 
\(-\mathcal{B} = \ln[-\epsilon_0(z + c_1)]\), and hence \(\mathcal{B}=\ln\big[\tfrac{-\epsilon_0}{c_1 + z}\big]\). But we know that \(\mathcal{B} = \ln[\mathcal{A}'(z)]\). Therefore, \(\ln[\mathcal{A}'(z)]=\ln\big[\tfrac{-\epsilon_0}{c_1 + z}\big]\).
Exponentiating both sides gives \(\mathcal{A}'(z)=\tfrac{-\epsilon_0}{c_1 + z}\). Integrating both sides, we obtain \(\mathcal{A}(z)=\int\frac{-\epsilon_0}{c_1 + z}=-\epsilon_0\ln(c_1 + z)+c_2\) where \(c_1, c_2 \in \mathbb{C}\) are constants. Now, we compute \(A(z) = \frac{\mathcal{A}}{\sqrt{k}}\). Since \(c_2\) is an arbitrary constant, we may absorb the factor \(\frac{1}{\sqrt{k}}\) into it and denote it again as \(c_2\). Thus, we have \(A(z)= \frac{-\epsilon_0}{\sqrt{k}}\ln(c_1 + z)+c_2\). By definition, \(A(z)\) satisfies \(u(z, \bar{z}) = A(z) + \overline{A(z)}\). Hence,
\[
 u(z, \bar{z}) = \tfrac{-\epsilon_0}{\sqrt{k}} \ln(z + c_1) + c_2 + \tfrac{-\epsilon_0}{\sqrt{k}} \ln(\bar{z} + \bar{c}_1) + \bar{c}_2 = \tfrac{-2\epsilon_0}{\sqrt{k}} \Re[\ln(z + c_1)] + 2\Re[c_2].
\]
So, we obtain the formula for \(u\): \(\boxed{u(z,\bar z) = \tfrac{-2\epsilon_0}{\sqrt{k}} \Re[\ln(c_1 + z)] + 2\Re[c_2]}\,\) for some complex constants \(c_1\) and \(c_2\), where \(\epsilon_0 = \pm 1\). Now, let us calculate the second characteristic function \(j\). We have \(\frac{u_z^2u_{\bar z\bar z}+u_{\bar z}^2u_{zz}}{u_z u_{\bar z}} = j(u).\)

We compute the derivatives: \(u_z = \tfrac{-\epsilon_0}{(c_1 + z)\sqrt{k}}, \,\,\, u_{\bar z} = \tfrac{-\epsilon_0}{(\bar c_1 + \bar z)\sqrt{k}}, \,\,\, u_{zz} = \tfrac{\epsilon_0}{(c_1 + z)^2\sqrt{k}}, \,\,\, u_{\bar z\bar z} = \tfrac{\epsilon_0}{(\bar c_1 + \bar z)^2\sqrt{k}}\). Substituting these into the defining equation for \(j(u)\), we get
\[
\tfrac{u_z^2u_{\bar z\bar z}+u_{\bar z}^2u_{zz}}{u_z u_{\bar z}} = \big[\tfrac{2\epsilon_0}{(\bar c_1 + \bar z)^2 (c_1 + z)^2 k \sqrt{k}}\big]/\big[\tfrac{1}{(\bar c_1 + \bar z)(c_1 + z) k}\big]= \tfrac{2\epsilon_0}{(\bar c_1 + \bar z)(c_1 + z)\sqrt{k}} = j(u).
\]
We observe that: \(\tfrac{2\epsilon_0}{j(u)\sqrt{k}} = (\bar c_1 + \bar z)(c_1 + z)
= e^{-\epsilon_0 \sqrt{k}(u - 2\Re[c_2])}.\) Hence \(\boxed{j(u) = \tfrac{2\epsilon_0}{\sqrt{k}} e^{\epsilon_0 \sqrt{k} (u - 2\Re[c_2])}}.\)

Thus, we have found the pair \((u,j)\) that satisfies the Modified NMG Transformation Problem:
\[
\boxed{u(z,\bar z)=\tfrac{-2\epsilon_0}{\sqrt{k}}\Re[\ln(c_1 +  z)]+2\Re[c_2]}\qquad\text{and}\qquad\boxed{j(s)=\tfrac{2\epsilon_0}{\sqrt{k}}e^{\epsilon_0\cdot\sqrt{k}\cdot(s-2\Re[c_2])}}
\]
We note that the second characteristic function \(j\) remains non-vanishing throughout \(\mathbb{R}\). Likewise, the derivatives \(u_z\) and \(u_{\bar{z}}\) are non-zero everywhere on \(\mathbb{C}\). Consequently, for \(\Omega_0 = u^{-1}(\mathbb{R})\), we have \(u_z, u_{\bar{z}} \neq 0\) on \(\Omega_0\), and \(u(\Omega_0) = J\). Hence, the pair \((u, j)\) constitutes a valid solution of the Modified NMG Transformation Problem, where \(u : \Omega_0 \to \mathbb{R}\) and \(j : J \to \mathbb{R}\), both satisfy the required non-vanishing conditions on their respective domains. Therefore, \((u, j)\) meets the hypotheses of Theorem \ref{thm: Equivalence thm}.

By Theorem \ref{thm: Equivalence thm}, there exists a corresponding Minimal Graph Surface \(f\) and an associated First Characteristic Function \(h\) that together solve the NMG Transformation Problem. Next, applying the converse part of Theorem \ref{thm: Equivalence thm}, we can construct all possible open subintervals \(J_0 \subseteq J\) and corresponding pairs \((f, h)\) which solve the NMG Transformation Problem \ref{eqn: Complex Non-Trivial MGT: 2} on the domains \(\Omega = u^{-1}(J_0)\) and \(I = f(\Omega)\) such that the corresponding pair \((u_0, j_0)\) for the modified problem satisfies \(u_0 = u|_{\Omega}\) and \(j_0 = j|_{J_0}.\)

We now proceed to determine the explicit form of the minimal surface \(f\) and the associated first characteristic function \(h\).

\textbf{Finding solutions \((f,h)\) to the Nontrivial Minimal Graph Transformation Problem \ref{eqn: Complex Non-Trivial MGT: 2}:} We will proceed in the construction given in the converse direction of the Equivalence Theorem \ref{thm: Equivalence thm}: We have \(j_1(s)=-2\int j(s)ds=-2\int \tfrac{2\epsilon_0}{\sqrt{k}}e^{\epsilon_0\cdot\sqrt{k}\cdot(s-2\Re[c_2])} ds\) which is: \(j_1(s)=\tfrac{-4}{k}e^{\epsilon_0\cdot\sqrt{k}\cdot(s-2\Re[c_2])}+C_1\), as we see in the theorem \ref{thm: Equivalence thm} we must make choice of \(C_1\) such that \(j_1>0\), but since \(\tfrac{-4}{k}e^{\epsilon_0\cdot\sqrt{k}\cdot(s-2\Re[c_2])}<0\) as \(k>0\) and exponential map also positive, we must have \(C_1>0\).
We now compute \(j_2(s)=\int \tfrac{1}{\sqrt{j_1(s)}}\). Now \footnotesize \(\int \tfrac{1}{\sqrt{j_1(s)}}=\int \tfrac{ds}{\sqrt{\frac{-4}{k}e^{\epsilon_0\cdot\sqrt{k}\cdot(s-2\Re[c_2])}+C_1}}\)\normalsize . Now let us make the following variable change: \footnotesize \(u=\sqrt{\tfrac{-4}{k}e^{\epsilon_0\cdot\sqrt{k}\cdot(s-2\Re[c_2])}+C_1}\) \normalsize, then we have \footnotesize \(du=\frac{\epsilon_0\sqrt{k}\,\big[\frac{-4}{k}e^{\epsilon_0\cdot\sqrt{k}\cdot(s-2\Re[c_2])}\big]}{2\sqrt{\frac{-4}{k}e^{\epsilon_0\cdot\sqrt{k}\cdot(s-2\Re[c_2])}+C_1}}ds=\frac{\epsilon_0\sqrt{k}[u^2-C_1]ds}{2u}\) \normalsize, That gives us that \(\tfrac{2}{\epsilon_0\sqrt{k}}\cdot \tfrac{du}{[u^2-C_1]}=\frac{ds}{u}\), note that \(C_1>u^2\). So we rewrite the previous equation with \(C_1-u^2\) instead of \(u^2-C_1\) in the denominator, we then put an integral on both sides: \(\tfrac{-2}{\epsilon_0\sqrt{k}}\cdot \int\tfrac{du}{[C_1-u^2]}=\int\tfrac{ds}{u}=\int \frac{1}{\sqrt{j_1(s)}}=j_2(s)\), Now since \(C_1>0\), we have this to be a standard integral and we refer to \cite{JD}, page 152, Equation 3.5.1.7 for its solution and we have \(j_2(s)=\tfrac{-2}{\epsilon_0\sqrt{kC_1}}\cdot\tanh^{-1}\big(\tfrac{u}{\sqrt{C_1}}\big)\). We now compute the inverse \(j_3 = j_2^{-1}\) of \(j_2\). To do this, set \(j_2(s) = t\). Then, \(\tfrac{-\epsilon_0 t \sqrt{k C_1}}{2} = \tanh^{-1}\big(\tfrac{u}{\sqrt{C_1}}\big)\), which implies that \(\tanh\big(\tfrac{-\epsilon_0 t \sqrt{k C_1}}{2}\big) = \tfrac{u}{\sqrt{C_1}}\). By squaring both sides, we obtain \(C_1 \tanh^2\big(\tfrac{-\epsilon_0 t \sqrt{k C_1}}{2}\big) = u^2 = \tfrac{-4}{k} e^{\epsilon_0 \sqrt{k} (s - 2\Re[c_2])} + C_1\). Hence,  
\(C_1 \big[\tanh^2\big(\tfrac{-\epsilon_0 t \sqrt{k C_1}}{2}\big) - 1\big] = \tfrac{-4}{k} e^{\epsilon_0 \sqrt{k} (s - 2\Re[c_2])}\). Recalling the hyperbolic identity \(1 - \tanh^2(-x) = 1 - \tanh^2(x) = \operatorname{sech}^2(x)\), we obtain  \(\tfrac{k C_1}{4} \operatorname{sech}^2\big(\tfrac{\epsilon_0 t \sqrt{k C_1}}{2}\big) = e^{\epsilon_0 \sqrt{k} (s - 2\Re[c_2])}\). Simplifying further, we find  \(j_3(t) = \tfrac{2}{\epsilon_0 \sqrt{k}} \ln\big[\operatorname{sech}\big(\tfrac{\epsilon_0 t \sqrt{k C_1}}{2}\big)\big] + \tfrac{1}{\epsilon_0 \sqrt{k}} \ln\big[\tfrac{k C_1}{4}\big] + 2\Re[c_2]\). Now we compute \(H(s)\) using the formula \(H(s) = j_3(s + C_2)\), which gives  us the formula for \(H(s) = \tfrac{2}{\epsilon_0 \sqrt{k}} \ln\big[\operatorname{sech}\big(\tfrac{\epsilon_0 \sqrt{k C_1}}{2} [s + C_2]\big)\big] + \tfrac{1}{\epsilon_0 \sqrt{k}} \ln\big[\tfrac{k C_1}{4}\big] + 2\Re[c_2]\). Differentiating, we obtain  \(H'(s) = -\sqrt{C_1} \tanh\big(\tfrac{\epsilon_0 \sqrt{k C_1}}{2} [s + C_2]\big)\), and  \(H''(s) = -\tfrac{\epsilon_0 C_1 \sqrt{k}}{2} \operatorname{sech}^2\big(\tfrac{\epsilon_0 \sqrt{k C_1}}{2} [s + C_2]\big)\). Using these expressions, we can now find the Second Characteristic Function of the Nontrivial Minimal Graph Transformation Problem using the formula \(h(s) = -\tfrac{H''(s)}{H'(s)}\).
\[
  h(s)=-\tfrac{-\tfrac{\epsilon_0C_1\sqrt{k}}{2} \operatorname{sech}^2\big(\tfrac{\epsilon_0\sqrt{k\cdot C_1}}{2}[s+C_2]\big)}{-\sqrt{C_1}\tanh\big(\tfrac{\epsilon_0\sqrt{k\cdot C_1}}{2}[s+C_2]\big)}=-\tfrac{\epsilon_0\sqrt{kC_1}}{2\sinh\big(\tfrac{\epsilon_0\sqrt{k\cdot C_1}}{2}[s+C_2]\big)\cosh\big(\tfrac{\epsilon_0\sqrt{k\cdot C_1}}{2}[s+C_2]\big)}
\]
Hence,  \(h(s) = -\epsilon_0 \sqrt{k C_1}\, \operatorname{csch}\big(\epsilon_0 \sqrt{k C_1}\,[s + C_2]\big)\).  
Since \(\operatorname{csch}\) is an odd function, we can take \(\epsilon_0\) outside, giving  
\(h(s) = -\sqrt{k C_1}\, \operatorname{csch}\big(\sqrt{k C_1}\,[s + C_2]\big)\). Next, we find the inverse diffeomorphism \(K\) of \(H\) to determine \(f\), using the formula \(K(u) = f\).  
We determine that \(K(t) = \tfrac{2}{\epsilon_0 \sqrt{k C_1}} \cosh^{-1}\!\Big[e^{-\tfrac{\epsilon_0 \sqrt{k}}{2} \big[t - \tfrac{2}{\epsilon_0 \sqrt{k}} \ln(\tfrac{\sqrt{k C_1}}{2}) - 2\Re[c_2]\big]}\Big] - C_2.\) We rewrite  
\(u(z, \bar{z}) = \tfrac{-2\epsilon_0}{\sqrt{k}} \Re[\ln(c_1 + z)] + 2\Re[c_2] = \tfrac{-2\epsilon_0}{\sqrt{k}} \ln(|c_1 + z|) + 2\Re[c_2].\) Substituting this into the formula \(K(u) = f\), we can determine the Minimal Graph Surface \(f(x, y)\).
\[
  f(z,\bar z)=K(u)=\tfrac{2}{\epsilon_0\sqrt{k\cdot C_1}}\cosh^{-1}\Big[\tfrac{\sqrt{kC_1}}{2}|c_1 +  z|\Big]-C_2
\] 
Now, if we set \(\lambda = \tfrac{\sqrt{k C_1}}{2}\), we obtain \(f(z, \bar{z}) = \tfrac{\epsilon_0}{\lambda} \cosh^{-1}\!\big[\lambda |c_1 + z|\big] - C_2.\) Clearly, this represents the \textbf{Catenoid} family of Minimal Surfaces. Hence, we have the following pair:
\begin{equation}
	\label{eqn: soln k>0, c_0=0}
	\boxed{f(x,y)=\frac{\epsilon_0}{\lambda}\cdot \cosh^{-1}\Big[\lambda|c_1 +  z|\Big]-C_2} \;\;\text{and}\;\;\boxed{h(s)=-2\lambda\operatorname{csch}\big(2\lambda[s+C_2]\big)}
\end{equation}
where \(z=x+iy\in\mathbb{C}\) and $\lambda=\tfrac{\sqrt{kC_1}}{2}$

\textbf{Finding Non-Trivial Minimal Graph Transformations.} We can now determine the NMG Transformation \(g\) directly using Formula~\ref{eqn: conv h to g}, which is given by  \(g = \pm \int \frac{1}{\sqrt{1 \pm C^2 e^{2\int h\, dt}}}\, dt + D.\) To proceed, we first compute \(\int h(s)\, ds\).  Recalling the standard integral \(\int \text{csch}(2ax)=\tfrac{1}{2a}\ln|\tanh(ax)|+C\), Refer \cite{JD} Page 189 Equation 7.1.1.1.5 for this integral, and noting that \(h(s)=-2\lambda\operatorname{csch}\big(2\lambda[s+C_2]\big)\), we obtain  
\(\int h(s)\, ds = -\ln|\tanh(\lambda[s+C_2])|\). Next, we evaluate \(e^{2\int h(s)\, ds} = e^{-2\ln|\tanh(\lambda[s+C_2])|} = \coth^2(\lambda[s+C_2]).\) Substituting this result into the transformation formula, we find the NMG Transformation: \(g = \pm \int \frac{1}{\sqrt{1 \pm C^2\coth^2(\lambda[s+C_2])}}\, ds + D.\)

We now make the substitution \(t = \lambda[s + C_2]\), which gives \(\tfrac{dt}{\lambda} = ds\). Substituting this into the previous expression, we obtain  \(g = \tfrac{\pm 1}{\lambda} \int \frac{1}{\sqrt{1 \pm C^2 \coth^2(t)}}\, dt + D.\)
Introducing the notation \(\epsilon_1 = \pm 1\), we can rewrite this as  \(g = \tfrac{\pm 1}{\lambda} \int \frac{\sinh(t)}{\sqrt{(1 + \epsilon_1 C^2)\cosh^2(t) - 1}}\, dt + D.\) To ensure that \(g\) remains real-valued, we must have \(1 + \epsilon_1 C^2 > 0\); we therefore choose \(C\) accordingly.  Under this condition, the expression simplifies to  \(g = \tfrac{\pm 1}{\lambda \sqrt{1 + \epsilon_1 C^2}} \int \frac{\sinh(t)}{\sqrt{\cosh^2(t) - \tfrac{1}{1 + \epsilon_1 C^2}}}\, dt + D.\) We can evaluate this integral using the substitution \(v = \cosh t\), \(dv = \sinh t\, dt\), which yields  \(g = \tfrac{\pm 1}{\lambda \sqrt{1 + \epsilon_1 C^2}} \cosh^{-1}\!\big(v \sqrt{1 + \epsilon_1 C^2}\big) + D.\) Call \(\mu=\sqrt{1 +\epsilon_1 C^2}\) as it can attain any positive value by choosing appropriate value for \(C\) Finally, substituting back \(v = \cosh t\) and then \(t = \lambda[s + C_2]\), we obtain the following expression for \(g\).
\[
 \boxed{ g = \frac{\pm1}{\lambda\mu}\cosh^{-1}\big(\mu\cdot\cosh[\lambda[s+C_2]]\big)+D}
\] 

\textbf{Finding the Transformed Minimal Surface.}  
We substitute \(f\) in place of \(s\), giving the transformed minimal surface:  
\[
  g(f) =  \frac{\pm1}{\lambda\mu}\cosh^{-1}\big(\mu\cosh[\lambda[\frac{\epsilon_0}{\lambda}\cdot \cosh^{-1}\Big[\lambda|c_1 +  z|\Big]-C_2+C_2]]\big)+D
\]
We simplify this and hence, we obtain  
\[
\boxed{g(f) = \frac{\pm1}{\lambda\mu}\cosh^{-1}\big(\lambda\mu|c_1 +  z|\big)+D.}
\]
This represents another surface within the family of Catenoids.

\textbf{Solution and its Maximal Domain of Definition:} 
The characterisation of the solution to the NMG Transformation Problem \ref{eqn: Complex Non-Trivial MGT: 2} under the Choice of Domain as in \ref{section: assumptions} and \ref{section: assumption 2} and the condition that the real constant $k>0$ and \(|c_0|=0\), is the following:

If such a Minimal Graph $f$ with a NMG Transformation $g$ exists under the above criteria, then it must be the \textbf{Catenoids} given by Formula:
\[
\boxed{f=\frac{\epsilon_0}{\lambda}\cdot \cosh^{-1}\Big[\lambda|c_1 +  z|\Big]-C_2} \;\;\text{and}\;\;\boxed{h(s)=-2\lambda\operatorname{csch}\big(2\lambda[s+C_2]\big)}
\]
where \(z=x+iy\in\mathbb{C}\) and $\lambda=\tfrac{\sqrt{kC_1}}{2}$. Moreover, the NMG Transformation $g$ for this Minimal Graph Surface \(f\) should be:
\[
\boxed{g(s)=\frac{\pm1}{\lambda\mu}\cosh^{-1}\big(\mu\cdot\cosh[\lambda[s+C_2]]\big)+D}
\]
where $\mu=\sqrt{1 +\epsilon_1 C^2}$. Finally, the Transformed Minimal Surface under the the NMG Transformation $g$  is also \textbf{Catenoid}, and is given by: 
\[
\boxed{g(f) = \frac{\pm1}{\lambda\mu}\cosh^{-1}\big(\lambda\mu|c_1 +  z|\big)+D.}
\]

 \subsection{Subcase 2: $c_0\neq0$}\label{section: c_0 neq 0 k-positive}
  \textbf{Finding solutions \((u,j)\) to the Modified Nontrivial Minimal Graph Transformation Problem \ref{eqn: Modified Complex Non-Trivial MGT: 2}:} We aim to solve the equation \(\mathcal{B}_z = \pm \sqrt{c_0 + e^{2\mathcal{B}}}\). Letting \(\pm 1 = \epsilon_0\), we have \(\mathcal{B}_z = \epsilon_0 \sqrt{c_0 + e^{2\mathcal{B}}} \Rightarrow \epsilon_0 \int \frac{d\mathcal{B}}{\sqrt{c_0 + e^{2\mathcal{B}}}} = z + c_1.\) Now, substitute \(v = \sqrt{c_0 + e^{2\mathcal{B}}}\), which implies \(dv = \frac{e^{2\mathcal{B}}\, d\mathcal{B}}{\sqrt{c_0 + e^{2\mathcal{B}}}} = \frac{v^2 - c_0}{v} d\mathcal{B}\). Hence, \(\epsilon_0 \int \frac{d\mathcal{B}}{\sqrt{c_0 + e^{2\mathcal{B}}}} =\epsilon_0\int\frac{d\mathcal{B}}{v}=\epsilon_0\int\frac{dv}{v^2-c_0}= \frac{\epsilon_0}{2\sqrt{c_0}} \ln\big(\frac{v - \sqrt{c_0}}{v + \sqrt{c_0}}\big)\). This can be rewritten as following: \(\frac{\epsilon_0}{2\sqrt{c_0}}\ln\big(\frac{v-\sqrt{c_0}}{v+\sqrt{c_0}}\big)=\frac{\epsilon_0}{\sqrt{c_0}}\cdot\frac{1}{2} \ln\big(\frac{[-v/c_0] + 1}{[-v/c_0] - 1}\big) =\frac{-\epsilon_0}{\sqrt{c_0}} \coth^{-1}\big(\frac{v}{\sqrt{c_0}}\big)\). Substituting back \(v = \sqrt{c_0 + e^{2\mathcal{B}}}\), we get \(z + c_1 = \frac{-\epsilon_0}{\sqrt{c_0}} \coth^{-1}\big(\sqrt{\frac{c_0 + e^{2\mathcal{B}}}{c_0}}\big)\). Thus, \(e^{2\mathcal{B}} = c_0 (\coth^2(\epsilon_0 \sqrt{c_0}[z + c_1]) - 1)\), now using the hyperbolic identity \(\coth^2(x)-1=\operatorname{csch}^2(x)\), we get that \(e^{2\mathcal{B}}=c_0\operatorname{csch}^2(\epsilon_0 \sqrt{c_0}[z + c_1])=c_0\operatorname{csch}^2(\sqrt{c_0}[z + c_1])\) as \(\operatorname{csch}^2(-x)=\operatorname{csch}^2(x)\), which implies \(e^{\mathcal{B}} = \epsilon_1 \sqrt{c_0}\, \operatorname{csch}(\sqrt{c_0}[z + c_1])\), where \(\epsilon_1 = \pm 1.\) We know that \(\mathcal{B}(z)\) is given by the relation \(\mathcal{B}(z) = \ln[\mathcal{A}'(z)]\), we find that \(\mathcal{B} = \ln\big(\epsilon_1 \sqrt{c_0}\operatorname{csch}(\sqrt{c_0}[z + c_1])\big) = \ln[\mathcal{A}'(z)]\), exponentiating on both sides gives that \(\mathcal{A}'(z) = \epsilon_1 \sqrt{c_0}\operatorname{csch}(\sqrt{c_0}[z + c_1])\). We integrate this using the standard complex integral identity: \(\int \operatorname{csch}(x)\, dx = \ln(\tanh(\tfrac{x}{2})) + C\) to get an expression for \(\mathcal{A}(z)\): \(\mathcal{A}(z) + c_2 = \epsilon_1 \sqrt{c_0} \int \operatorname{csch}(\sqrt{c_0}[z + c_1])\, dz = \epsilon_1\ln\big(\tanh(\tfrac{\sqrt{c_0}(z + c_1)}{2})\big)\). Now, since \(A = \tfrac{\mathcal{A}}{\sqrt{k}}\), we have \(A(z) = \tfrac{\epsilon_1}{\sqrt{k}}\ln\big(\tanh(\tfrac{\sqrt{c_0}(z + c_1)}{2})\big) + c_2\). Finally, the function \(u(z, \bar{z})\) is given by:
  $$u(z, \bar{z}) = \tfrac{\epsilon_1}{\sqrt{k}} \ln\big(\tanh(\tfrac{\sqrt{c_0}(z + c_1)}{2})\big) + c_2 + \tfrac{\epsilon_1}{\sqrt{k}} \ln\big(\tanh(\tfrac{\sqrt{\bar{c}_0}(\bar{z} + \bar{c}_1)}{2})\big) + \bar{c}_2.$$
 
  We now calculate \(u_z, u_{\bar z}, u_{zz}\), and \(u_{\bar z \bar z}\) to determine the second characteristic function \(j(u)\). We have \(u_z = \tfrac{\epsilon_1 \sqrt{c_0}}{\sqrt{k}} \operatorname{csch}(\sqrt{c_0}[z + c_1])\), \(u_{\bar z} = \tfrac{\epsilon_1 \sqrt{\bar{c}_0}}{\sqrt{k}} \operatorname{csch}(\sqrt{\bar{c}_0}[\bar{z} + \bar{c}_1])\), \(u_{zz} = -\tfrac{\epsilon_1 c_0}{\sqrt{k}} \operatorname{csch}(\sqrt{c_0}[z + c_1]) \operatorname{coth}(\sqrt{c_0}[z + c_1])\), and \(u_{\bar z \bar z} = -\tfrac{\epsilon_1 \bar{c}_0}{\sqrt{k}} \operatorname{csch}(\sqrt{\bar{c}_0}[\bar{z} + \bar{c}_1]) \operatorname{coth}(\sqrt{\bar{c}_0}[\bar{z} + \bar{c}_1])\). Using these, we compute $\frac{u_{zz}u_{\overline{z}}^2+u_{\overline{z}\overline{z}}u_z^2}{u_zu_{\overline{z}}}$ to get \(j\)
  \[
    \frac{u_{zz}u_{\overline{z}}^2+u_{\overline{z}\overline{z}}u_z^2}{u_zu_{\overline{z}}}=\tfrac{-\epsilon_1|c_0|}{\sqrt{k}}\cdot\frac{\cosh(\sqrt c_0[z+c_1])+\cosh(\sqrt{\bar c_0}\;[\bar z+\bar c_1])}{\sinh\big(\sqrt c_0[z+c_1]\big)\sinh\big(\sqrt{\bar c_0}\;[\bar z+\bar c_1]\big)}=j(u)
  \]
 
 Now let \(a = \tfrac{-\epsilon_1 |c_0|}{\sqrt{k}}\) and \(2v = \sqrt{c_0}[z + c_1]\). Using the hyperbolic identities \(\cosh(2x) + \cosh(2y) = 2\cosh(x + y)\cosh(x - y)\) and \(\sinh(2x) = 2\sinh(x)\cosh(x)\), we obtain \(j(u) = a \cdot \frac{\cosh(2v) + \cosh(2\bar{v})}{\sinh(2v)\sinh(2\bar{v})} = a \cdot \frac{\cosh(v + \bar{v})\cosh(v - \bar{v})}{2\sinh(v)\cosh(v)\sinh(\bar{v})\cosh(\bar{v})}\). Applying identities of hyperbolic functions: \(2\sinh(v)\sinh(\bar{v}) = \cosh(v + \bar{v}) - \cosh(v - \bar{v})\) and \(2\cosh(v)\cosh(\bar{v}) = \cosh(v + \bar{v}) + \cosh(v - \bar{v})\), we simplify to \(j(u) = 2a \cdot \frac{\cosh(v + \bar{v})\cosh(v - \bar{v})}{(\cosh(v + \bar{v}) + \cosh(v - \bar{v}))(\cosh(v + \bar{v}) - \cosh(v - \bar{v}))}.\)
 
 Call \(\mathbf{a} = 2a\), \(\mathbf{b} = \cosh(v + \bar{v})\), and \(\mathbf{c} = \cosh(v - \bar{v})\). Then we can write \(\tfrac{j(u)}{\mathbf{a}} = \tfrac{\mathbf{b}\mathbf{c}}{\mathbf{b}^2 - \mathbf{c}^2}\), and denote this quantity by \(U\). From the definition of \(u\), we also have \(e^{\sqrt{k}\epsilon_1[u - 2\Re(c_2)]} = \tanh(v)\tanh(\bar{v}) = \tfrac{\sinh(v)\sinh(\bar{v})}{\cosh(v)\cosh(\bar{v})} = \tfrac{\mathbf{b} - \mathbf{c}}{\mathbf{b} + \mathbf{c}} = \tfrac{(\mathbf{b} - \mathbf{c})^2}{\mathbf{b}^2 - \mathbf{c}^2}\), which we denote by \(V\). Observing that \(V + 4U = \tfrac{1}{V}\), it follows that \(U = \tfrac{1}{4}\big(\tfrac{1}{V} - V\big)\). Therefore, \(j(u) = \mathbf{a}U = 2aU = \tfrac{-2\epsilon_1|c_0|}{\sqrt{k}}U = \tfrac{-2\epsilon_1|c_0|}{4\sqrt{k}}\big(\tfrac{1}{V} - V\big)\). Hence:
 \[
   j(u)=\tfrac{-2\epsilon_1|c_0|}{\sqrt{k}}U=\tfrac{-\epsilon_1|c_0|}{2\sqrt{k}}\big(e^{-\sqrt k\epsilon_1[u-2\Re(c_2)]}-e^{\sqrt k\epsilon_1[u-2\Re(c_2)]}\big)=\tfrac{-\epsilon_1|c_0|}{\sqrt{k}}\sinh(-\sqrt k\epsilon_1[u-2\Re(c_2)])
 \]
 Since \(\sinh\) is an odd function: \(j(u)=\tfrac{|c_0|}{\sqrt{k}}\sinh(\sqrt k[u-2\Re(c_2)])\). Hence we have the following pair
 \[
 \boxed{u(z,\bar z)=\tfrac{2\epsilon_1}{\sqrt{k}} \Re\Big[\ln\big(\tanh[\tfrac{\sqrt{c_0}(z + c_1)}{2}]\big)\Big] + 2\Re(c_2)}\;\;\;\;\text{and}\;\;\;\;\boxed{j(s)=\tfrac{|c_0|}{\sqrt{k}}\sinh(\sqrt k[s-2\Re(c_2)])}
 \]
 We note that the second characteristic function $j$ only vanishes at $s=2\Re(c_2)$, hence $j$ is never zero in the intervals $J_1=(2\Re(c_2),\infty)$ and $J_2=(-\infty, 2\Re(c_2))$. Likewise, the derivatives $u_z=\tfrac{\epsilon_1 \sqrt{c_0}}{\sqrt{k}} \operatorname{csch}(\sqrt{c_0}[z + c_1])$ and $u_{\bar{z}} = \tfrac{\epsilon_1 \sqrt{\bar{c}_0}}{\sqrt{k}} \operatorname{csch}(\sqrt{\bar{c}_0}[\bar{z} + \bar{c}_1])$ are non-zero everywhere on $\mathbb{C}$, but they attain $\infty$ at $z=-c_1+\tfrac{2\pi i n}{\sqrt{c_0}}$ for all integer $n$. So for $\Omega=\mathbb{C}\setminus\{-c_1+\tfrac{2\pi i n}{\sqrt{c_0}} \mid n\in\mathbb{Z}\}$, $u$ is well defined. Consequently, for $\Omega_0 = u^{-1}(J_1)\cap\Omega$ or $\Omega_0 = u^{-1}(J_2)\cap\Omega$, we have $u_z, u_{\bar{z}} \neq 0$ on $\Omega_0$, and define $u(\Omega_0) = J$. Hence, the pair $(u, j)$ constitutes a valid solution of the Modified NMG Transformation Problem, where $u : \Omega_0 \to \mathbb{R}$ and $j : J \to \mathbb{R}$, both satisfy the required non-vanishing conditions on their respective domains. Therefore, $(u, j)$ meets the hypotheses of Theorem \ref{thm: Equivalence thm}.

 By Theorem \ref{thm: Equivalence thm}, there exists a corresponding Minimal Graph Surface $f$ and an associated First Characteristic Function $h$ that together solve the NMG Transformation Problem. Next, applying the converse part of Theorem \ref{thm: Equivalence thm}, we can construct all possible open subintervals $J_0 \subseteq J$ and corresponding pairs $(f, h)$ which solve the NMG Transformation Problem \ref{eqn: Complex Non-Trivial MGT: 2} on the domains $\Omega = u^{-1}(J_0)$ and $I = f(\Omega)$ such that the corresponding pair $(u_0, j_0)$ for the modified problem satisfies $u_0 = u|_{\Omega}$ and $j_0 = j|_{J_0}$.

 We now proceed to determine the explicit form of the minimal surface $f$ and the associated first characteristic function $h$.\\
  
 \textbf{Finding solutions \((f,h)\) to the Modified Nontrivial Minimal Graph Transformation Problem \ref{eqn: Complex Non-Trivial MGT: 2}:}  We begin by computing the function \(j_1(s) = -2\int j(s)\,ds + C_1 = -2\int \tfrac{|c_0|}{\sqrt{k}}\sinh\big(\sqrt k[s-2\Re(c_2)]\big)\,ds + C_1 = C_1 - \tfrac{2|c_0|}{k}\cosh\big(\sqrt k[s-2\Re(c_2)]\big)\). Next we calculate \(j_2(s) = \scaleobj{1}{\int} \tfrac{1}{\sqrt{j_1(s)}}\,ds = \int \tfrac{1}{\sqrt{C_1 - \frac{2|c_0|}{k}\cosh\big(\sqrt k[s-2\Re(c_2)]\big)}}\,ds\). For this to be real-valued, we require that the expression inside squareroot to be positive, that is: \(C_1 - \tfrac{2|c_0|}{k}\cosh\big(\sqrt k[s-2\Re(c_2)]\big) > 0\). Since \(\cosh(x) \ge 1\) and \(k>0\), an appropriate choice for \(C_1\) is that which obey: \(C_1 > \frac{2|c_0|}{k} > 0\); then on a suitable subinterval \(J_0\) we have \(\frac{C_1k}{2|c_0|} > \cosh\big(\sqrt k[s-2\Re(c_2)]\big) > 1\). With this choice of \(C_1\) we evaluate \(j_2(s)\) via the substitution \footnotesize\(x = \sqrt{C_1 - \frac{2|c_0|}{k}\cosh\big(\sqrt k[s-2\Re(c_2)]\big)}\)\normalsize, which gives \footnotesize\(dx = \tfrac{-\frac{2|c_0|}{\sqrt k}\sinh\big(\sqrt k[s-2\Re(c_2)]\big)}{2\sqrt{C_1 - \frac{2|c_0|}{k}\cosh\big(\sqrt k[s-2\Re(c_2)]\big)}}\,ds\)\normalsize. Hence \footnotesize\(2x\,dx = -\tfrac{2|c_0|}{\sqrt k}\sinh\big(\sqrt k[s-2\Re(c_2)]\big)ds = -\frac{2\epsilon_2|c_0|}{\sqrt k}\sqrt{\cosh^2\big(\sqrt k[s-2\Re(c_2)]\big)-1} \;ds\)\normalsize. Observing that \(\cosh^2\big(\sqrt k[s-2\Re(c_2)]\big) = \frac{k^2(C_1-x^2)^2}{4|c_0|^2}\), we obtain \(2x\,dx = -\frac{2\epsilon_2|c_0|}{\sqrt k}\sqrt{\frac{k^2(C_1-x^2)^2}{4|c_0|^2}-1}\;ds\). This implies that:
 \[
   j_2(s)=\int \frac{ds}{x}=- \int \tfrac{dx}{\frac{\epsilon_2|c_0|}{\sqrt k}\sqrt{\frac{k^2(C_1-x^2)^2}{4|c_0|^2}-1}}=-\tfrac{2\epsilon_2}{\sqrt k} \int \tfrac{dx}{\sqrt{(C_1-x^2)^2-\big[\frac{2|c_0|}{k}\big]^2}}
 \]
 Hence we obtain \(j_2(s)=-\frac{2\epsilon_2}{\sqrt k} \scaleobj{1.3}{\int} \tfrac{dx}{\sqrt{\left(C_1-\frac{2|c_0|}{k}-x^2\right)\left(C_1+\frac{2|c_0|}{k}-x^2\right)}}\).  
Define  \(\alpha= \sqrt{C_1-\frac{2|c_0|}{k}}\)  and \(\beta=\sqrt{C_1+\frac{2|c_0|}{k}}\), noting that \(0<\alpha<\beta\); then  
\(j_2(s)=-\tfrac{2\epsilon_2}{\alpha \beta \sqrt k} \scaleobj{1.3}{\int} \tfrac{dx}{\sqrt{\left[1-(\frac{x}{\alpha})^2\right]\left[1-(\frac{x}{\beta})^2\right]}}.\) 
Set \(y=\frac{x}{\alpha}\) and \(\gamma=\frac{\alpha}{\beta}\); clearly \(0<\gamma<1\) and \(dx=\alpha\,dy\).  
Moreover \(|y|<1\), since \(y^2=\frac{x^2}{\alpha^2}= \frac{C_1-\frac{2|c_0|}{k}\cosh\!\big(\sqrt k[s-2\Re(c_2)]\big)}{C_1-\frac{2|c_0|}{k}}<1\). We have: 
 \[
   j_2(s)=-\tfrac{2\epsilon_2}{\beta\sqrt k} \int \tfrac{dy}{\sqrt{\left(1-y^2\right)\left(1-\gamma^2y^2\right)}}   \qquad\text{$0<\gamma<1$ and $|y|\leq1$} 
 \]
 The integral on the right‑hand side is the \textbf{Elliptic Integral of the First Kind}; the corresponding inverse Jacobi elliptic functions can be found in \cite{La} page~50, Equation 3.1.2, yielding \(j_2(s)=-\tfrac{2\epsilon_2}{\beta\sqrt k}\,\operatorname{sn}^{-1}(y,\gamma).\) Now substituting the value for \(y\) we get: \footnotesize \(j_2(s)=-\tfrac{2\epsilon_2}{\beta\sqrt k}\,\operatorname{sn}^{-1} \big(\tfrac{1}{\alpha}\sqrt{C_1 - \tfrac{2|c_0|}{k}\cosh\big(\sqrt k[s-2\Re(c_2)]\big)},\gamma \big)\) \normalsize. 
 
 Because of the inherent difficulty in dealing with elliptic functions, their inverses and derivatives, we now make a general remark concerning Theorem \ref{thm: Equivalence thm} that will help us bypass the calculation of the diffeomorphism \(H(s)\) and its derivatives, allowing us instead to compute \(K=H^{-1}\) and \(H'\) directly for this case: \(c_0\neq 0\) as well as next case where \(k<0\).  We stress that this bypass is not our preferred approach, since the diffeomorphism \(H\) is a crucial part of the machinery for solving the NMG Transformation Problem, but here we avoid that computation due to its pointless complexity.
\begin{remark}\label{remark: bypassing H}
	We have defined \(j_2(s)=\int\frac{1}{\sqrt{j_1(s)}}\,ds\), \(j_3=j_2^{-1}\), \(H(s)=j_3(s+C_2)\), and \(K=H^{-1}\).  From these definitions we get the following directly:
	
	1. \(H'(s)=\tfrac{1}{{j'_2}[j_3(s+C_2)]}=\tfrac{1}{j'_2[H(s)]}=\sqrt{j_1(H(s))}\). 
	
	2. \(K(s)=j_2(s)-C_2\).
	
	3. \(j_2(H(s))=j_2[j_3(s+C_2)]=s+C_2\).
\end{remark} 
We now apply these remarks to find the solution \((f,h)\) to the NMG Transformation Problem \ref{eqn: Complex Non-Trivial MGT: 2}. We have  \footnotesize\(s+C_2=j_2(H(s))=-\tfrac{2\epsilon_2}{\beta\sqrt k}\,\operatorname{sn}^{-1} \big(\tfrac{1}{\alpha}\sqrt{C_1 - \tfrac{2|c_0|}{k}\cosh\big(\sqrt k[H(s)-2\Re(c_2)]\big)},\gamma \big)\)\normalsize. Now from the definition of \(j_1\) we get the following: \(s+C_2=-\tfrac{2\epsilon_2}{\beta\sqrt k}\,\operatorname{sn}^{-1} \big(\tfrac{1}{\alpha}\sqrt{j_1(H(s))},\gamma \big)=-\tfrac{2\epsilon_2}{\beta\sqrt k}\,\operatorname{sn}^{-1} \big(\tfrac{1}{\alpha} H'(s),\gamma \big)\). Therefore we get \(H'(s)=\alpha\operatorname{sn}\big[\tfrac{-\epsilon_2\beta\sqrt k}{2}(s+C_2),\gamma\big]\). Now using the fact \(\operatorname{sn}(-x,k)=-\operatorname{sn}(x,k)\) (see \cite{JD}, page 247, Equation 12.2.2.1.2) we obtain \(H'(s)=-\alpha\epsilon_2\operatorname{sn}\big[\tfrac{\beta\sqrt k}{2}(s+C_2),\gamma\big]\). We now calculate \(H''(s)\);(see \cite{JD}, page 249, Equation 12.3.1.1.2 for \(\tfrac{d}{du}\operatorname{sn}(u,k)\)), we have \footnotesize\(H''(s)=-\epsilon_2\tfrac{\alpha\beta\sqrt k}{2}\operatorname{cn}\big[\tfrac{\beta\sqrt k}{2}(s+C_2),\gamma\big]\operatorname{dn}\big[\tfrac{\beta\sqrt k}{2}(s+C_2),\gamma\big]\)\normalsize. Finally we compute the First Characteristic Function \(h(s)-\tfrac{H''(s)}{H'(s)}\) as follows:

\noindent
 \footnotesize\(
   h(s)=\tfrac{\epsilon_2\tfrac{\alpha\beta\sqrt k}{2}\operatorname{cn}\big[\tfrac{\beta\sqrt k}{2}(s+C_2),\gamma\big]\operatorname{dn}\big[\tfrac{\beta\sqrt k}{2}(s+C_2),\gamma\big]}{-\alpha\epsilon_2\operatorname{sn}\big[\tfrac{\beta\sqrt k}{2}(s+C_2),\gamma\big]}=-\tfrac{\beta\sqrt k}{2}\operatorname{cs}\big[\tfrac{\beta\sqrt k}{2}(s+C_2),\gamma\big]\operatorname{dn}\big[\tfrac{\beta\sqrt k}{2}(s+C_2),\gamma\big]
 \)\normalsize  
 
  Now let us calculate \(f=K(u)\); by Remark \ref{remark: bypassing H} we have \(K(s)=j_2(s)-C_2\). Substituting the expressions for \(u\) and \(j_2(s)\) we obtain:
  \footnotesize\[
    f(z,\bar z)=j_2(u)-C_2=-\tfrac{2\epsilon_2}{\beta\sqrt k}\,\operatorname{sn}^{-1} \Big(\tfrac{1}{\alpha}\sqrt{C_1 - \tfrac{2|c_0|}{k}\cosh\big(2\epsilon_1\Re\big[\ln\big[\tanh[\tfrac{\sqrt{c_0}(z + c_1)}{2}]\big]\big]\big)},\gamma\Big)-C_2
  \]\normalsize
  We use the fact that \(\cosh\) is an even function to get the following expression for \(f(z,\bar z)\): \footnotesize\(f(z, \bar z)=-\tfrac{2\epsilon_2}{\beta\sqrt k}\operatorname{sn}^{-1}\Big(\tfrac{1}{\alpha}\sqrt{C_1 - \tfrac{2|c_0|}{k}\cosh\big(\ln\big[\tanh[\tfrac{\sqrt{c_0}(z + c_1)}{2}]\tanh[\tfrac{\sqrt{\bar c_0}(\bar z + \bar c_1)}{2}]\big]\big)},\gamma\Big)-C_2\)\normalsize
  
   Hence we have the following pair:
   \footnotesize
  \begin{equation}
  	\label{eqn: soln k>0, c_0 neq 0, f}
  	\boxed{f(z, \bar z)=-\tfrac{2\epsilon_2}{\beta\sqrt k}\operatorname{sn}^{-1}\Big(\tfrac{1}{\alpha}\sqrt{C_1 - \tfrac{2|c_0|}{k}\cosh\big(2\ln|\tanh[\tfrac{\sqrt{c_0}(z + c_1)}{2}]|\big)},\gamma\Big)-C_2}
  \end{equation}
  \begin{equation}
  	\label{eqn: soln k>0, c_0 neq 0, h}
  	\boxed{h(s)=-\tfrac{\beta\sqrt k}{2}\operatorname{cs}\big[\tfrac{\beta\sqrt k}{2}(s+C_2),\gamma\big]\operatorname{dn}\big[\tfrac{\beta\sqrt k}{2}(s+C_2),\gamma\big]}
  \end{equation}
  where \footnotesize \(\alpha= \sqrt{C_1-\frac{2|c_0|}{k}}\), \(\beta=\sqrt{C_1+\frac{2|c_0|}{k}}\)\normalsize and \(\gamma=\tfrac{\alpha}{\beta}\)
  \begin{figure}[htbp]
  	\centering
  	\begin{subfigure}{0.32\textwidth}
  		\centering
  		\includegraphics[width=\linewidth]{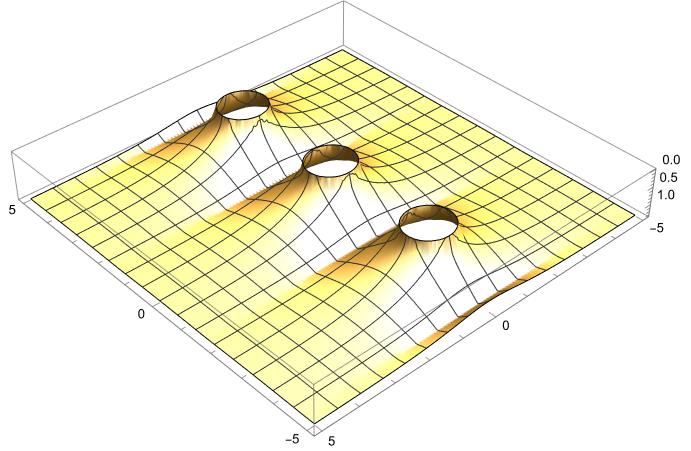}
  		\caption{ \footnotesize Top view of one layer \normalsize}
  	\end{subfigure}
  	\hfill
  	\begin{subfigure}{0.32\textwidth}
  		\centering
  		\includegraphics[width=\linewidth]{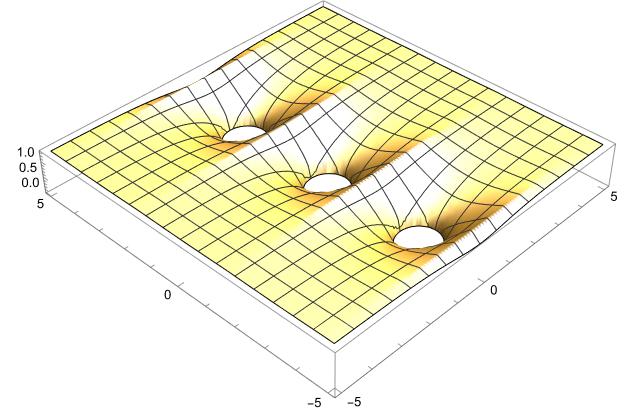}
  		\caption{\footnotesize Bottom view of one layer \normalsize}
  	\end{subfigure}
    \hfill
    \begin{subfigure}{0.32\textwidth}
    	\centering
    	\includegraphics[width=\linewidth]{Pillarsparametric.jpeg}
    	\caption{\footnotesize Impicit Surface Plot \normalsize}
    \end{subfigure}
  \caption{\footnotesize The images above depict the minimal surface \( f(z,\bar z) \) given in equation \ref{eqn: soln k>0, c_0 neq 0, f}. They were generated with the parameter values	\( k=1 \), \( c_0=1 \), \( c_1=0 \), \( C_1=10 \), \( C_2=0 \) and \( \epsilon_2=-1 \).\normalsize}
  \label{pic: Pillars}
  \end{figure}

  \textbf{Finding Non-Trivial Minimal Graph Transformations: } The NMG transformation \(g\) can be determined directly via Formula~\ref{eqn: conv h to g}, which gives \(g = \pm \int \frac{1}{\sqrt{1 \pm C^2 e^{2\int h\, dt}}}\, dt + D\); however, due to the complexity of integrating expressions involving elliptic functions, we introduce a general remark that allows us to bypass the explicit evaluation of \(\int h\), this will also be very useful for the case \(k<0\).
  \begin{remark}\label{remark: bypassing int h}
  	We know that \(h(s)=\frac{-H''(s)}{H'(s)}=-\frac{d}{dt}[\ln H'(s)]\), so \(\int h\,dt=-\ln H'(s)\); this yields \(e^{2\int h\,dt}=\big[\frac{1}{H'(s)}\big]^2\), and with this the formula for the Minimal Graph Transformation simplifies to \footnotesize\(g =  \pm \displaystyle \int \tfrac{1}{\sqrt{1\; \pm\; \left(\frac{C}{H'(s)}\right)^2}}\, dt + D=\pm \displaystyle \int \tfrac{H'(s)}{\sqrt{[H'(s)]^2\; \pm\; C^2}}\, dt + D\)\normalsize.
  \end{remark}
  We now calculate the Minimal Graph Transformation for our case using the remark and the fact that \(H'(s)=\alpha\operatorname{sn}\big[\tfrac{-\epsilon_2\beta\sqrt k}{2}(s+C_2),\gamma\big]\)
  \footnotesize\[
    g=\pm \int \tfrac{\alpha\operatorname{sn}\big[\tfrac{-\epsilon_2\beta\sqrt k}{2}(s+C_2),\gamma\big]}{\sqrt{\alpha^2\operatorname{sn}^2\big[\tfrac{-\epsilon_2\beta\sqrt k}{2}(s+C_2),\gamma\big]\; \pm\; C^2}}\, dt + D=\pm \int \tfrac{\alpha\operatorname{sn}\big[\tfrac{-\epsilon_2\beta\sqrt k}{2}(s+C_2),\gamma\big]}{\sqrt{1\; \pm\; C^2-\alpha^2\operatorname{cn}^2\big[\tfrac{-\epsilon_2\beta\sqrt k}{2}(s+C_2),\gamma\big]}}\, dt + D
  \]\normalsize

  Since \(\alpha>0\), we factor \(\alpha\) out of the numerator and denominator; noting that \(\operatorname{sn}\) and \(\operatorname{cn}\) are respectively odd and even functions (see e.g. \cite{JD}, page 247, Equation: 12.2.2.1.2 and 12.2.2.1.3), and because of \(\pm1\) infront of the integral, we may remove \(-\epsilon_2\) from the arguments of \(\operatorname{cn}\) and \(\operatorname{sn}\), yielding \(g=\pm \displaystyle \int \tfrac{\operatorname{sn}\!\big[\frac{\beta\sqrt k}{2}(s+C_2),\gamma\big]}{\sqrt{\big[\frac{1\pm C^2}{\alpha^2}\big]-\operatorname{cn}^2\!\big[\frac{\beta\sqrt k}{2}(s+C_2),\gamma\big]}}\, dt + D\). For this expression to be well‑defined we require \(\frac{1\pm C^2}{\alpha^2}>0\); denote this quantity by \(C_3^2\). Recalling that \(\operatorname{dn}(u,\gamma)=\sqrt{1-\gamma^2+\gamma^2\operatorname{cn}^2(u,\gamma)}\) (see \cite{JD}, page: 248, Equation: 12.2.2.2.3) and that \(\frac{d}{du}\big[\operatorname{cn}(u,\gamma)\big]=-\operatorname{sn}(u,\gamma)\operatorname{dn}(u,\gamma)\) (see \cite{JD}, page: 249, Equation: 12.3.1.1.2), we proceed as follows:
 \footnotesize\[
  g=\pm \tfrac{2}{\beta\sqrt k}\displaystyle \int \tfrac{-\frac{\beta\sqrt k}{2}\operatorname{sn}\!\big[\frac{\beta\sqrt k}{2}(s+C_2),\gamma\big]\operatorname{dn}\!\big[\frac{\beta\sqrt k}{2}(s+C_2),\gamma\big]}{\sqrt{\big[C_3^2-\operatorname{cn}^2\!\big[\frac{\beta\sqrt k}{2}(s+C_2),\gamma\big]\big]\big[1-\gamma^2+\gamma^2\operatorname{cn}^2\!\big[\frac{\beta\sqrt k}{2}(s+C_2),\gamma\big]\big]}}\, dt + D   
 \]\normalsize
 Let \(\Gamma^2=\frac{1-\gamma^2}{\gamma^2}>0\) and set \(v=\operatorname{cn}\!\big[\tfrac{\beta\sqrt k}{2}(s+C_2),\gamma\big]\); then the integral becomes \(g=\frac{\pm2}{\beta\sqrt{k}\gamma}\int\frac{dv}{\sqrt{[C_3^2-v^2]\cdot[\Gamma^2+v^2]}}+D=\frac{\pm2}{\alpha\sqrt{k}}\int\frac{dv}{\sqrt{[C_3^2-v^2]\cdot[\Gamma^2+v^2]}}+D.\) This is an elliptic integral whose corresponding inverse Jacobi elliptic functions can be found in \cite{La} page~50, Equation~3.2.7; we thus obtain \(g=\frac{\pm2}{\alpha\sqrt{k(C_3^2+\Gamma^2)}}\operatorname{sd}^{-1}\!\big[\tfrac{\sqrt{(C_3^2+\Gamma^2)}\;v}{\Gamma C_3},\tfrac{C_3}{\sqrt{C_3^2+\Gamma^2}}\big]+D\). Substituting back \(v\) yields:
 \[
   \boxed{g=\tfrac{\pm2}{\alpha\sqrt{k(C_3^2+\Gamma^2)}}\operatorname{sd}^{-1}\Big[\tfrac{\sqrt{(C_3^2+\Gamma^2)}\operatorname{cn}\big[\frac{\beta\sqrt k}{2}(s+C_2),\;\gamma\big]}{\Gamma C_3},\tfrac{C_3}{\sqrt{C_3^2+\Gamma^2}}\Big]+D}
 \]
 
 \textbf{Finding the Transformed Minimal Surface.} We now compute the transformed minimal surface given by \(g\circ f\); before doing so we evaluate \(\operatorname{cn}\big[\tfrac{\beta\sqrt k}{2}(f+C_2),\gamma\big]\), recalling that \(\operatorname{cn}\) is an even function, which yields:
 \footnotesize\[
   \operatorname{cn}\big[\tfrac{\beta\sqrt k}{2}(f+C_2),\;\gamma\big]=\operatorname{cn}\Big[\operatorname{sn}^{-1}\big[\tfrac{1}{\alpha}\sqrt{C_1-\tfrac{2|c_0|}{k}\cosh\big[2\ln|\tanh[\tfrac{\sqrt{c_0}(z + c_1)}{2}]|\big]},\gamma\big]\Big]
 \]\normalsize
 Now \(\operatorname{sn}^{-1}(x,k)=\operatorname{cn}^{-1}\big(\sqrt{1-x^2},k\big)\) (see \cite{BF} Page: 31, Equation 131.01); this implies that \(\operatorname{cn}\!\big[\operatorname{sn}^{-1}(x,k)\big]=\sqrt{1-x^2}\). Hence:
 \footnotesize\[
   \operatorname{cn}\big[\tfrac{\beta\sqrt k}{2}(f+C_2),\;\gamma\big]=\tfrac{1}{\alpha}\sqrt{\tfrac{2|c_0|}{k}\cosh\big[2\ln|\tanh[\tfrac{\sqrt{c_0}(z + c_1)}{2}]|\big]-\tfrac{2|c_0|}{k}}
 \]\normalsize
 Hence: \footnotesize\(g\circ f=\tfrac{\pm2}{\alpha\sqrt{k(C_3^2+\Gamma^2)}}\operatorname{sd}^{-1}\Big[\tfrac{\sqrt{(C_3^2+\Gamma^2)}\sqrt{\frac{2|c_0|}{k}\cosh\big[2\ln|\tanh[\frac{\sqrt{c_0}(z + c_1)}{2}]|\big]-\frac{2|c_0|}{k}}}{\alpha\Gamma C_3},\tfrac{C_3}{\sqrt{C_3^2+\Gamma^2}}\Big]+D\)\normalsize

 We use the identity \(\operatorname{sd}^{-1}\!\big([x/k'],k\big)+\operatorname{cn}^{-1}(x,k)=\operatorname{sn}^{-1}(1,k)\) (see \cite{La} Page: 54, Equation: 3.2.18). Setting \(k=\tfrac{C_3}{\sqrt{C_3^2+\Gamma^2}}\) and consequently \(k'=\tfrac{\Gamma}{\sqrt{C_3^2+\Gamma^2}}\), we obtain the formula for the transformed minimal surface in this case:
 \footnotesize\[
  \boxed{ g\circ f=\tfrac{\pm2}{\alpha\sqrt{k(C_3^2+\Gamma^2)}}\operatorname{cn}^{-1}\Big[\tfrac{\sqrt{\frac{2|c_0|}{k}\cosh\big[2\ln|\tanh[\frac{\sqrt{c_0}(z + c_1)}{2}]|\big]-\frac{2|c_0|}{k}}}{\alpha C_3},\tfrac{C_3}{\sqrt{C_3^2+\Gamma^2}}\Big]+D+\operatorname{sn}^{-1}(1,k)}
 \]\normalsize

 \textbf{Proving that \(f\) and \(g\circ f\) belong to the same class of minimal surfaces.} First, set \(n^2=\tfrac{4|c_0|}{k}>0\); then one readily verifies that \(\beta=\sqrt{\alpha^2+n^2}\) and \(\gamma=\tfrac{\alpha}{\beta}=\tfrac{\alpha}{\sqrt{\alpha^2+n^2}}\). Now define \(a=\alpha C_3\) and \(b=\alpha\sqrt{C_3^2+\Gamma^2}\); since \(\alpha^2\Gamma^2=\alpha^2\tfrac{1-\gamma^2}{\gamma^2}=\beta^2-\alpha^2=\tfrac{4|c_0|}{k}=n^2\), which implies that \(k=\tfrac{4|c_0|}{n^2}\) and we obtain \(b=\sqrt{a^2+n^2}\) and \(\tfrac{C_3}{\sqrt{C_3^2+\Gamma^2}}=\tfrac{\alpha C_3}{\sqrt{\alpha^2 C_3^2+\alpha^2\Gamma^2}}=\tfrac{a}{b}=\tfrac{a}{\sqrt{a^2+n^2}}\). Denoting the constant \(D+\operatorname{sn}^{-1}(1,k)\) by \(C_4\), we arrive at a formula for \(g\circ f\).
 \footnotesize\[
   g\circ f=\tfrac{\pm n}{\sqrt{|c_0|(a^2+n^2)}}\operatorname{cn}^{-1}\Big[\tfrac{1}{a}\sqrt{n^2\big[\cosh\big[2\ln|\tanh(\tfrac{\sqrt{c_0}(z + c_1)}{2})|\big]-1\big]},\tfrac{a}{\sqrt{a^2+n^2}}\Big]+C_4
 \]\normalsize
 Using the identity \(\operatorname{sn}^{-1}(x,k)=\operatorname{cn}^{-1}\big(\sqrt{1-x^2},k\big)\) applied to Equation~\ref{eqn: soln k>0, c_0 neq 0, f} we obtain: \footnotesize\(f=\tfrac{\pm2}{\beta\sqrt k}\operatorname{cn}^{-1}\Big(\tfrac{1}{\alpha}\sqrt{\tfrac{2|c_0|}{k}\cosh\big(2\ln|\tanh[\tfrac{\sqrt{c_0}(z + c_1)}{2}]|\big)-\tfrac{2|c_0|}{k}},\gamma\Big)-C_2\)\normalsize.  Making the appropriate substitutions, we obtain:
 \footnotesize\[
   f=\tfrac{\pm n}{\sqrt{|c_0|(\alpha^2+n^2)}}\operatorname{cn}^{-1}\Big[\tfrac{1}{\alpha}\sqrt{n^2\big[\cosh\big[2\ln|\tanh(\tfrac{\sqrt{c_0}(z + c_1)}{2})|\big]-1\big]},\tfrac{\alpha}{\sqrt{\alpha^2+n^2}}\Big]-C_2
 \]\normalsize
  We see that both \(f\) and \(g\circ f\) have the same functional form and therefore belong to the same family of minimal surfaces. Finally, recall that \(\gamma=\tfrac{\alpha}{\sqrt{\alpha^2+n^2}}\) is the modulus of the elliptic function \(\operatorname{cn}\); denote the corresponding conjugate modulus by \(\gamma'=\tfrac{n}{\sqrt{\alpha^2+n^2}}\). Similarly, for the transformed minimal graph surface we set \(\Lambda=\tfrac{a}{\sqrt{a^2+n^2}}\) as the modulus of the elliptic function \(\operatorname{cn}\) that defines the transformed surface and \(\Lambda'=\tfrac{n}{\sqrt{a^2+n^2}}\) as its conjugate modulus. We also see that \(\tfrac{n}{\alpha}=\tfrac{\gamma'}{\gamma}\) and \(\tfrac{n}{a}=\tfrac{\Lambda'}{\Lambda}\). Using these notations we obtain the final expression for the solution in this case: \(k>0\)
  
  \textbf{Solution:}  
  
  The characterisation of the solution to the NMG Transformation Problem \ref{eqn: Complex Non-Trivial MGT: 2} under the Choice of Domain as in \ref{section: assumptions} and \ref{section: assumption 2} and under the condition that the real constant \(k>0\) and \(c_0\neq0\), is as follows:  
  
  If a Minimal Graph \(f\) with an NMG Transformation \(g\) exists under the above criteria, then it must be given by the formula:
  \[
    \boxed{f=\tfrac{\pm \gamma'}{\sqrt{|c_0|}}\operatorname{cn}^{-1}\Big[\tfrac{\gamma'}{\gamma}\sqrt{\cosh\big[2\ln|\tanh(\tfrac{\sqrt{c_0}(z + c_1)}{2})|\big]-1}\;,\;\gamma\Big]-C_2}
  \]
  where \(z=x+iy\in\mathbb{C}\). Note that the constants \(\gamma\) can attain any positive value between \(0\) and \(1\) by their definition, so we treat them as arbitrary parameters which act as modulus for the elliptic function. Next substituting the constants \(\gamma, \Lambda\), appropriately into the formula for the NMG transformation \(g\) yields:
  \[
  \boxed{g=\tfrac{\pm \Lambda'}{\sqrt{|c_0|}}\operatorname{sd}^{-1}\Big[\tfrac{\gamma}{\gamma'\Lambda}\operatorname{cn}\big[\tfrac{\sqrt{|c_0|}}{\gamma'}(s+C_2),\;\gamma\big],\Lambda\Big]+D.}
  \]
  Finally, the transformed minimal graph surface under the NMG transformation \(g\) is also a minimal graph surface of the same type as \(f\) and is given by:
  \[
   \boxed{ g\circ f=\tfrac{\pm \Lambda'}{\sqrt{|c_0|}}\operatorname{cn}^{-1}\Big[\tfrac{\Lambda'}{\Lambda}\sqrt{\cosh\big[2\ln|\tanh(\tfrac{\sqrt{c_0}(z + c_1)}{2})|\big]-1}\;,\;\Lambda\Big]+C_4}
  \]
  Again, the the constants \(\Lambda\) can attain any positive value between \(0\) and \(1\) by their definition, so we treat them as arbitrary parameters which act as modulus for the elliptic function.
  \begin{remark}
  	The minimal graph surface presented above constitutes, to the best of the author's knowledge, a new class of minimal surfaces. A qualitative study of the geometry of the minimal graph surface around the singular points at which the conditions on the choice of domain in \ref{section: assumptions} and \ref{section: assumption 2} fail will be presented in a later section for this as well as other minimal graph surfaces presented in this paper.
  \end{remark}
  \section{Solution to the Weakened Equation for \(k<0\)}
  In this section we solve Equation \ref{eqn: weakened eqn 2 k<0}: \([\ln(\mathcal{A}'(z))]_{zz} = -[\mathcal{A}'(z)]^2\) analytically, obtaining all minimal surfaces with non‑zero first characteristic function \(h(s)\) and negative constant \(k\), together with their associated NMG transformations \(g\) and corresponding transformed minimal surfaces. Substituting \(\mathcal{B}(z)=\ln[\mathcal{A}'(z)]\) gives \(\mathcal{B}_{zz}=-e^{2\mathcal{B}}\). Multiplying both sides by \(2\mathcal{B}_z\) leads to \(2\mathcal{B}_z\mathcal{B}_{zz}=[\mathcal{B}_z^2]_z=-[e^{2\mathcal{B}}]_z=-e^{2\mathcal{B}}2\mathcal{B}_z\); hence:
  \[
  \mathcal{B}_z^2=-e^{2\mathcal{B}}+c_0 \qquad\text{for some complex constant} c_0
  \]
  then:
  \[
   \mathcal{B}_z=\pm \sqrt{c_0-e^{2\mathcal{B}}}
  \]
  \textbf{This leads us to 2 cases}: $c_0=0$ and $c_0\neq0$
  
  \subsection{Subcase 1: $c_0=0$}
   We set the complex constant \(c_0 = 0\). Then \(\mathcal{B}_z = \pm \sqrt{-e^{2\mathcal{B}}} = \pm i e^{\mathcal{B}}\). Denoting \(\pm 1 = \epsilon_0\), we obtain \(dz = -\epsilon_0 i e^{-\mathcal{B}} d\mathcal{B}\); integrating gives \(\epsilon_0 i e^{-\mathcal{B}} = z + c_1\) for some complex constant \(c_1\). Consequently \(-\mathcal{B} = \ln[-\epsilon_0 i (z + c_1)]\), so \(\mathcal{B}=\ln\!\big[\tfrac{-\epsilon_0 i}{c_1 + z}\big]\). Since \(\mathcal{B} = \ln[\mathcal{A}'(z)]\), we have \(\ln[\mathcal{A}'(z)]=\ln\!\big[\tfrac{-\epsilon_0 i}{c_1 + z}\big]\); exponentiating yields \(\mathcal{A}'(z)=\tfrac{-\epsilon_0 i}{c_1 + z}\). Integrating gives \(\mathcal{A}(z)=-\epsilon_0 i\ln(c_1 + z)+c_2\) with constants \(c_1,c_2\in\mathbb{C}\). Recalling that \(A(z)=\frac{\mathcal{A}}{\sqrt{k}}\) and absorbing the factor \(\frac{1}{\sqrt{k}}\) into the constant \(c_2\), we write \(A(z)= \frac{-\epsilon_0 i}{\sqrt{|k|}}\ln(c_1 + z)+c_2\). By definition \(u(z,\bar{z}) = A(z) + \overline{A(z)}\); therefore,
  \[
  u(z, \bar{z}) = \tfrac{-\epsilon_0 i}{\sqrt{|k|}} \ln(z + c_1) + c_2 + \tfrac{ \epsilon_0 i}{\sqrt{|k|}} \ln(\bar{z} + \bar{c}_1) + \bar{c}_2 = \tfrac{-2\epsilon_0 i}{\sqrt{|k|}} \Im[\ln(z + c_1)] + 2\Re[c_2].
  \]
  Thus we obtain the formula for \(u\): \(\boxed{u(z,\bar z) = \tfrac{2\epsilon_0}{\sqrt{|k|}} \Im[\ln(c_1 + z)] + 2\Re[c_2]}\) (We used the fact: \(z-\bar z=2i\Im(z)\)) for some complex constants \(c_1\) and \(c_2\), where \(\epsilon_0 = \pm 1\). Now we compute the second characteristic function \(j\). We have \(\frac{u_z^2u_{\bar z\bar z}+u_{\bar z}^2u_{zz}}{u_z u_{\bar z}} = j(u)\). The derivatives are \(u_z = \tfrac{-\epsilon_0 i}{(c_1 + z)\sqrt{|k|}},\; u_{\bar z} = \tfrac{\epsilon_0 i}{(\bar c_1 + \bar z)\sqrt{|k|}},\; u_{zz} = \tfrac{\epsilon_0 i}{(c_1 + z)^2\sqrt{|k|}},\; u_{\bar z\bar z} = \tfrac{-\epsilon_0 i}{(\bar c_1 + \bar z)^2\sqrt{|k|}}\). Substituting these into the expression for \(j(u)\) yields:
  \[
  \tfrac{u_z^2u_{\bar z\bar z}+u_{\bar z}^2u_{zz}}{u_z u_{\bar z}} = \big[\tfrac{0}{(\bar c_1 + \bar z)^2 (c_1 + z)^2 |k| \sqrt{|k|}}\big]/\big[\tfrac{1}{(\bar c_1 + \bar z)(c_1 + z) |k|}\big]= 0 = j(u).
  \]
  Clearly this does not provide a solution to System 2 of Theorem \ref{thm: Equivalence thm} because it violates the non‑vanishing requirement for the function \(j\); therefore we proceed to the next case, \(c_0\neq 0\).
  
  \subsection{Subcase 2: $c_0\neq0$}\label{section: c_0 neq 0 k-negative}
  
  \textbf{Finding solutions \((u,j)\) to the Modified Nontrivial Minimal Graph Transformation Problem \ref{eqn: Modified Complex Non-Trivial MGT: 2}:} First, we solve the equation \(\mathcal{B}_z=\pm \sqrt{c_0-e^{2\mathcal{B}}}\), set \(\epsilon_0=\pm1\), so \(\epsilon_0 \int\tfrac{d\mathcal{B}}{\sqrt{c_0-e^{2\mathcal{B}}}}=z+c_1\). Substituting \(v=\sqrt{c_0-e^{2\mathcal{B}}}\) gives \(dv=\tfrac{-e^{2\mathcal{B}}d\mathcal{B}}{\sqrt{c_0-e^{2\mathcal{B}}}}=\tfrac{v^2-c_0}{v}d\mathcal{B}\); hence \(\epsilon_0 \int\tfrac{d\mathcal{B}}{\sqrt{c_0-e^{2\mathcal{B}}}}=\epsilon_0\int \tfrac{d\mathcal{B}}{v}=\epsilon_0 \int \tfrac{dv}{v^2-c_0}\). This integral was already computed in subsection \ref{section: c_0 neq 0 k-positive} and equals \(\tfrac{-\epsilon_0}{\sqrt c_0}\coth^{-1}\!\big(\tfrac{v}{\sqrt c_0}\big)\). Substituting back \(v=\sqrt{c_0-e^{2\mathcal{B}}}\) yields \footnotesize\(z+c_1=\tfrac{-\epsilon_0}{\sqrt c_0}\coth^{-1}\!\big[\frac{\sqrt{c_0-e^{2\mathcal{B}}}}{\sqrt c_0}\big]\)\normalsize, from which we obtain \footnotesize\(-e^{2\mathcal{B}}=c_0\big[\coth^2\!\big(-\epsilon_0\sqrt c_0[z+c_1]\big)-1\big]=c_0\big[\coth^2\!\big(\sqrt c_0[z+c_1]\big)-1\big]=c_0\operatorname{csch}^2\!\big(\sqrt c_0[z+c_1]\big)\)\normalsize. Consequently, \(e^{\mathcal{B}}=i\epsilon_1\sqrt c_0\operatorname{csch}\!\big(\sqrt c_0[z+c_1]\big)\) with \(\epsilon_1=\pm1\). We have the reltion: \(\mathcal{B}(z)=\ln[\mathcal{A}'(z)]\), we exponentiate on both sides to get \(e^\mathcal{B}=\mathcal{A}'\), so \(\mathcal{A}'(z)=i\epsilon_1\sqrt c_0\operatorname{csch}\!\big(\epsilon_0\sqrt c_0[z+c_1]\big)\). The integration of \(\mathcal{A}'(z)\) to obtain \(\mathcal{A}(z)\) proceeds exactly as in Subsection \ref{section: c_0 neq 0 k-positive}, except for the constant factor \(i\); hence \(\mathcal{A}(z)=i\epsilon_1\ln\!\big[\tanh\!\big(\tfrac{\sqrt c_0(z+c_1)}{2}\big)\big]+c_2\). Now we determine \(A(z)\) using the relation \(A=\frac{\mathcal{A}}{\sqrt{|k|}}\), giving \(A(z)=\frac{i\epsilon_1}{\sqrt{|k|}}\ln\!\big[\tanh\!\big(\tfrac{\sqrt c_0(z+c_1)}{2}\big)\big]+c_2\). Since \(u(z,\bar z)=A(z)+\overline{A(z)}\), we get formula for \(u(z,\bar z):\) 
  \footnotesize\(
    u(z,\bar z)=\tfrac{i\epsilon_1}{\sqrt{|k|}}\ln\!\big[\tanh\!\big(\tfrac{\sqrt c_0(z+c_1)}{2}\big)\big]+c_2-\tfrac{i\epsilon_1}{\sqrt{|k|}}\ln\!\big[\tanh\!\big(\tfrac{\sqrt{\bar c_0}(\bar z+\bar c_1)}{2}\big)\big]+\bar c_2
  \)\normalsize
  
  Hence we get \footnotesize\(\boxed{u(z,\bar z)=\tfrac{i\epsilon_1}{\sqrt{|k|}}\ln\big[\tanh(\tfrac{\sqrt c_0(z+c_1)}{2})\coth(\tfrac{\sqrt{\bar c_0}(\bar z+\bar c_1)}{2})\big]+2\Re(c_2)}.\;\)\normalsize We calculate various derivatives of \(u\): \(u_z=\tfrac{i\epsilon_1\sqrt c_0}{\sqrt{|k|}}\operatorname{csch}\!\big(\sqrt c_0[z+c_1]\big)\) and \(u_{\bar z}=\tfrac{-i\epsilon_1\sqrt{\bar c_0}}{\sqrt{|k|}}\operatorname{csch}\!\big(\sqrt{\bar c_0}[\bar z+\bar c_1]\big)\), the second derivatives are: \(u_{zz}=-\tfrac{i\epsilon_1 c_0}{\sqrt{|k|}} \operatorname{csch}(\sqrt{c_0}[z + c_1]) \operatorname{coth}(\sqrt{c_0}[z + c_1])\) and \footnotesize\(u_{\bar z \bar z} = \tfrac{i\epsilon_1 \bar{c}_0}{\sqrt{|k|}} \operatorname{csch}(\sqrt{\bar{c}_0}[\bar{z} + \bar{c}_1]) \operatorname{coth}(\sqrt{\bar{c}_0}[\bar{z} + \bar{c}_1])\). \( \tfrac{u_{zz}u_{\overline{z}}^2+u_{\overline{z}\overline{z}}u_z^2}{u_zu_{\overline{z}}}=\tfrac{i\epsilon_1|c_0|}{\sqrt{|k|}}\cdot\tfrac{\cosh(\sqrt c_0[z+c_1])-\cosh(\sqrt{\bar c_0}\;[\bar z+\bar c_1])}{\sinh\big(\sqrt c_0[z+c_1]\big)\sinh\big(\sqrt{\bar c_0}\;[\bar z+\bar c_1]\big)}=j(u)
  \)\normalsize 
  
  Let \(a = \tfrac{i\epsilon_1 |c_0|}{\sqrt{|k|}}\) and \(2v = \sqrt{c_0}[z + c_1]\). Using the hyperbolic identities \(\cosh(2x) - \cosh(2y) = 2\sinh(x + y)\sinh(x - y)\) and \(\sinh(2x) = 2\sinh(x)\cosh(x)\), we obtain \(j(u) = a \cdot \frac{\cosh(2v) - \cosh(2\bar{v})}{\sinh(2v)\sinh(2\bar{v})} = a \cdot \frac{\sinh(v + \bar{v})\sinh(v - \bar{v})}{2\sinh(v)\cosh(v)\sinh(\bar{v})\cosh(\bar{v})}=a\cdot \frac{\sinh(v+\bar v)\sinh(v-\bar v)}{2\sinh(v)\cosh(\bar v)\sinh(\bar v)\cosh(v)}\).  Employing the identity \(2\sinh(x)\cosh(y) =\tfrac{1}{2}\big(\sinh(x+y)+\sinh(x-y)\big)\) yields \(j(u) = 2a\cdot\frac{\sinh(v+\bar v)\sinh(v-\bar v)}{(\sinh(v+\bar v)+\sinh(v-\bar v))\cdot (\sinh(\bar v+v)+\sinh(\bar v-v))}\).  Set \(\mathbf{a}=2a\), \(\mathbf{b}=\sinh(v+\bar v)\) and \(\mathbf{c}=\sinh(v-\bar v)\); then \(\tfrac{j(u)}{\mathbf{a}}=\frac{\mathbf{b}\mathbf{c}}{\mathbf{b}^2-\mathbf{c}^2}\) and call this \(U\). Now \(e^{-i\sqrt{|k|}\epsilon_1[u-2\Re(c_2)]}=\tfrac{\sinh(v)\cosh(\bar v)}{\cosh(v)\sinh(\bar v)}\), so \(e^{-i\sqrt{|k|}\epsilon_1[u-2\Re(c_2)]}=\tfrac{\mathbf{b}+\mathbf{c}}{\mathbf{b}-\mathbf{c}}=\tfrac{(\mathbf{b}+\mathbf{c})^2}{\mathbf{b}^2-\mathbf{c}^2}\) and call this \(V\). Observe that \(V-4U=\frac{1}{V}\), hence \(U=\tfrac{1}{4}\big(V-\tfrac{1}{V}\big)\). Then \(j(u)=\mathbf{a} U=\tfrac{2a}{4}\big(V-\tfrac{1}{V}\big)\), i.e. \(\tfrac{2i\epsilon_1|c_0|}{4\sqrt{|k|}}\big(e^{-i\sqrt{|k|}\epsilon_1[u-2\Re(c_2)]}-e^{i\sqrt{|k|}\epsilon_1[u-2\Re(c_2)]}\big)=\tfrac{\epsilon_1 |c_0|}{\sqrt{|k|}}\sin\big(\sqrt{|k|}\epsilon_1[u-2\Re(c_2)]\big)\). Using the fact that \(\sin(x)\) is an odd function  we finally obtain the following pair:
  \footnotesize\[
    \boxed{u(z,\bar z)=\tfrac{i\epsilon_1}{\sqrt{|k|}}\ln\big[\tanh(\tfrac{\sqrt c_0(z+c_1)}{2})\coth(\tfrac{\sqrt{\bar c_0}(\bar z+\bar c_1)}{2})\big]+2\Re(c_2)}\;\;\text{and}\;\; \boxed{j(s)=\tfrac{|c_0|}{\sqrt{|k|}}\sin\big(\sqrt{|k|}[s-2\Re(c_2)]\big)}
  \]\normalsize
  We note that the second characteristic function \(j\) vanishes only at \footnotesize\(s=2\Re(c_2)+\tfrac{n\pi}{\sqrt{|k|}}\)\normalsize for each integer \(n\); hence \(j\) is never zero in the intervals \footnotesize\(J'=\big(2\Re(c_2)+\tfrac{n\pi}{\sqrt{|k|}},\,2\Re(c_2)+\tfrac{(n+1)\pi}{\sqrt{|k|}}\big)\;\)\normalsize for every integer \(n\). Likewise, the complex derivatives \(u_z=\tfrac{i\epsilon_1 \sqrt{c_0}}{\sqrt{k}}\operatorname{csch}(\sqrt{c_0}[z+c_1])\) and \(u_{\bar{z}}=\tfrac{-i\epsilon_1 \sqrt{\bar{c}_0}}{\sqrt{k}}\operatorname{csch}(\sqrt{\bar{c}_0}[\bar{z}+\bar{c}_1])\) are non‑zero everywhere on \(\mathbb{C}\) but become \(\infty\) at \(z=-c_1+\tfrac{2\pi i n}{\sqrt{c_0}}\) for all integers \(n\). Therefore, on \(\Omega=\mathbb{C}\setminus\{-c_1+\tfrac{2\pi i n}{\sqrt{c_0}}\mid n\in\mathbb{Z}\}\) the function \(u\) is well defined. Choosing \(\Omega_0=u^{-1}(J')\cap\Omega\) ensures that \(u_z,u_{\bar{z}}\neq0\) on \(\Omega_0\) and we set \(J=u(\Omega_0)\). Consequently the pair \((u,j)\) yields a valid solution of the Modified NMG Transformation Problem, with \(u:\Omega_0\to\mathbb{R}\) and \(j:J\to\mathbb{R}\) satisfying the required non‑vanishing conditions on their domains; thus \((u,j)\) fulfills the hypotheses of Theorem \ref{thm: Equivalence thm}.
  
  By Theorem \ref{thm: Equivalence thm}, there exists a corresponding Minimal Graph Surface $f$ and an associated First Characteristic Function $h$ that together solve the NMG Transformation Problem. Next, applying the converse part of Theorem \ref{thm: Equivalence thm}, we can construct all possible open subintervals $J_0 \subseteq J$ and corresponding pairs $(f, h)$ which solve the NMG Transformation Problem \ref{eqn: Complex Non-Trivial MGT: 2} on the domains $\Omega = u^{-1}(J_0)$ and $I = f(\Omega)$ such that the corresponding pair $(u_0, j_0)$ for the modified problem satisfies $u_0 = u|_{\Omega}$ and $j_0 = j|_{J_0}$.
  
  We now proceed to determine the explicit form of the minimal surface $f$ and the associated first characteristic function $h$.\\
  
  \textbf{Finding solutions \((f,h)\) to the Modified Nontrivial Minimal Graph Transformation Problem \ref{eqn: Complex Non-Trivial MGT: 2}:} We begin by computing the function: \(j_1(s) = -2\int j(s)\,ds + C_1 = -2\int \tfrac{|c_0|}{\sqrt{|k|}}\sin\big(\sqrt{|k|}[s-2\Re(c_2)]\big)\,ds + C_1 = C_1 + \tfrac{2|c_0|}{|k|}\cos\big(\sqrt{|k|}[s-2\Re(c_2)]\big)\). Next we calculate \footnotesize\(j_2(s) = \int \tfrac{1}{\sqrt{j_1(s)}}\,ds = \int \tfrac{1}{\sqrt{C_1 + \tfrac{2|c_0|}{|k|}\cos\big(\sqrt{|k|}[s-2\Re(c_2)]\big)}}\,ds\)\normalsize. For this to be real‑valued the quantity inside the square root must be positive, i.e. \(C_1 + \tfrac{2|c_0|}{|k|}\cos\big(\sqrt{|k|}[s-2\Re(c_2)]\big) > 0\) at least on some subdomain \(J_0\subseteq J\) of \(j_1\). A necessary and sufficient condition for \(j_1(s)\) to be positive on some interval is \(C_1 + \tfrac{2|c_0|}{|k|} > 0\); we choose \(C_1\) accordingly. With this choice we evaluate \(j_2(s)\) via the following substitution \footnotesize\(x = \sqrt{C_1 + \tfrac{2|c_0|}{|k|}\cos\big(\sqrt{|k|}[s-2\Re(c_2)]\big)}\)\normalsize, which yields \footnotesize\(dx = \tfrac{-\frac{2|c_0|}{\sqrt{|k|}}\sin\big(\sqrt{|k|}[s-2\Re(c_2)]\big)}{2\sqrt{C_1 + \frac{2|c_0|}{|k|}\cos\big(\sqrt{|k|}[s-2\Re(c_2)]\big)}}\,ds\)\normalsize. Hence \footnotesize\(2x\,dx = -\tfrac{2|c_0|}{\sqrt{|k|}}\sin\big(\sqrt{|k|}[s-2\Re(c_2)]\big)ds = -\frac{2\epsilon_2|c_0|}{\sqrt{|k|}}\sqrt{1-\cos^2\big(\sqrt{|k|}[s-2\Re(c_2)]\big)}ds\)\normalsize. We observe that \(\cos^2\big(\sqrt{|k|}[s-2\Re(c_2)]\big)=\tfrac{|k|^2(x^2-C_1)^2}{4|c_0|^2}\), so we get:   
  \(
    j_2(s)=\int \frac{ds}{x}=- \int \tfrac{dx}{\frac{\epsilon_2|c_0|}{\sqrt{|k|}}\sqrt{1-\frac{|k|^2(x^2-C_1)^2}{4|c_0|^2}}}=-\tfrac{2\epsilon_2}{\sqrt{|k|}} \int \tfrac{dx}{\sqrt{\big[\frac{2|c_0|}{|k|}\big]^2-\,[x^2-C_1]^2}}
  \)
  
  Hence we get \footnotesize\(j_2(s)=-\tfrac{2\epsilon_2}{\sqrt{|k|}} \displaystyle \int \tfrac{dx}{\sqrt{\big[\frac{2|c_0|}{|k|}+C_1-x^2\big]\big[\frac{2|c_0|}{|k|}-C_1+x^2\big]}}\)\normalsize. Set \(\alpha^2=\frac{2|c_0|}{|k|}+C_1\) as this must be positive. Now consider \(\frac{2|c_0|}{|k|}-C_1\), we have the following three cases: 
  
  \(\boxed{\textit{Case 1:}\;\,\tfrac{2|c_0|}{|k|}-C_1=0}\qquad\) \(\boxed{\textit{Case 2:}\;\,\tfrac{2|c_0|}{|k|}-C_1>0} \qquad\) \(\boxed{\textit{Case 3:}\;\,\tfrac{2|c_0|}{|k|}-C_1<0}\)\\
  
  \textit{CASE 1}: We obtain \(j_2(s)=-\tfrac{2\epsilon_2}{\sqrt{|k|}}\int \tfrac{dx}{x\sqrt{\alpha^2-x^2}}\). Referring to \cite{JD} Page: 172 Equation: 4.3.3.1.15 gives \(j_2(s)=-\tfrac{2\epsilon_2}{\sqrt{|k|}}\cdot\tfrac{1}{\alpha} \ln\!\big|\tfrac{\alpha-\sqrt{\alpha^2-x^2}}{x}\big|\). Since \(\alpha>x>0\) by definition, this simplifies to \(j_2(s)=-\tfrac{2\epsilon_2}{\sqrt{|k|}}\cdot\tfrac{1}{\alpha}\ln\!\big(\tfrac{\alpha-\sqrt{\alpha^2-x^2}}{x}\big)\). Using the logarithmic representation of \(\operatorname{sech}^{-1}\) (see \cite{AS} Page: 88 Equation: 4.6.24) we find \footnotesize\(j_2(s)=-\tfrac{2\epsilon_2}{\alpha \sqrt{|k|}}\operatorname{sech}^{-1}[\tfrac{x}{\alpha}]\)\normalsize, which is basically \footnotesize\(-\tfrac{2\epsilon_2}{\alpha \sqrt{|k|}}\operatorname{sech}^{-1}\big[\tfrac{1}{\alpha}\sqrt{C_1 + \tfrac{2|c_0|}{|k|}\cos\big(\sqrt{|k|}[s-2\Re(c_2)]\big)}\;\big]\)\normalsize. Due to the complexity of this expression we invoke Remark~\ref{remark: bypassing H} to bypass the explicit construction of the diffeomorphism \(H\). First note that \footnotesize\(j_2(H(s))=-\tfrac{2\epsilon_2}{\alpha \sqrt{|k|}}\operatorname{sech}^{-1}\!\Big(\tfrac{1}{\alpha}\sqrt{j_1(s)}\Big)\)\normalsize. By Remark~\ref{remark: bypassing H} we have \footnotesize\(s+C_2=j_2(H(s))=-\tfrac{2\epsilon_2}{\alpha \sqrt{|k|}}\operatorname{sech}^{-1}\!\Big(\tfrac{1}{\alpha}\sqrt{j_1(H(s))}\Big)=-\tfrac{2\epsilon_2}{\alpha \sqrt{|k|}}\operatorname{sech}^{-1}\!\Big(\tfrac{1}{\alpha}H'(s)\Big)\)\normalsize, which yields \footnotesize\(H'(s)=\alpha\operatorname{sech}\!\big[\tfrac{-\alpha \sqrt{|k|}}{2\epsilon_2}(s+C_2)\big]\)\normalsize. Because \(\operatorname{sech}\) is even, this becomes \footnotesize\(H'(s)=\alpha\operatorname{sech}\!\big[\tfrac{\alpha \sqrt{|k|}}{2}(s+C_2)\big]\)\normalsize; consequently \footnotesize\(H''(s)=-\tfrac{\alpha^2\sqrt{|k|}}{2}\operatorname{sech}\!\big(\tfrac{\alpha\sqrt{|k|}}{2}(s+C_2)\big)\tanh\!\big(\tfrac{\alpha\sqrt{|k|}}{2}(s+C_2)\big)\)\normalsize. Hence the first characteristic function for Case~1 (where \(\tfrac{2|c_0|}{|k|}-C_1=0\)) is \footnotesize\(h(s)=-\tfrac{H''(s)}{H'(s)}=\tfrac{\alpha\sqrt{|k|}}{2}\tanh\!\big[\frac{\alpha\sqrt{|k|}}{2}(s+C_2)\big]\)\normalsize. Again using Remark~\ref{remark: bypassing H} we obtain the minimal graph surface \footnotesize\(f(z,\bar z)=K(u)=j_2(u)-C_2=-\tfrac{2\epsilon_2}{\alpha \sqrt{|k|}}\operatorname{sech}^{-1}\!\Big(\tfrac{1}{\alpha}\sqrt{C_1 + \tfrac{2|c_0|}{|k|}\cos\!\big(\sqrt{|k|}[u-2\Re(c_2)]\big)}\Big)-C_2\)\normalsize. Substituting the expression for \(u\) gives: \footnotesize\(
    f(z,\bar z)=-\tfrac{2\epsilon_2}{\alpha \sqrt{|k|}}\operatorname{sech}^{-1}\big[\tfrac{1}{\alpha}\sqrt{C_1 + \tfrac{2|c_0|}{|k|}\cos\big(i\epsilon_1\ln\big[\tanh(\tfrac{\sqrt c_0(z+c_1)}{2})\coth(\tfrac{\sqrt{\bar c_0}(\bar z+\bar c_1)}{2})\big]\big)}\;\big]-C_2
  \)\normalsize. Now using the evenness of \(\cos\), the identity \(\cos(ix)=\cosh(x)\), and the condition that: \(\tfrac{2|c_0|}{|k|}-C_1=0\) (which implies \(\alpha=2\sqrt{\tfrac{|c_0|}{|k|}}\)), we obtain the following pair:
  \footnotesize
  \begin{equation}
  	\label{eqn: soln k<0, c_0 neq 0, f: case 1}
  	\boxed{f(z, \bar z)=\tfrac{\epsilon_2}{\sqrt{|c_0|}}\operatorname{sech}^{-1}\Big[\tfrac{1}{\sqrt{2}}\sqrt{1 + \cosh\big(\ln\big[\tanh(\tfrac{\sqrt c_0(z+c_1)}{2})\coth(\tfrac{\sqrt{\bar c_0}(\bar z+\bar c_1)}{2})\big]\big)}\;\Big]-C_2}
  \end{equation}
  \begin{equation}
     \label{eqn: soln k<0, c_0 neq 0, h: case 1}
     \boxed{h(s)=\sqrt{|c_0|}\tanh\big[\frac{\alpha\sqrt{|k|}}{2}(s+C_2)\big]}
  \end{equation}
  \normalsize
  \begin{figure}[htbp]
  	\centering
  	\begin{subfigure}{0.22\textwidth}
  		\centering
  		\includegraphics[width=\linewidth]{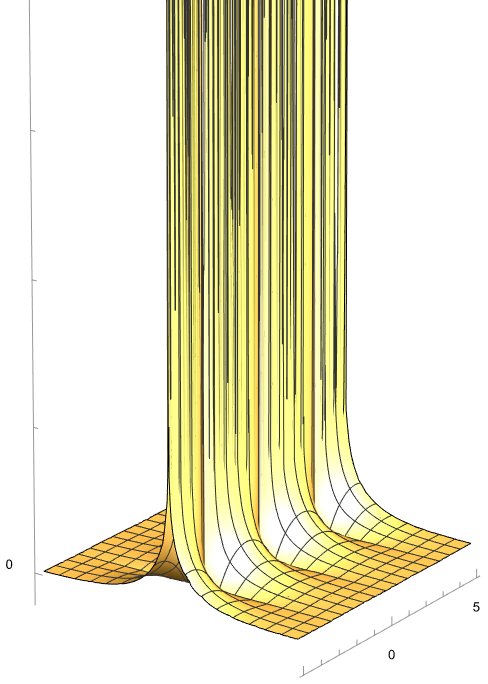}
  		\caption{\footnotesize Top view of one layer\normalsize}
  	\end{subfigure}
  	\hfill
  	\begin{subfigure}{0.22\textwidth}
  		\centering
  		\includegraphics[width=\linewidth]{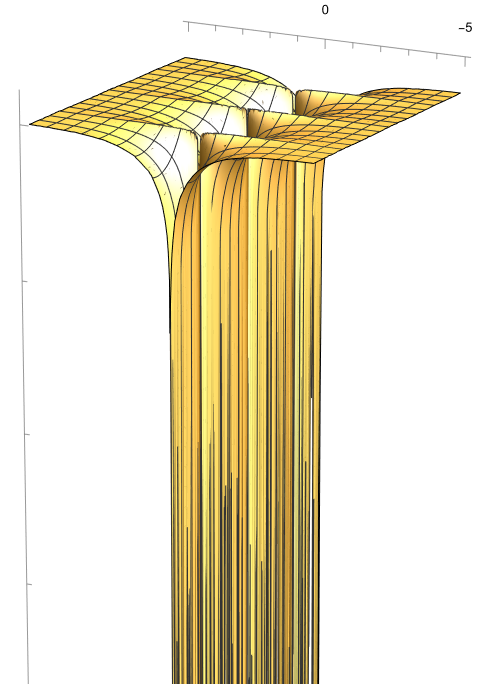}
  		\caption{\footnotesize Bottom view of one layer\normalsize}
  	\end{subfigure}
    \hfill
    \begin{subfigure}{0.22\textwidth}
    	\centering
    	\includegraphics[width=\linewidth]{Great_wall_implicit.png}
    	\caption{\footnotesize Implicit Surface Plot \normalsize}
    \end{subfigure}
  	\caption{\footnotesize The images above depict the minimal surface \( f(z,\bar z) \) given in equation \ref{eqn: soln k<0, c_0 neq 0, f: case 1}. They were generated with the parameter values	\( c_0=1 \), \( c_1=0 \), \( C_2=0 \) and \( \epsilon_2=1 \).\normalsize}
  	\label{pic: Great Wall}
  \end{figure}

  \textit{CASE 2}: We set \(\tfrac{2|c_0|}{|k|}-C_1=\beta^2>0\) and so \(j_2(s)=-\tfrac{2\epsilon_2}{\sqrt{|k|}}\int \tfrac{dx}{\sqrt{[\alpha^2-x^2][\beta^2+x^2]}}\). Note that this is an elliptic integral and we refer to \cite{La} Page 53 Equation 3.2.7 for this: \footnotesize\(j_2(s)=-\tfrac{2\epsilon_2}{\sqrt{|k|}\sqrt{\alpha^2+\beta^2}}\operatorname{sd}^{-1}\Big[\tfrac{\sqrt{\alpha^2+\beta^2}x}{\alpha\beta},\tfrac{\alpha}{\sqrt{\alpha^2+\beta^2}}\Big]\)\normalsize.  Recall the identity: $\operatorname{sd}^{-1}(x/k',k)=\operatorname{sn}^{-1}(1,k)-\operatorname{cn}^{-1}(x,k)$. Using this we simplify \(j_2(s)\) to the following: \footnotesize\(j_2(s)=\tfrac{2\epsilon_2}{\sqrt{|k|}\sqrt{\alpha^2+\beta^2}}\operatorname{sn}^{-1}\big(1,\tfrac{\alpha}{\sqrt{\alpha^2+\beta^2}}\big)+\tfrac{2\epsilon_2}{\sqrt{|k|}\sqrt{\alpha^2+\beta^2}}\operatorname{cn}^{-1}\Big[\tfrac{x}{\alpha},\tfrac{\alpha}{\sqrt{\alpha^2+\beta^2}}\Big]\)\normalsize. Let \(D\) denotes \footnotesize\(\tfrac{2\epsilon_2}{\sqrt{|k|}\sqrt{\alpha^2+\beta^2}}\operatorname{sn}^{-1}\big(1,\tfrac{\alpha}{\sqrt{\alpha^2+\beta^2}}\big).\;\)\normalsize Substituting \(x\), we get the following: \footnotesize\(j_2(s)=D\;+\;\tfrac{2\epsilon_2}{\sqrt{|k|}\sqrt{\alpha^2+\beta^2}}\operatorname{cn}^{-1}\Big[\tfrac{1}{\alpha}\sqrt{C_1 + \tfrac{2|c_0|}{|k|}\cos\big(\sqrt{|k|}[s-2\Re(c_2)]\big)},\tfrac{\alpha}{\sqrt{\alpha^2+\beta^2}}\Big]\)\normalsize. 
  Again, due to the difficulty in dealing with this complicated expression we will make use of the Remark: \ref{remark: bypassing H} to bypass the explicit construction of the diffeomorphism \(H\). Notice that \footnotesize\(j_2(s)=D+\tfrac{2\epsilon_2}{\sqrt{|k|}\sqrt{\alpha^2+\beta^2}}\operatorname{cn}^{-1}\Big[\tfrac{1}{\alpha}\sqrt{j_1(s)},\tfrac{\alpha}{\sqrt{\alpha^2+\beta^2}}\Big]\)\normalsize. So by Remark: \ref{remark: bypassing H}, we have: \footnotesize\(s+C_2=D+\tfrac{2\epsilon_2}{\sqrt{|k|}\sqrt{\alpha^2+\beta^2}}\operatorname{cn}^{-1}\Big[\tfrac{1}{\alpha}\sqrt{j_1(H(s))},\tfrac{\alpha}{\sqrt{\alpha^2+\beta^2}}\Big]=D+\tfrac{2\epsilon_2}{\sqrt{|k|}\sqrt{\alpha^2+\beta^2}}\operatorname{cn}^{-1}\Big[\tfrac{1}{\alpha}H'(s),\tfrac{\alpha}{\sqrt{\alpha^2+\beta^2}}\Big].\;\)\normalsize Absorb the constant \(D\) inside \(C_2\) and use the fact that \(\operatorname{cn}\) is an even function, we get an expression for \(H'(s)\): \footnotesize\(H'(s)=\alpha\operatorname{cn}\big[\tfrac{\sqrt{|k|}\sqrt{\alpha^2+\beta^2}}{2}(s+C_2),\tfrac{\alpha}{\sqrt{\alpha^2+\beta^2}}\big].\;\)\normalsize We compute \(H''(s)\), we refer to \cite{JD} Page: 249 Equation: 12.3.1.1.2 for the derivative of \(\operatorname{cn}\): \(\tfrac{d}{dt}\operatorname{cn}(s)=-\operatorname{sn}(s)\operatorname{dn}(s)\). So: \footnotesize\(H''(s)=-\alpha\tfrac{\sqrt{|k|}\sqrt{\alpha^2+\beta^2}}{2}\operatorname{sn}\big[\tfrac{\sqrt{|k|}\sqrt{\alpha^2+\beta^2}}{2}(s+C_2),\tfrac{\alpha}{\sqrt{\alpha^2+\beta^2}}\big]\operatorname{dn}\big[\tfrac{\sqrt{|k|}\sqrt{\alpha^2+\beta^2}}{2}(s+C_2),\tfrac{\alpha}{\sqrt{\alpha^2+\beta^2}}\big].\;\)\normalsize Hence we get that: \footnotesize\(h(s)=\tfrac{-H''(s)}{H'(s)}=\tfrac{\sqrt{|k|}\sqrt{\alpha^2+\beta^2}}{2}\operatorname{sc}\big[\tfrac{\sqrt{|k|}\sqrt{\alpha^2+\beta^2}}{2}(s+C_2),\tfrac{\alpha}{\sqrt{\alpha^2+\beta^2}}\big]\operatorname{dn}\big[\tfrac{\sqrt{|k|}\sqrt{\alpha^2+\beta^2}}{2}(s+C_2),\tfrac{\alpha}{\sqrt{\alpha^2+\beta^2}}\big].\;\)\normalsize Now see that: \footnotesize\(\sqrt{\alpha^2+\beta^2}=2\sqrt{\tfrac{|c_0|}{|k|}}\;\)\normalsize and consequently, \footnotesize\(\tfrac{\sqrt{|k|}\sqrt{\alpha^2+\beta^2}}{2}=\sqrt{|c_0|},\;\)\normalsize Hence: \footnotesize\(h(s)=\sqrt{|c_0|}\operatorname{sc}\big[\sqrt{|c_0|}(s+C_2),\tfrac{\alpha}{\sqrt{\alpha^2+\beta^2}}\big]\operatorname{dn}\big[\sqrt{|c_0|}(s+C_2),\tfrac{\alpha}{\sqrt{\alpha^2+\beta^2}}\;\big].\;\)\normalsize We see from Remark \ref{remark: bypassing H} that \footnotesize\(K(u)=j_2(u)-C_2\)\normalsize, using this and the fact that \(f(z,\bar z)=K(u)\) we find the minimal graph surface \footnotesize\(f(z,\bar z)=\tfrac{2\epsilon_2}{\sqrt{|k|}\sqrt{\alpha^2+\beta^2}}\operatorname{cn}^{-1}\Big[\tfrac{1}{\alpha}\sqrt{j_1(u)},\tfrac{\alpha}{\sqrt{\alpha^2+\beta^2}}\Big]-C_2\)\normalsize. Now substituting the expression for \(u\) and applying properties of \(\cos\) we get:
  
  \footnotesize\(
  f(z,\bar z)=\tfrac{\pm1}{\sqrt{|c_0|}}\operatorname{cn}^{-1}\Big[\tfrac{1}{\alpha}\sqrt{C_1 + \tfrac{2|c_0|}{|k|}\cosh\big(\ln\big[\tanh(\tfrac{\sqrt c_0(z+c_1)}{2})\coth(\tfrac{\sqrt{\bar c_0}(\bar z+\bar c_1)}{2})\big]\big)},\tfrac{\alpha}{\sqrt{\alpha^2+\beta^2}}\Big]-C_2
  \)\normalsize
  
  \noindent
  Now $\operatorname{cn}^{-1}(x,k)=\operatorname{dn}^{-1}[\sqrt{k'^2+k^2x^2},k]$  (see \cite{BF} Page: 31, Equation 131.01), using this identity and then calling: \(\gamma=\tfrac{\alpha}{\sqrt{\alpha^2+\beta^2}}\) we obtain the following expression for \(f(z, \bar z)\): 
  
  \footnotesize\(
  f(z,\bar z)=\tfrac{\pm1}{\sqrt{|c_0|}}\operatorname{dn}^{-1}\Big[\tfrac{1}{\sqrt{2}}\sqrt{1 + \cosh\big(\ln\big[\tanh(\tfrac{\sqrt c_0(z+c_1)}{2})\coth(\tfrac{\sqrt{\bar c_0}(\bar z+\bar c_1)}{2})\big]\big)},\gamma\Big]-C_2
  \)\normalsize
  
  \noindent  
  Hence we have the following pair:
  \footnotesize
  \begin{equation}
  	\label{eqn: soln k<0, c_0 neq 0, f: case 2}
  	\boxed{f(z,\bar z)=\tfrac{\epsilon_2}{\sqrt{|c_0|}}\operatorname{dn}^{-1}\Big[\tfrac{1}{\sqrt{2}}\sqrt{1 + \cosh\big(\ln\big[\tanh(\tfrac{\sqrt c_0(z+c_1)}{2})\coth(\tfrac{\sqrt{\bar c_0}(\bar z+\bar c_1)}{2})\big]\big)},\gamma\Big]-C_2}
  \end{equation}
  \begin{equation}
  	\label{eqn: soln k<0, c_0 neq 0, h: case 2}
  	\boxed{h(s)=\sqrt{|c_0|}\operatorname{sc}\Big[\sqrt{|c_0|}(s+C_2),\gamma\Big]\operatorname{dn}\Big[\sqrt{|c_0|}(s+C_2),\gamma\;\Big].\;}
  \end{equation}
  \normalsize
  \begin{figure}[htbp]
  	\centering
  	\begin{subfigure}{0.30\textwidth}
  		\centering
  		\includegraphics[width=\linewidth]{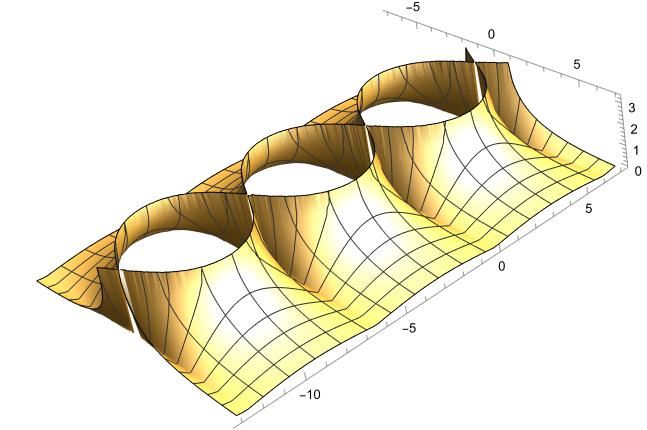}
  		\caption{\footnotesize Top view of one layer \normalsize}
  	\end{subfigure}
  	\hfill
  	\begin{subfigure}{0.27\textwidth}
  		\centering
  		\includegraphics[width=\linewidth]{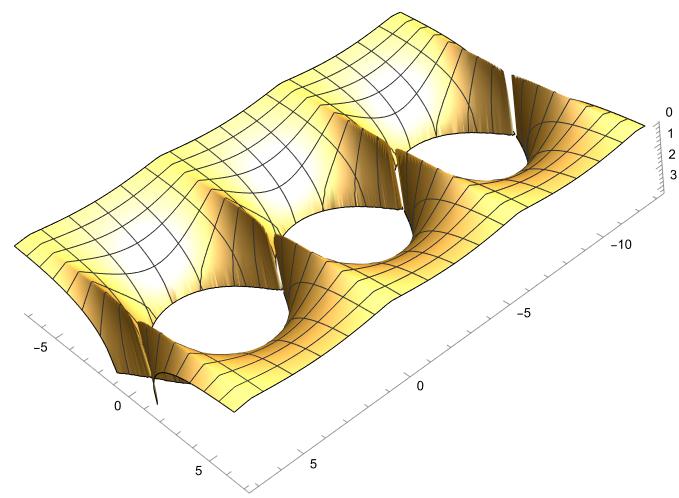}
  		\caption{\footnotesize Bottom view of one layer \normalsize}
  	\end{subfigure}
  	\hfill
  	\begin{subfigure}{0.26\textwidth}
  		\centering
  		\includegraphics[width=\linewidth]{Thick_wall_implicit.jpeg}
  		\caption{\footnotesize Implicit Surface Plot \normalsize}
  	\end{subfigure}
  	\caption{\footnotesize The images above depict the minimal surface \( f(z,\bar z) \) given in equation \ref{eqn: soln k<0, c_0 neq 0, f: case 2}. They were generated using with the parameter values	\( c_0=0.25 \), \( c_1=0 \), \( C_2=0 \), \(\gamma=0.2\) and \( \epsilon_2=1 \).\normalsize}
  	\label{pic: Thick Wall}
  \end{figure}

  \textit{CASE 3}: We suppose \(\tfrac{2|c_0|}{|k|}-C_1=-\beta^2<0\), so that \(j_2(s)=-\tfrac{2\epsilon_2}{\sqrt{|k|}}\int \tfrac{dx}{\sqrt{[\alpha^2-x^2][x^2-\beta^2]}}\). This is an elliptic integral; referring to \cite{La} Page: 53, Equation: 3.2.10 gives \(j_2(s)=-\tfrac{2\epsilon_2}{\alpha\sqrt{|k|}}\operatorname{nd}^{-1}\!\big[\tfrac{x}{\beta},\frac{\sqrt{\alpha^2-\beta^2}}{\alpha}\big]\). Using the identity \(\operatorname{nd}^{-1}(x,k)=\operatorname{sn}^{-1}(1,k)-\operatorname{dn}^{-1}(xk',k)\) (see \cite{La} Page: 54, Equation: 3.2.21 and setting \(D=-\tfrac{2\epsilon_2}{\sqrt{|k|}}\operatorname{sn}^{-1}(1,k)\), we obtain \(j_2(s)=D+\tfrac{2\epsilon_2}{\alpha\sqrt{|k|}}\operatorname{dn}^{-1}\!\big[\tfrac{x}{\alpha},\frac{\sqrt{\alpha^2-\beta^2}}{\alpha}\big]\). Substituting the expression for \(x\) yields the following: \footnotesize\(j_2(s)=D+\tfrac{2\epsilon_2}{\alpha\sqrt{|k|}}\operatorname{dn}^{-1}\!\big[\tfrac{1}{\alpha}\sqrt{C_1 + \tfrac{2|c_0|}{|k|}\cos\!\big(\sqrt{|k|}[s-2\Re(c_2)]\big)},\frac{\sqrt{\alpha^2-\beta^2}}{\alpha}\big]\)\normalsize. To find the first characteristic function, use Remark~\ref{remark: bypassing H} and we have \footnotesize\(j_2(s)=D+\tfrac{2\epsilon_2}{\alpha\sqrt{|k|}}\operatorname{dn}^{-1}\!\big[\tfrac{1}{\alpha}\sqrt{j_1(s)},\frac{\sqrt{\alpha^2-\beta^2}}{\alpha}\big]\)\normalsize, therefore \footnotesize\(s+C_2=j_2(H(s))=D+\tfrac{2\epsilon_2}{\alpha\sqrt{|k|}}\operatorname{dn}^{-1}\!\big[\tfrac{1}{\alpha}\sqrt{j_1(H(s))},\frac{\sqrt{\alpha^2-\beta^2}}{\alpha}\big]=D+\tfrac{2\epsilon_2}{\alpha\sqrt{|k|}}\operatorname{dn}^{-1}\!\big[\tfrac{H'(s)}{\alpha},\frac{\sqrt{\alpha^2-\beta^2}}{\alpha}\big]\)\normalsize. Absorbing the constant \(D\) into \(C_2\) and inverting, then using the evenness of \(\operatorname{dn}\), gives the explicit expression \footnotesize\(H'(s)=\alpha\operatorname{dn}\!\big[\tfrac{\alpha\sqrt{|k|}}{2}(s+C_2), \frac{\sqrt{\alpha^2-\beta^2}}{\alpha}\big]\)\normalsize. From \cite{JD} Page: 249, Equation: 12.3.1.1.3 we have the derivative:  \(\tfrac{d}{dt}\operatorname{dn}(s,k)=-k^2\operatorname{sn}(s,k)\operatorname{cn}(s,k)\); and thus we have \footnotesize\(H''(s)=-\tfrac{(\alpha^2-\beta^2)\sqrt{|k|}}{2}\operatorname{sn}\!\big[\tfrac{\alpha\sqrt{|k|}}{2}(s+C_2), \frac{\sqrt{\alpha^2-\beta^2}}{\alpha}\big]\operatorname{cn}\!\big[\tfrac{\alpha\sqrt{|k|}}{2}(s+C_2), \frac{\sqrt{\alpha^2-\beta^2}}{\alpha}\big]\)\normalsize. Consequently, \footnotesize\(h(s)=\tfrac{-H''(s)}{H'(s)}=\tfrac{[\alpha^2-\beta^2]\sqrt{|k|}}{2\alpha}\operatorname{sd}\!\big[\tfrac{\alpha\sqrt{|k|}}{2}(s+C_2), \frac{\sqrt{\alpha^2-\beta^2}}{\alpha}\big]\operatorname{cn}\!\big[\tfrac{\alpha\sqrt{|k|}}{2}(s+C_2), \frac{\sqrt{\alpha^2-\beta^2}}{\alpha}\big]\)\normalsize. Substituting back \(\alpha,\beta\) yields \footnotesize\(h(s)=\tfrac{2|c_0|}{\alpha\sqrt{|k|}}\operatorname{sd}\!\big[\tfrac{\alpha\sqrt{|k|}}{2}(s+C_2), \tfrac{2\sqrt{|c_0|}}{\alpha\sqrt{|k|}}\big]\operatorname{cn}\!\big[\tfrac{\alpha\sqrt{|k|}}{2}(s+C_2),\tfrac{2\sqrt{|c_0|}}{\alpha\sqrt{|k|}}\big]\)\normalsize. To simplify \(j_2\) we use the identity \footnotesize\(\operatorname{dn}^{-1}(x,k)=\operatorname{cn}^{-1}\!\big[\sqrt{\tfrac{x^2-k'^2}{k^2}},k\big]\)\normalsize, giving \footnotesize\(j_2(s)=D+\tfrac{2\epsilon_2}{\alpha\sqrt{|k|}}\operatorname{cn}^{-1}\!\big[\tfrac{1}{\sqrt{2}}\sqrt{1 + \cos\!\big(\sqrt{|k|}[s-2\Re(c_2)]\big)},\tfrac{2\sqrt{|c_0|}}{\alpha\sqrt{|k|}}\big]\)\normalsize. By Remark~\ref{remark: bypassing H} we have \(K(s)=j_2(s)-C_2\) (with \(D\) absorbed into \(C_2\)), and since \(f(z,\bar z)=K(u)\) we obtain:
  
  \footnotesize\(
      f(z,\bar z)=\tfrac{2\epsilon_2}{\alpha\sqrt{|k|}}\operatorname{cn}^{-1}\Big[\tfrac{1}{\sqrt{2}}\sqrt{1 + \cosh\big(\ln\big[\tanh(\tfrac{\sqrt c_0(z+c_1)}{2})\coth(\tfrac{\sqrt{\bar c_0}(\bar z+\bar c_1)}{2})\big]\big)},\tfrac{2\sqrt{|c_0|}}{\alpha\sqrt{|k|}}\Big]-C_2
  \)\normalsize
  
  \noindent
  Notice that \footnotesize\(\tfrac{2\sqrt{|c_0|}}{\alpha\sqrt{|k|}}\)\normalsize can attain any positive value between 0 and 1, hence call it \(\gamma\) and hence we have the following pair:
  \footnotesize
  \begin{equation}
  	\label{eqn: soln k<0, c_0 neq 0, f: case 3}
  	\boxed{f(z,\bar z)=\epsilon_2\tfrac{\gamma}{\sqrt{|c_0|}}\operatorname{cn}^{-1}\Big[\tfrac{1}{\sqrt{2}}\sqrt{1 + \cosh\big(\ln\big[\tanh(\tfrac{\sqrt c_0(z+c_1)}{2})\coth(\tfrac{\sqrt{\bar c_0}(\bar z+\bar c_1)}{2})\big]\big)},\gamma\Big]-C_2}
  \end{equation}
  \begin{equation}
  	\label{eqn: soln k<0, c_0 neq 0, h: case 3}
  	\boxed{h(s)=\gamma\sqrt{|c_0|}\operatorname{sd}\big[\tfrac{\sqrt{|c_0|}}{\gamma}(s+C_2),\gamma\big]\operatorname{cn}\big[\tfrac{\sqrt{|c_0|}}{\gamma}(s+C_2),\gamma\big]}
  \end{equation}
  \normalsize
  \begin{figure}[htbp]
  	\centering
  	\begin{subfigure}{0.32\textwidth}
  		\centering
  		\includegraphics[width=\linewidth]{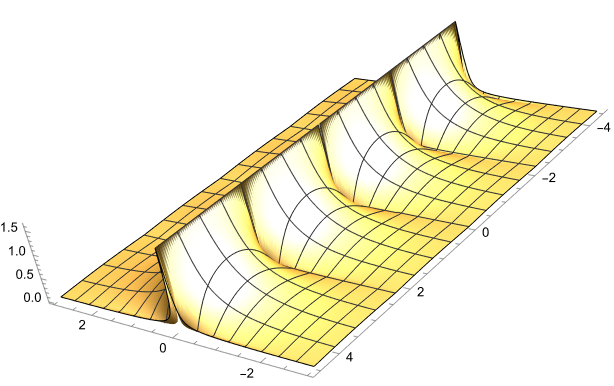}
  		\caption{\footnotesize Top view of one layer \normalsize}
  	\end{subfigure}
  	\hfill
  	\begin{subfigure}{0.30\textwidth}
  		\centering
  		\includegraphics[width=\linewidth]{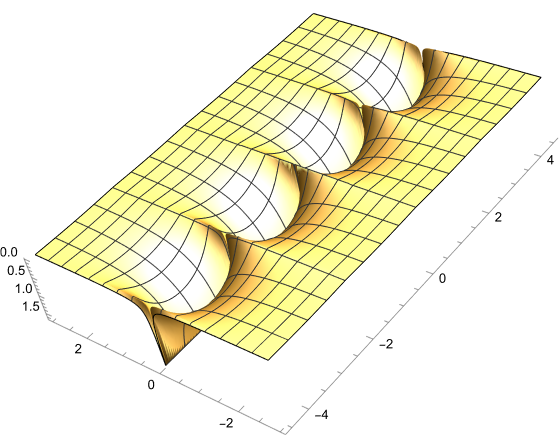}
  		\caption{\footnotesize Bottom view of one layer \normalsize}
  	\end{subfigure}
  	\hfill
  	\begin{subfigure}{0.28\textwidth}
  		\centering
  		\includegraphics[width=\linewidth]{Sharp_wall_implicit.png}
  		\caption{\footnotesize Implicit Surface Plot \normalsize}
  	\end{subfigure}
  	\caption{\footnotesize The images above depict the minimal surface \( f(z,\bar z) \) given in equation \ref{eqn: soln k<0, c_0 neq 0, f: case 3}. They were generated using with the parameter values	\( c_0=2 \), \( c_1=0 \), \( C_2=0 \), \(\gamma=0.8\) and \( \epsilon_2=1 \).\normalsize}
  	\label{pic: Sharp Wall}
  \end{figure}

  \textbf{Finding Nontrivial Minimal Graph Transformations: } From Remark~\ref{remark: bypassing int h} we have \footnotesize\(g = \pm \displaystyle \int \tfrac{1}{\sqrt{1\; \pm\; \left(\frac{C}{H'(s)}\right)^2}}\, dt + C_3=\pm \displaystyle \int \tfrac{H'(s)}{\sqrt{[H'(s)]^2\; \pm\; C^2}}\, dt + C_3\)\normalsize. We now apply this formula to each of the three cases:\\

  \textit{CASE 1:} Recall that \(\operatorname{sech}\) is an even function, so \footnotesize\(H'(s)=\alpha\operatorname{sech}\!\big[\tfrac{\alpha \sqrt{|k|}}{2}(s+C_2)\big]\)\normalsize; substituting \footnotesize\(\alpha=2\tfrac{\sqrt{|c_0|}}{\sqrt{|k|}}\)\normalsize gives \footnotesize\(H'(s)=\alpha\operatorname{sech}\!\big[\sqrt{|c_0|}(s+C_2)\big]\)\normalsize. Using this, we get NMG Transformation \(g\): \footnotesize\(g =\pm \int \tfrac{\alpha\operatorname{sech}\!\big[\sqrt{|c_0|}(s+C_2)\big]}{\sqrt{\alpha^2\operatorname{sech}^2\!\big[\sqrt{|c_0|}(s+C_2)\big]\; \pm\; C^2}}\, dt + C_3\)\normalsize. Multiply the factor \footnotesize\(-\sqrt{|c_0|}\operatorname{tanh}\!\big[\sqrt{|c_0|}(s+C_2)\;\)\normalsize to the numerator and denominator of the integrand: \footnotesize\(\tfrac{\pm 1}{\sqrt{|c_0|}} \int \tfrac{-\sqrt{|c_0|}\operatorname{sech}[\sqrt{|c_0|}(s+C_2)]\operatorname{tanh}[\sqrt{|c_0|}(s+C_2)]}{\operatorname{tanh}[\sqrt{|c_0|}(s+C_2)]\sqrt{\operatorname{sech}^2[\sqrt{|c_0|}(s+C_2)]\; \pm\; \frac{C^2}{\alpha^2}}}\, dt + C_3\)\normalsize. Further we obtain : \footnotesize\(g=\tfrac{\pm 1}{\sqrt{|c_0|}} \int \tfrac{-\sqrt{|c_0|}\operatorname{sech}[\sqrt{|c_0|}(s+C_2)]\operatorname{tanh}[\sqrt{|c_0|}(s+C_2)]}{\sqrt{\big[1-\operatorname{sech}^2[\sqrt{|c_0|}(s+C_2)]\big]\big[\operatorname{sech}^2[\sqrt{|c_0|}(s+C_2)]\; \pm\; \tfrac{C^2}{\alpha^2}\big]}}\, dt + C_3\;\)\normalsize using the identity \footnotesize\(\tanh(x)=\sqrt{1-\operatorname{sech}^2(x)}\;\)\normalsize. Setting \footnotesize\(y=\operatorname{sech}[\sqrt{|c_0|}(s+C_2)]\,\)\normalsize gives: \footnotesize\(g=\tfrac{\pm 1}{\sqrt{|c_0|}} \int \tfrac{dy}{\sqrt{[1-y^2][y^2 \pm \frac{C^2}{\alpha^2}]}} + C_3\)\normalsize. Let \(C_4=\tfrac{|C|}{\alpha}>0\); since this can attain any positive value,  we treat \(C_4\) as an arbitrary positive constant yeilding the following 2 integrals:
  
  \footnotesize\(
  \big(I_1\big):\;\;\; g=\tfrac{\pm 1}{\sqrt{|c_0|}} \displaystyle \int \tfrac{dy}{\sqrt{[1-y^2][y^2 + C_4^2]}} + C_3\qquad\qquad \big(I_2\big):\;\;\; g=\tfrac{\pm 1}{\sqrt{|c_0|}} \displaystyle \int \tfrac{dy}{\sqrt{[1-y^2][y^2 - C_4^2]}} + C_3
  \)\normalsize

  \noindent
  We refer \cite{La} page 53, Equation: 3.2.7 and Equation: 3.2.10 for exact formulae of these integrals. 
  
   \textit{Calculating \(I_1\): } We obtain \(I_1\) as \footnotesize\(g=\tfrac{\pm 1}{\sqrt{1+C_4^2}\sqrt{|c_0|}}\operatorname{sd}^{-1}\big[\tfrac{y\sqrt{1+C_4^2}}{C_4},\tfrac{1}{\sqrt{1+C_4^2}}\big]+C_3\)\normalsize. Using the identity \(\operatorname{sd}^{-1}\!\big([x/k'],k\big)+\operatorname{cn}^{-1}(x,k)=\operatorname{sn}^{-1}(1,k)\) and absorbing \(\tfrac{\pm 1}{\sqrt{1+C_4^2}\sqrt{|c_0|}}\operatorname{sn}^{-1}(1,k)\) into \(C_3\), we rewrite \(g\) as \footnotesize\(g=\tfrac{\pm 1}{\sqrt{1+C_4^2}\sqrt{|c_0|}}\operatorname{cn}^{-1}\big[y,\tfrac{1}{\sqrt{1+C_4^2}}\big]+C_3\)\normalsize. Substituting back \(y\) gives:
  
  \begin{equation}\label{eqn: soln k<0, c_0 neq 0, g: case 1.1}
  	\boxed{I_1:\qquad g=\tfrac{\pm 1}{\sqrt{1+C_4^2}\sqrt{|c_0|}}\operatorname{cn}^{-1}\big[\operatorname{sech}\!\big[\sqrt{|c_0|}(s+C_2)\big],\tfrac{1}{\sqrt{1+C_4^2}}\big]+C_3}
  \end{equation}

  \textit{Calculating \(I_2\): } We find that \(I_2\) is: \footnotesize\(g=\tfrac{\pm 1}{\sqrt{|c_0|}}\operatorname{nd}^{-1}\big[\tfrac{y}{C_4},\sqrt{1-C_4^2}\big]+C_3\)\normalsize. Using the identity: \(\operatorname{dn}^{-1}\!\big(xk',k\big)+\operatorname{nd}^{-1}(x,k)=\operatorname{sn}^{-1}(1,k)\) and then absorbing \(\tfrac{\pm 1}{\sqrt{|c_0|}}\operatorname{sn}^{-1}(1,k)\) into \(C_3\). We obtain \footnotesize\(g=\tfrac{\pm 1}{\sqrt{|c_0|}}\operatorname{dn}^{-1}\big[y,\sqrt{1-C_4^2}\big]+C_3\)\normalsize. Substituting back \(y\) yeilds:
  
  \begin{equation}\label{eqn: soln k<0, c_0 neq 0, g: case 1.2}
	\boxed{I_2:\qquad g=\tfrac{\pm 1}{\sqrt{|c_0|}}\operatorname{dn}^{-1}\Big[\operatorname{sech}\!\big[\sqrt{|c_0|}(s+C_2)\big],\sqrt{1-C_4^2}\;\Big]+C_3}
 \end{equation}
  
  \noindent
  The conditions required for the integrals to be real‑valued are: for \(I_1\) we need \(C_4>0\), and for \(I_2\) we require \(0<C_4<1\).\\

  \textit{CASE 2:} We have \footnotesize\(H'(s)=\alpha\operatorname{cn}\big[\sqrt{|c_0|}(s+C_2),\gamma\big]\)\normalsize. Substitute this into \(g\) gives: \footnotesize\(g =\pm \displaystyle \int \tfrac{\alpha\operatorname{cn}\big[\sqrt{|c_0|}(s+C_2),\gamma\big]}{\sqrt{\alpha^2\operatorname{cn}^2\big[\sqrt{|c_0|}(s+C_2),\gamma\big]\; \pm\; C^2}}\, dt + C_3\)\normalsize. Using the identity \footnotesize\(\operatorname{cn}(x,k)=\sqrt{1-\operatorname{sn}^2(x,k)}\;\)\normalsize (see \cite{JD} Page: 248 Equation: 12.2.2.2.1) to get: \footnotesize\(g =\pm \displaystyle \int \tfrac{\operatorname{cn}\big[\sqrt{|c_0|}(s+C_2),\gamma\big]}{\sqrt{\frac{\alpha^2 \pm C^2}{\alpha^2}-\operatorname{sn}^2\big[\sqrt{|c_0|}(s+C_2),\gamma\big]}}\, dt + C_3\)\normalsize. Set \footnotesize\(y=\operatorname{sn}\big[\sqrt{|c_0|}(s+C_2),\gamma\big]\)\normalsize; for the integral to be real valued we require \footnotesize\(\frac{\alpha^2 \pm C^2}{\alpha^2}>0\)\normalsize, which we denote by \(C_4^2\).  Using:  \footnotesize\(\operatorname{dn}(x,k)=\sqrt{1-k^2\operatorname{sn}^2(x,k)}\)\normalsize, we find: \footnotesize\(g=\tfrac{\pm 1}{\sqrt{|c_0|}} \displaystyle \int \tfrac{dy}{\sqrt{[1-\gamma^2y^2][C_4^2-y^2]}} + C_3=\tfrac{\pm 1}{\gamma\sqrt{|c_0|}} \displaystyle \int \tfrac{dy}{\sqrt{\big[\tfrac{1}{\gamma^2}-y^2\big][C_4^2-y^2]}} + C_3\)\normalsize. This lead to folloing three cases:
  
  \(I_1: (C_4^2=\tfrac{1}{\gamma^2}) \implies g=\tfrac{\pm 1}{\gamma\sqrt{|c_0|}} \displaystyle \int \tfrac{dy}{\big[\frac{1}{\gamma^2}-y^2\big]} + C_3\)
  
  \(I_2: (C_4^2>\tfrac{1}{\gamma^2}) \implies g=\tfrac{\pm 1}{\gamma\sqrt{|c_0|}} \displaystyle \int \tfrac{dy}{\sqrt{\big[\tfrac{1}{\gamma^2}-y^2\big][C_4^2-y^2]}} + C_3\)
  
  \(I_3: (C_4^2<\tfrac{1}{\gamma^2}) \implies g=\tfrac{\pm 1}{\gamma\sqrt{|c_0|}} \displaystyle \int \tfrac{dy}{\sqrt{\big[\tfrac{1}{\gamma^2}-y^2\big][C_4^2-y^2]}} + C_3\)

 \textit{Calculating \(I_1\): } First note that by definition of \(\gamma\) and \(y\) we have \(\tfrac{1}{\gamma^2}>1\) and \(y^2<1<\tfrac{1}{\gamma^2}\); under this condition we refer to \cite{JD} Page: 152 Equation: 3.5.1.7 for finding the antiderivative, obtaining \(I_1\) as: \footnotesize\(g=\tfrac{\pm 1}{\sqrt{|c_0|}}\tanh^{-1}\big[y\gamma\big]+C_3\)\normalsize. Using the identity \(\tanh^{-1}(x)=\operatorname{sech}^{-1}(\sqrt{1-x^2})\) together with \footnotesize\(y=\operatorname{sn}\big[\sqrt{|c_0|}(s+C_2),\gamma\big]\;\)\normalsize gives the following: \footnotesize\(g=\tfrac{\pm 1}{\sqrt{|c_0|}}\operatorname{sech}^{-1}\big[\sqrt{1-\gamma^2\operatorname{sn}^2\big[\sqrt{|c_0|}(s+C_2),\gamma\big]}\big]+C_3\)\normalsize. Applying the identity \footnotesize\(\operatorname{dn}(x,k)=\sqrt{1-k^2\operatorname{sn}^2(x,k)}\;\)\normalsize once more to get:
 
 \begin{equation}\label{eqn: soln k<0, c_0 neq 0, g: case 2.1}
 	\boxed{I_1:\qquad g=\tfrac{\pm 1}{\sqrt{|c_0|}}\operatorname{sech}^{-1}\big[\operatorname{dn}\big[\sqrt{|c_0|}(s+C_2),\gamma\big]\big]+C_3}
 \end{equation}
 
 \textit{Calculating \(I_2\): } We refer to \cite{La} page 53, Equation: 3.2.6 for this integral, obtaining \(I_2\) as \footnotesize\(g=\tfrac{\pm 1}{\gamma C_4\sqrt{|c_0|}} \operatorname{cd}^{-1}\big[y\gamma,\tfrac{1}{\gamma C_4}\big] + C_3\)\normalsize. Using the identity: \(\operatorname{cd}^{-1}\!\big(x,k\big)+\operatorname{sn}^{-1}(x,k)=\operatorname{sn}^{-1}(1,k)\) (see \cite{La} Page: 54 Equation: 3.2.16) and absorbing \footnotesize\(\tfrac{\pm 1}{\gamma C_4\sqrt{|c_0|}}\operatorname{sn}^{-1}(1,\tfrac{1}{\gamma C_4})\;\)\normalsize into \(C_3\), we obtain: \footnotesize\(g=\tfrac{\pm 1}{\gamma C_4\sqrt{|c_0|}} \operatorname{sn}^{-1}\big[y\gamma,\tfrac{1}{\gamma C_4}\big] + C_3\)\normalsize. Applying the identity \footnotesize\(\operatorname{sn}^{-1}\big[x,k\big]=\operatorname{cn}^{-1}\big[\sqrt{1-x^2},k\big]\)\normalsize, (see \cite{BF} Page: 31 Equation: 131.01) gives \footnotesize\(g=\tfrac{\pm 1}{\gamma C_4\sqrt{|c_0|}} \operatorname{cn}^{-1}\big[1-y^2\gamma^2,\tfrac{1}{\gamma C_4}\big] + C_3\)\normalsize. Substituting back \(y\) and using \footnotesize\(\operatorname{dn}(x,k)=\sqrt{1-k^2\operatorname{sn}^2(x,k)}\;\)\normalsize yeilds:
 
 \begin{equation}\label{eqn: soln k<0, c_0 neq 0, g: case 2.2}
 	\boxed{I_2:\qquad g=\tfrac{\pm 1}{\gamma C_4\sqrt{|c_0|}} \operatorname{cn}^{-1}\big[\operatorname{dn}\big[\sqrt{|c_0|}(s+C_2),\gamma\big],\tfrac{1}{\gamma C_4}\big] + C_3}
 \end{equation}

 \textit{Calculating \(I_3\): } We refer to \cite{La} page 53, Equation: 3.2.6 for this integral, obtaining \(I_2\) as: \footnotesize\(g=\tfrac{\pm 1}{\sqrt{|c_0|}} \operatorname{cd}^{-1}\big[\tfrac{y}{C_4},C_4 \gamma\big] + C_3\)\normalsize. Using the identity \(\operatorname{cd}^{-1}\!\big(x,k\big)+\operatorname{sn}^{-1}(x,k)=\operatorname{sn}^{-1}(1,k)\) (see \cite{La} Page: 54 Equation: 3.2.16) and absorbing the constant \footnotesize\(\tfrac{\pm 1}{\sqrt{|c_0|}}\operatorname{sn}^{-1}(1,C_4 \gamma)\;\)\normalsize into \(C_3\), we obtain \footnotesize\(g=\tfrac{\pm 1}{\sqrt{|c_0|}} \operatorname{sn}^{-1}\big[\tfrac{y}{C_4},C_4 \gamma\big] + C_3\)\normalsize. Applying the identity \footnotesize\(\operatorname{sn}^{-1}\big[x,k\big]=\operatorname{dn}^{-1}\big[\sqrt{1-k^2x^2},k\big]\)\normalsize, (see \cite{BF} Page: 31 Equation: 131.01) gives \footnotesize\(g=\tfrac{\pm 1}{\sqrt{|c_0|}} \operatorname{dn}^{-1}\big[\sqrt{1-\gamma^2 y^2},C_4 \gamma\big] + C_3\)\normalsize. Substituting back \(y\) and using \footnotesize\(\operatorname{dn}(x,k)=\sqrt{1-k^2\operatorname{sn}^2(x,k)}\)\normalsize, we get:
 
 \begin{equation}\label{eqn: soln k<0, c_0 neq 0, g: case 2.3}
 	\boxed{I_3:\qquad g=\tfrac{\pm 1}{\sqrt{|c_0|}} \operatorname{dn}^{-1}\big[\operatorname{dn}\big[\sqrt{|c_0|}(s+C_2),\gamma\big],C_4 \gamma\big] + C_3}
 \end{equation}

 \textit{CASE 3:} We have \footnotesize\(H'(s)=\alpha\operatorname{dn}\!\big[\frac{\alpha\sqrt{|k|}}{2}(s+C_2), \frac{\sqrt{\alpha^2-\beta^2}}{\alpha}\big]\)\normalsize. Substituting this into \(g\) gives \footnotesize\(g =\pm \displaystyle \int \tfrac{\alpha\operatorname{dn}\!\big[\frac{\alpha\sqrt{|k|}}{2}(s+C_2), \frac{\sqrt{\alpha^2-\beta^2}}{\alpha}\big]}{\sqrt{\alpha^2\operatorname{dn}^2\!\big[\frac{\alpha\sqrt{|k|}}{2}(s+C_2), \frac{\sqrt{\alpha^2-\beta^2}}{\alpha}\big]\; \pm\; C^2}}\, dt + C_3\)\normalsize. Using \footnotesize\(\operatorname{dn}(x,k)=1-k^2\operatorname{sn}^2(x,k)\;\)\normalsize and setting \footnotesize\(\gamma=\frac{\sqrt{\alpha^2-\beta^2}}{\alpha}\)\normalsize, we have \footnotesize\(\frac{\alpha\sqrt{|k|}}{2}=\tfrac{\sqrt{|c_0|}}{\gamma}\)\normalsize; thus \footnotesize\(g =\tfrac{\pm 1}{\gamma}\displaystyle \int \tfrac{\operatorname{dn}\!\big[\frac{\sqrt{|c_0|}}{\gamma}(s+C_2), \gamma\big]}{\sqrt{\frac{\alpha^2 \pm C^2}{\gamma^2\alpha^2}-\operatorname{sn}^2\big[\frac{\sqrt{|c_0|}}{\gamma}(s+C_2), \gamma\big]}}\, dt + C_3\)\normalsize. Fix \footnotesize\(y=\operatorname{sn}\big[\frac{\sqrt{|c_0|}}{\gamma}(s+C_2), \gamma\big]\)\normalsize, for this integral to be real valued on some domain we require \footnotesize\(\frac{\alpha^2 \pm C^2}{\gamma^2\alpha^2}>0\)\normalsize, which we denote by \(C_4^2\).  Using \footnotesize\(\operatorname{cn}(x,k)=\sqrt{1-\operatorname{sn}^2(x,k)}\)\normalsize, we obtain: \footnotesize\(g=\tfrac{\pm 1}{\sqrt{|c_0|}} \displaystyle \int \tfrac{dy}{\sqrt{[1-y^2][C_4^2-y^2]}} + C_3\)\normalsize. This lead to folloing three cases:
 
 \(I_1: (C_4^2=1) \implies g=\frac{\pm 1}{\sqrt{|c_0|}} \displaystyle \int \tfrac{dy}{1-y^2} + C_3\)
 
 \(I_2: (C_4^2>1) \implies g=\tfrac{\pm 1}{\sqrt{|c_0|}} \displaystyle \int \tfrac{dy}{\sqrt{[1-y^2][C_4^2-y^2]}} + C_3\)
 
 \(I_3: (C_4^2<1) \implies g=\tfrac{\pm 1}{\sqrt{|c_0|}} \displaystyle \int \tfrac{dy}{\sqrt{[1-y^2][C_4^2-y^2]}} + C_3\)

 \textit{Calculating \(I_1\): } Analogous to \(I_1\) of case 2, we refer to \cite{JD} Page: 152 Equation: 3.5.1.7 for finding this antiderivative, which yeilds \(I_1\) as: \footnotesize\(g=\tfrac{\pm 1}{\sqrt{|c_0|}}\tanh^{-1}\big[y\big]+C_3\)\normalsize. With the identity \(\tanh^{-1}(x)=\operatorname{sech}^{-1}(\sqrt{1-x^2})\) and the substitution \footnotesize\(y=\operatorname{sn}\big[\frac{\sqrt{|c_0|}}{\gamma}(s+C_2),\gamma\big]\)\normalsize, we obtain \footnotesize\(g=\tfrac{\pm 1}{\sqrt{|c_0|}}\operatorname{sech}^{-1}\big[\sqrt{1-\operatorname{sn}^2\big[\frac{\sqrt{|c_0|}}{\gamma}(s+C_2),\gamma\big]}\big]+C_3\)\normalsize. Employing the identity \footnotesize\(\operatorname{cn}(x,k)=\sqrt{1-\operatorname{sn}^2(x,k)}\;\)\normalsize leads to
 
 \begin{equation}\label{eqn: soln k<0, c_0 neq 0, g: case 3.1}
 	\boxed{I_1:\qquad g=\tfrac{\pm 1}{\sqrt{|c_0|}}\operatorname{sech}^{-1}\Big[\operatorname{cn}\big[\tfrac{\sqrt{|c_0|}}{\gamma}(s+C_2),\gamma\big]\Big]+C_3}
 \end{equation}
 
 \textit{Calculating \(I_2\): } We refer \cite{La} page 53, Equation: 3.2.6 for this integral, obtaining \(I_2\) as: \footnotesize\(g=\tfrac{\pm 1}{C_4\sqrt{|c_0|}} \operatorname{cd}^{-1}\big[y,\tfrac{1}{C_4}\big] + C_3\)\normalsize. Using the identity \(\operatorname{cd}^{-1}\!\big(x,k\big)+\operatorname{sn}^{-1}(x,k)=\operatorname{sn}^{-1}(1,k)\) (see \cite{La} Page: 54 Equation: 3.2.16) and absorbing the constant \footnotesize\(\tfrac{\pm 1}{C_4\sqrt{|c_0|}}\operatorname{sn}^{-1}(1,\tfrac{1}{C_4})\;\)\normalsize into \(C_3\), we obtain \footnotesize\(g=\tfrac{\pm 1}{C_4\sqrt{|c_0|}} \operatorname{sn}^{-1}\big[y,\tfrac{1}{C_4}\big] + C_3\)\normalsize. Applying the identity \footnotesize\(\operatorname{sn}^{-1}\big[x,k\big]=\operatorname{cn}^{-1}\big[\sqrt{1-x^2},k\big]\)\normalsize, (see \cite{BF} Page: 31 Equation: 131.01) gives \footnotesize\(g=\tfrac{\pm 1}{ C_4\sqrt{|c_0|}} \operatorname{cn}^{-1}\big[\sqrt{1-y^2},\tfrac{1}{C_4}\big] + C_3\)\normalsize. Substituting back \(y\) and using \footnotesize\(\operatorname{cn}(x,k)=\sqrt{1-\operatorname{sn}^2(x,k)}\;\)\normalsize yeilds: 
 
 \begin{equation}\label{eqn: soln k<0, c_0 neq 0, g: case 3.2}
 	\boxed{I_2:\qquad g=\tfrac{\pm 1}{C_4\sqrt{|c_0|}} \operatorname{cn}^{-1}\big[\operatorname{cn}\big[\tfrac{\sqrt{|c_0|}}{\gamma}(s+C_2),\gamma\big],\tfrac{1}{C_4}\big] + C_3}
 \end{equation}
 
 \textit{Calculating \(I_3\): } The integral to be computed is the same as \(I_2\) except that here \(C_4^2<1\); we obtain \(I_2\) as \footnotesize\(g=\tfrac{\pm 1}{\sqrt{|c_0|}} \operatorname{cd}^{-1}\big[\tfrac{y}{C_4},C_4\big] + C_3\)\normalsize, Using the identity \(\operatorname{cd}^{-1}\!\big(x,k\big)+\operatorname{sn}^{-1}(x,k)=\operatorname{sn}^{-1}(1,k)\) (see \cite{La} Page: 54 Equation: 3.2.16) and absorbing the constant  \footnotesize\(\tfrac{\pm 1}{\sqrt{|c_0|}}\operatorname{sn}^{-1}(1,C_4)\;\)\normalsize into \(C_3\) we get \footnotesize\(g=\tfrac{\pm 1}{\sqrt{|c_0|}} \operatorname{sn}^{-1}\big[\tfrac{y}{C_4},C_4\big] + C_3\)\normalsize. Applying the identity \footnotesize\(\operatorname{sn}^{-1}\big[x,k\big]=\operatorname{dn}^{-1}\big[\sqrt{1-k^2x^2},k\big]\)\normalsize, (see \cite{BF} Page: 31 Equation: 131.01) yeilds \footnotesize\(g=\tfrac{\pm 1}{\sqrt{|c_0|}} \operatorname{dn}^{-1}\big[\sqrt{1-y^2},C_4\big] + C_3\)\normalsize. Substituting back \(y\) and using \footnotesize\(\operatorname{cn}(x,k)=\sqrt{1-\operatorname{sn}^2(x,k)}\;\)\normalsize gives
 
 \begin{equation}\label{eqn: soln k<0, c_0 neq 0, g: case 3.3}
 	\boxed{I_3:\qquad g=\tfrac{\pm 1}{\sqrt{|c_0|}} \operatorname{dn}^{-1}\big[\operatorname{cn}\big[\tfrac{\sqrt{|c_0|}}{\gamma}(s+C_2),\gamma\big],C_4\big] + C_3}
 \end{equation}
 
  \textbf{Finding Transformed Minimal Surfaces:} We compute the transformed minimal surfaces by composing each minimal graph surface with its corresponding minimal graph transformation.  
  
  \textit{CASE 1:} For the minimal graph surface given by Equation~\ref{eqn: soln k<0, c_0 neq 0, f: case 1}, the corresponding minimal graph transformations are given by \ref{eqn: soln k<0, c_0 neq 0, g: case 1.1} and \ref{eqn: soln k<0, c_0 neq 0, g: case 1.2}.  
  
  \noindent
  Applying \ref{eqn: soln k<0, c_0 neq 0, g: case 1.1} yields the first transformed minimal surface in this case as 
   
  \(
  g\circ f=\pm\tfrac{\Gamma}{\sqrt{|c_0|}}\operatorname{cn}^{-1}\!\Big[\tfrac{1}{\sqrt{2}}\sqrt{1 + \cosh\!\big(\ln\!\big[\tanh(\tfrac{\sqrt c_0(z+c_1)}{2})\coth(\tfrac{\sqrt{\bar c_0}(\bar z+\bar c_1)}{2})\big]\big)},\Gamma\Big]-C_2 
  \) 
  
  \noindent
  with \(\Gamma=\tfrac{1}{\sqrt{1+C_4^2}}\); this is exactly of the form \ref{eqn: soln k<0, c_0 neq 0, f: case 3}. Applying \ref{eqn: soln k<0, c_0 neq 0, g: case 1.2} gives the second transformed minimal surface in this case as 
   
  \(
  g\circ f=\tfrac{\pm1}{\sqrt{|c_0|}}\operatorname{dn}^{-1}\!\Big[\tfrac{1}{\sqrt{2}}\sqrt{1 + \cosh\!\big(\ln\!\big[\tanh(\tfrac{\sqrt c_0(z+c_1)}{2})\coth(\tfrac{\sqrt{\bar c_0}(\bar z+\bar c_1)}{2})\big]\big)},\Gamma\Big]-C_2
  \)  
  
  \noindent
  with \(\Gamma=\sqrt{1-C_4^2}\); this is exactly of the form \ref{eqn: soln k<0, c_0 neq 0, f: case 2}. \\ 
  
  \textit{CASE 2:} For the minimal graph surface given by Equation~\ref{eqn: soln k<0, c_0 neq 0, f: case 2}, the corresponding minimal graph transformations are given by \ref{eqn: soln k<0, c_0 neq 0, g: case 2.1}, \ref{eqn: soln k<0, c_0 neq 0, g: case 2.2} and \ref{eqn: soln k<0, c_0 neq 0, g: case 2.3}.   Applying \ref{eqn: soln k<0, c_0 neq 0, g: case 2.1} yields  
  
  \(
  g\circ f=\pm\tfrac{1}{\sqrt{|c_0|}}\operatorname{sech}^{-1}\!\Big[\tfrac{1}{\sqrt{2}}\sqrt{1 + \cosh\!\big(\ln\!\big[\tanh(\tfrac{\sqrt c_0(z+c_1)}{2})\coth(\tfrac{\sqrt{\bar c_0}(\bar z+\bar c_1)}{2})\big]\big)}\;\Big]-C_2 
  \)  
  
  \noindent
  which is exactly of the form \ref{eqn: soln k<0, c_0 neq 0, f: case 1}.  Applying \ref{eqn: soln k<0, c_0 neq 0, g: case 2.2} gives  
  
  \(
  g\circ f=\pm\tfrac{\Gamma}{\sqrt{|c_0|}}\operatorname{cn}^{-1}\!\Big[\tfrac{1}{\sqrt{2}}\sqrt{1 + \cosh\!\big(\ln\!\big[\tanh(\tfrac{\sqrt c_0(z+c_1)}{2})\coth(\tfrac{\sqrt{\bar c_0}(\bar z+\bar c_1)}{2})\big]\big)},\Gamma\Big]-C_2 
  \)  
  
  \noindent
  with \(\Gamma=\tfrac{1}{\gamma C_4}\); this is exactly of the form \ref{eqn: soln k<0, c_0 neq 0, f: case 3}.  Applying \ref{eqn: soln k<0, c_0 neq 0, g: case 2.3} yields  
  
  \(
  g\circ f=\tfrac{\pm1}{\sqrt{|c_0|}}\operatorname{dn}^{-1}\!\Big[\tfrac{1}{\sqrt{2}}\sqrt{1 + \cosh\!\big(\ln\!\big[\tanh(\tfrac{\sqrt c_0(z+c_1)}{2})\coth(\tfrac{\sqrt{\bar c_0}(\bar z+\bar c_1)}{2})\big]\big)},\Gamma\Big]-C_2
  \) 
  
  \noindent
  with \(\Gamma=C_4\gamma\); this is exactly of the form \ref{eqn: soln k<0, c_0 neq 0, f: case 2}.  \\
  
  \textit{CASE 3:} For the minimal graph surface given by Equation~\ref{eqn: soln k<0, c_0 neq 0, f: case 3}, the corresponding minimal graph transformations are given by \ref{eqn: soln k<0, c_0 neq 0, g: case 3.1}, \ref{eqn: soln k<0, c_0 neq 0, g: case 3.2} and \ref{eqn: soln k<0, c_0 neq 0, g: case 3.3}.  
  
  \noindent
  Applying \ref{eqn: soln k<0, c_0 neq 0, g: case 3.1} yields  
  
  \(
  g\circ f=\pm\tfrac{1}{\sqrt{|c_0|}}\operatorname{sech}^{-1}\!\Big[\tfrac{1}{\sqrt{2}}\sqrt{1 + \cosh\!\big(\ln\!\big[\tanh(\tfrac{\sqrt c_0(z+c_1)}{2})\coth(\tfrac{\sqrt{\bar c_0}(\bar z+\bar c_1)}{2})\big]\big)}\;\Big]-C_2 
  \)  
  
  \noindent
  which is exactly of the form \ref{eqn: soln k<0, c_0 neq 0, f: case 1}.   Applying \ref{eqn: soln k<0, c_0 neq 0, g: case 3.2} gives  
  
  \(
  g\circ f=\pm\tfrac{\Gamma}{\sqrt{|c_0|}}\operatorname{cn}^{-1}\!\Big[\tfrac{1}{\sqrt{2}}\sqrt{1 + \cosh\!\big(\ln\!\big[\tanh(\tfrac{\sqrt c_0(z+c_1)}{2})\coth(\tfrac{\sqrt{\bar c_0}(\bar z+\bar c_1)}{2})\big]\big)},\Gamma\Big]-C_2 
  \)
  
  \noindent
  with \(\Gamma=\tfrac{1}{C_4}\); this is exactly of the form \ref{eqn: soln k<0, c_0 neq 0, f: case 3}.   Applying \ref{eqn: soln k<0, c_0 neq 0, g: case 3.3} yields  
  
  \(
  g\circ f=\tfrac{\pm1}{\sqrt{|c_0|}}\operatorname{dn}^{-1}\!\Big[\tfrac{1}{\sqrt{2}}\sqrt{1 + \cosh\!\big(\ln\!\big[\tanh(\tfrac{\sqrt c_0(z+c_1)}{2})\coth(\tfrac{\sqrt{\bar c_0}(\bar z+\bar c_1)}{2})\big]\big)},\Gamma\Big]-C_2
  \)
  
  \noindent
  with \(\Gamma=C_4\); this is exactly of the form \ref{eqn: soln k<0, c_0 neq 0, f: case 2}.
  
  \begin{remark}
  	The minimal graph surface presented in the above subsection (see Equations: \ref{eqn: soln k<0, c_0 neq 0, f: case 1}, \ref{eqn: soln k<0, c_0 neq 0, f: case 2}, \ref{eqn: soln k<0, c_0 neq 0, f: case 3} ) constitutes, to the best of the author's knowledge, again a new class of minimal surfaces.
  \end{remark}
  
  \section{Further Discussions}
  
  We have thus given a complete classification of all minimal graph surfaces \(f\) that can be composed with a smooth real‑valued function \(g\) such that \(g\circ f\) is again a minimal graph surface. Consequently, we have also classified those minimal graph transformations and the transformed minimal surfaces corresponding to each of them. Next we assign suitable names to the minimal surfaces described in Equations \ref{eqn: soln k<0, c_0 neq 0, f: case 1}, \ref{eqn: soln k<0, c_0 neq 0, f: case 2}, \ref{eqn: soln k<0, c_0 neq 0, f: case 3} and \ref{eqn: soln k>0, c_0 neq 0, f}.
  
\begin{enumerate}
	\item Because of its resemblance to long columns of pillars, we name the minimal graph surface obtained in the case \(k>0,\; c_0\neq 0\) and given by equation \ref{eqn: soln k>0, c_0 neq 0, f} the \textbf{Pillars}.
	\item Because of its similarity to an infinitely high wall, we name the minimal graph surface obtained in the case \(k<0,\; c_0\neq 0,\; C_1-\tfrac{2|c_0|}{k}=0\) and given by equation \ref{eqn: soln k<0, c_0 neq 0, f: case 1} the \textbf{Great Wall}.
	\item Because of its resemblance to a wall of moderate height but significant thickness, we name the minimal graph surface obtained in the case \(k<0,\; c_0\neq 0,\; C_1-\tfrac{2|c_0|}{k}>0\) and given by equation \ref{eqn: soln k<0, c_0 neq 0, f: case 2} the \textbf{Thick Wall}.
	\item Because of its similarity to a wall of moderate height but with a sharp edge, we name the minimal graph surface obtained in the case \(k<0,\; c_0\neq 0,\; C_1-\tfrac{2|c_0|}{k}<0\) and given by equation \ref{eqn: soln k<0, c_0 neq 0, f: case 3} the \textbf{Sharp Wall}.
\end{enumerate}

 The flowchart given in Figure~\ref{pic: flow chart} briefly summarises all the work done up to this section in the paper.
  
 \subsection{Singular Point Analysis and Domain of Well Definedness}
  In this section we analyse the singular points of the minimal graph surfaces given by equations \ref{eqn: soln k<0, c_0 neq 0, f: case 1}, \ref{eqn: soln k<0, c_0 neq 0, f: case 2}, \ref{eqn: soln k<0, c_0 neq 0, f: case 3} and \ref{eqn: soln k>0, c_0 neq 0, f}, and determine the domains on which these surfaces are defined.
  
  \subsubsection{Pillars}
  The Pillars is given by \footnotesize\(f=\tfrac{\pm \gamma'}{\sqrt{|c_0|}}\operatorname{cn}^{-1}\!\big[\tfrac{\gamma'}{\gamma}\sqrt{\cosh\!\big[2\ln\!|\tanh(\tfrac{\sqrt{c_0}(z + c_1)}{2})|\big]-1}\;,\;\gamma\big]-C_2\)\normalsize; from the derivative formula \(\tfrac{d}{dy}\operatorname{cn}^{-1}(y,k)=-\tfrac{1}{\sqrt{[1-y^2][k'^2+k^2y^2]}}\) (see \cite{JD}, Page: 250, Equation: 12.4.1.1.4) it follows that \(y=1\) is a singular point, which in our setting corresponds to the singular curve \(\tfrac{\gamma'}{\gamma}\sqrt{\cosh\!\big[2\ln\!|\tanh(\tfrac{\sqrt{c_0}(z + c_1)}{2})|\big]-1}=1\); additionally the square‑root in the expression for \(f\) produces a singularity when its argument vanishes, giving the second singular curve \(\cosh\!\big[2\ln\!|\tanh(\tfrac{\sqrt{c_0}(z + c_1)}{2})|\big]-1=0\). Simplifying the first curve using \(\cosh[\ln(x)]=\frac{x+\frac{1}{x}}{2}\) yields \(|\tanh(\tfrac{\sqrt{c_0}(z + c_1)}{2})|^2+\frac{1}{|\tanh(\frac{\sqrt{c_0}(z + c_1)}{2})|^2}=\tfrac{2}{\gamma'^2}\); denoting the right‑hand side by \(r\) and solving gives \(|\tanh(\tfrac{\sqrt{c_0}(z + c_1)}{2})|=\frac{[\sqrt{r+2}\;]\pm[\sqrt{r-2}\;]}{2}\). The second curve simplifies directly to \(|\tanh(\tfrac{\sqrt{c_0}(z + c_1)}{2})|=1\).
    \begin{figure}[htbp]
  	\centering
  	\begin{subfigure}{0.3\textwidth}
  		\centering
  		\includegraphics[width=\linewidth]{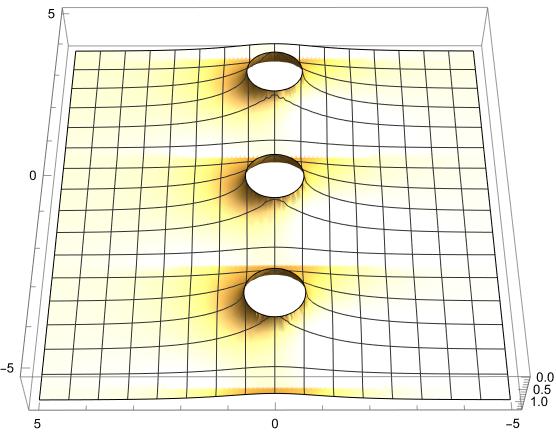}
  		\caption{\footnotesize Pillars \normalsize}
  	\end{subfigure}
  	\hfill
  	\begin{subfigure}{0.24\textwidth}
  		\centering
  		\includegraphics[width=\linewidth]{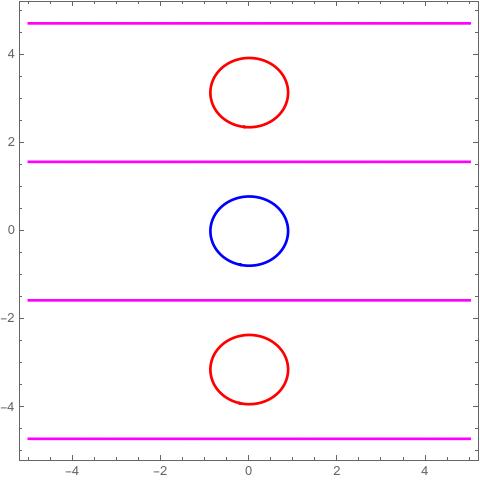}
  		\caption{\footnotesize Singular curves \normalsize}
  	\end{subfigure}
    \hfill
    \begin{subfigure}{0.25\textwidth}
   	    \centering
    	\includegraphics[width=\linewidth]{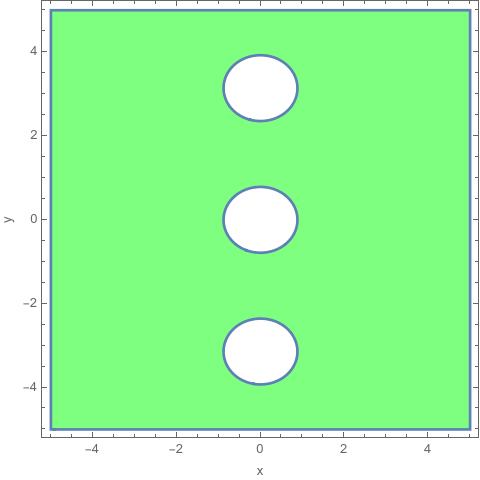}
   	    \caption{\footnotesize Domain \normalsize}
   \end{subfigure}
    \caption{\footnotesize The images above compare the Pillars surface with a contour plot of its singularities and the domain where it is well defined. The magenta line depicts the singular curve \(|\tanh(\tfrac{\sqrt{c_0}(z + c_1)}{2})|=1\); the red and blue lines represent the singular curves \(|\tanh(\tfrac{\sqrt{c_0}(z + c_1)}{2})|=\frac{[\sqrt{r+2}\,]\,+\,\epsilon[\sqrt{r-2}\,]}{2}\) with \(\epsilon = +1\) and \(\epsilon = -1\), respectively. The region shaded in green indicates the domain on which the minimal graph surface is well defined. The parameters are chosen exactly same as Figure \ref{pic: Pillars} \normalsize}
    \label{pic: singular point analysis Pillars}
   \end{figure}

  We note that these are the only singular curves for the Pillars surface; this is also clearly visible in Figure~\ref{pic: singular point analysis Pillars}. On these curves \(f_z\) and \(f_{\bar z}\) blow up to infinity, so \(\|\nabla f\|^2\) also diverges to infinity, which in turn forces the first characteristic function \(h(f)=\tfrac{\Delta f}{\|\nabla f\|^2}\) to vanish along these singular curves. Consequently, these singular curves are precisely the curves that must be excluded in the choice of domain as described in Section~\ref{section: assumptions}.

 To determine the domain on which the surface is well defined we observe that the range of the elliptic function \(\operatorname{cn}\) is the interval \([-1,1]\); thus we must have the inequality \(-1\leq\tfrac{\gamma'}{\gamma}\sqrt{\cosh\!\big[2\ln\!|\tanh(\tfrac{\sqrt{c_0}(z + c_1)}{2})|\big]-1}\leq 1\). Squaring and simplifying yields following inequality: \(|\tanh(\tfrac{\sqrt{c_0}(z + c_1)}{2})|^2+\frac{1}{|\tanh(\frac{\sqrt{c_0}(z + c_1)}{2})|^2}\leq \tfrac{2}{\gamma'^2}\). The region satisfying this inequality is shaded green in Figure~\ref{pic: singular point analysis Pillars}.

 \subsubsection{Great Wall}
  We have \footnotesize\(f=\pm\tfrac{1}{\sqrt{|c_0|}}\operatorname{sech}^{-1}\Big[\tfrac{1}{\sqrt{2}}\sqrt{1 + \cosh\big(\ln\big[\tanh(\tfrac{\sqrt c_0(z+c_1)}{2})\coth(\tfrac{\sqrt{\bar c_0}(\bar z+\bar c_1)}{2})\big]\big)}\;\Big]-C_2\)\normalsize; from the derivative formula \(\tfrac{d}{dy}\operatorname{sech}^{-1}(y)=-\tfrac{1}{y\sqrt{1-y^2}}\) (see \cite{JD}, Page: 152, Equation: 3.5.1.11) it follows that \(y=1\) and \(y=0\) are singular points, which in our setting correspond to the following singular curves \(\tfrac{1}{\sqrt{2}}\sqrt{1 + \cosh\!\big(\ln\!\big[\tanh(\tfrac{\sqrt c_0(z+c_1)}{2})\coth(\tfrac{\sqrt{\bar c_0}(\bar z+\bar c_1)}{2})\big]\big)}=1\) and \(\tfrac{1}{\sqrt{2}}\sqrt{1 + \cosh\!\big(\ln\!\big[\tanh(\tfrac{\sqrt c_0(z+c_1)}{2})\coth(\tfrac{\sqrt{\bar c_0}(\bar z+\bar c_1)}{2})\big]\big)}=0\). The square‑root in the expression for \(f\) also introduces a singularity when its argument vanishes, but this coincides with the second equation. Using the identity: \(\cosh[\ln(x)]=\frac{x+\frac{1}{x}}{2}\), the first curve simplifies to the equation: \(\Re\!\big[\tanh(\tfrac{\sqrt c_0(z+c_1)}{2})\coth(\tfrac{\sqrt{\bar c_0}(\bar z+\bar c_1)}{2})\big]=1\), and the second curve simplifies to \(\Re\!\big[\tanh(\tfrac{\sqrt c_0(z+c_1)}{2})\coth(\tfrac{\sqrt{\bar c_0}(\bar z+\bar c_1)}{2})\big]=-1\). These curves are plotted below.
 \begin{figure}[htbp]
 	\centering
 	\begin{subfigure}{0.26\textwidth}
 		\centering
 		\includegraphics[width=\linewidth]{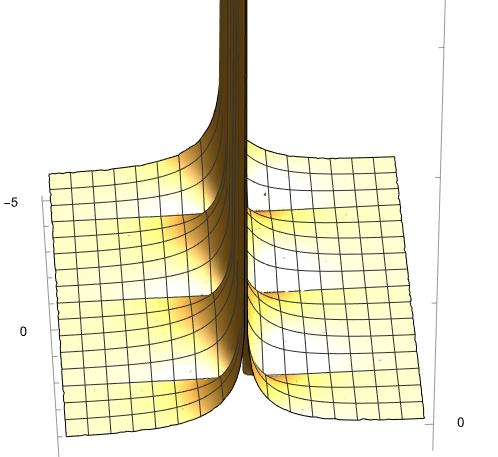}
 		\caption{\footnotesize Great Wall \normalsize}
 	\end{subfigure}
 	\hfill
 	\begin{subfigure}{0.24\textwidth}
 		\centering
 		\includegraphics[width=\linewidth]{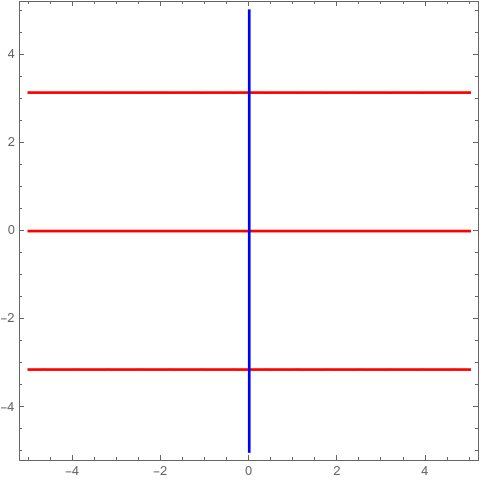}
 		\caption{\footnotesize Singular curves \normalsize}
 	\end{subfigure}
 	\hfill
 	\begin{subfigure}{0.25\textwidth}
 		\centering
 		\includegraphics[width=\linewidth]{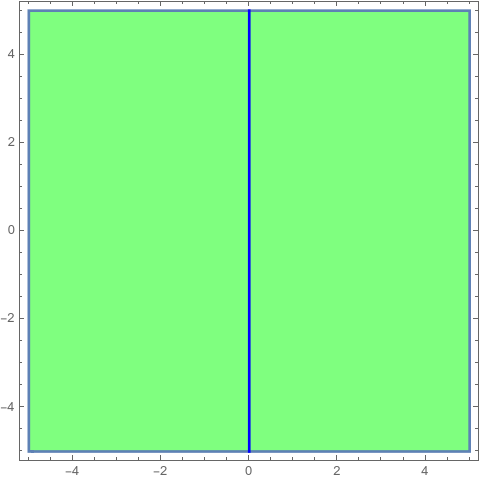}
 		\caption{\footnotesize Domain \normalsize}
 	\end{subfigure}
 	\caption{\footnotesize The images above compare the Great Wall surface with a contour plot of its singularities and the domain where it is well defined. The red lines depict the singular curve \(\Re\!\big[\tanh(\tfrac{\sqrt c_0(z+c_1)}{2})\coth(\tfrac{\sqrt{\bar c_0}(\bar z+\bar c_1)}{2})\big]=1\); the blue lines represent the singular curve \(\Re\!\big[\tanh(\tfrac{\sqrt c_0(z+c_1)}{2})\coth(\tfrac{\sqrt{\bar c_0}(\bar z+\bar c_1)}{2})\big]=-1\). The region shaded in green indicates the domain on which the minimal graph surface is well defined and on the blue line the minimal graph surface blows up to infinity. The parameters are chosen exactly same as Figure \ref{pic: Great Wall} \normalsize}
 	\label{pic: singular point analysis Great Wall}
 \end{figure}
 
 We note that these are the only singular curves for the Great Wall surface; this is also clearly visible in Figure~\ref{pic: singular point analysis Great Wall}. On these curves \(f_z\) and \(f_{\bar z}\) blow up to infinity, so \(\|\nabla f\|^2\) also diverges to infinity, which in turn forces the first characteristic function \(h(f)=\tfrac{\Delta f}{\|\nabla f\|^2}\) to vanish along these singular curves. Consequently, these singular curves are precisely the curves that must be excluded in the choice of domain as described in Section~\ref{section: assumptions}.
 
 To determine the domain on which the surface is well defined we observe that the range of the hyperbolic function \(\operatorname{sech}\) is the interval \((0,1]\); thus we must have the inequality \(0<\tfrac{1}{\sqrt{2}}\sqrt{1 + \cosh\big(\ln\big[\tanh(\tfrac{\sqrt c_0(z+c_1)}{2})\coth(\tfrac{\sqrt{\bar c_0}(\bar z+\bar c_1)}{2})\big]\big)}\leq 1\). Squaring and simplifying the inequality yeilds: \(-1<\Re\!\big[\tanh(\tfrac{\sqrt c_0(z+c_1)}{2})\coth(\tfrac{\sqrt{\bar c_0}(\bar z+\bar c_1)}{2})\big]\leq 1\). The region satisfying this inequality is shaded green in Figure~\ref{pic: singular point analysis Great Wall}.
  
  \subsubsection{Thick Wall}
  We have \footnotesize\(f=\tfrac{\pm1}{\sqrt{|c_0|}}\operatorname{dn}^{-1}\Big[\tfrac{1}{\sqrt{2}}\sqrt{1 + \cosh\big(\ln\big[\tanh(\tfrac{\sqrt c_0(z+c_1)}{2})\coth(\tfrac{\sqrt{\bar c_0}(\bar z+\bar c_1)}{2})\big]\big)},\gamma\Big]-C_2\)\normalsize. The derivative formula \(\tfrac{d}{dy}\operatorname{dn}^{-1}(y,k)=-\tfrac{1}{\sqrt{[1-y^2][y^2-k'^2]}}\) (see \cite{JD}, Page: 250, Equation: 12.4.1.1.5) shows that \(y=1\) and \(y=k'\) are singular points; in our context these give the singular curves, given by the equation \(\tfrac{1}{\sqrt{2}}\sqrt{1 + \cosh\!\big(\ln\!\big[\tanh(\tfrac{\sqrt c_0(z+c_1)}{2})\coth(\tfrac{\sqrt{\bar c_0}(\bar z+\bar c_1)}{2})\big]\big)}=1\) and \(\tfrac{1}{\sqrt{2}}\sqrt{1 + \cosh\!\big(\ln\!\big[\tanh(\tfrac{\sqrt c_0(z+c_1)}{2})\coth(\tfrac{\sqrt{\bar c_0}(\bar z+\bar c_1)}{2})\big]\big)}=\gamma'\). A further potential singularity could arise if the argument of the outer square‑root vanished, but this never happens because it would lead to the negative term \(-\gamma'^2\) inside the square‑root appearing in the derivative of \(\operatorname{dn}^{-1}\). Using \(\cosh[\ln(x)]=\frac{x+\frac{1}{x}}{2}\), the first curve simplifies to the equation \(\Re\!\big[\tanh(\tfrac{\sqrt c_0(z+c_1)}{2})\coth(\tfrac{\sqrt{\bar c_0}(\bar z+\bar c_1)}{2})\big]=1\) and the second to the equation \(\Re\!\big[\tanh(\tfrac{\sqrt c_0(z+c_1)}{2})\coth(\tfrac{\sqrt{\bar c_0}(\bar z+\bar c_1)}{2})\big]=2\gamma'^2-1\). These curves are plotted below.
  \begin{figure}[htbp]
  	\centering
  	\begin{subfigure}{0.25\textwidth}
  		\centering
  		\includegraphics[width=\linewidth]{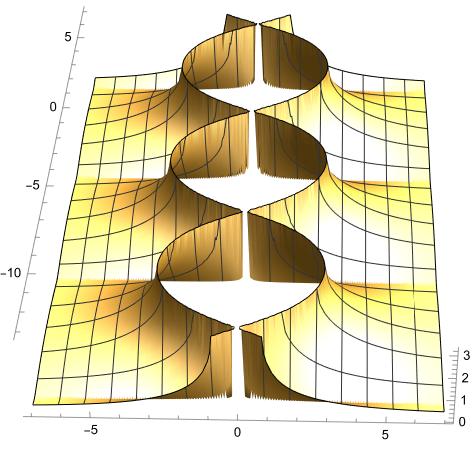}
  		\caption{\footnotesize Thick Wall \normalsize}
  	\end{subfigure}
  	\hfill
  	\begin{subfigure}{0.25\textwidth}
  		\centering
  		\includegraphics[width=\linewidth]{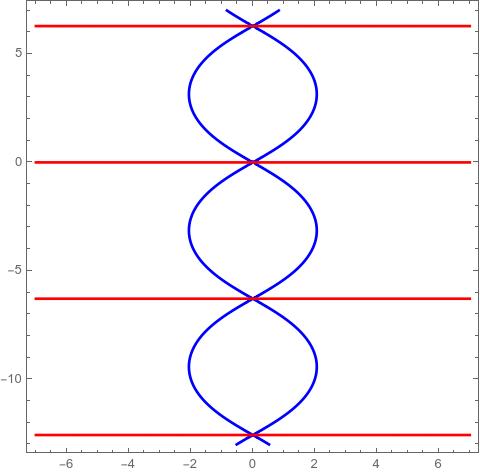}
  		\caption{\footnotesize Singular curves \normalsize}
  	\end{subfigure}
  	\hfill
  	\begin{subfigure}{0.26\textwidth}
  		\centering
  		\includegraphics[width=\linewidth]{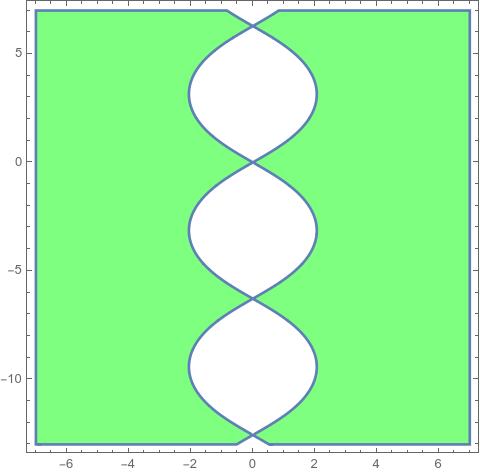}
  		\caption{\footnotesize Domain \normalsize}
  	\end{subfigure}
  	\caption{\footnotesize The images above compare the Thick Wall surface with a contour plot of its singularities and the domain where it is well defined. The red lines depict the singular curve \(\Re\!\big[\tanh(\tfrac{\sqrt c_0(z+c_1)}{2})\coth(\tfrac{\sqrt{\bar c_0}(\bar z+\bar c_1)}{2})\big]=1\); the blue lines represent the singular curve \(\Re\!\big[\tanh(\tfrac{\sqrt c_0(z+c_1)}{2})\coth(\tfrac{\sqrt{\bar c_0}(\bar z+\bar c_1)}{2})\big]=2\gamma'^2-1\). The region shaded in green indicates the domain on which the minimal graph surface is well defined. The parameters are chosen exactly as in Figure~\ref{pic: Thick Wall}. \normalsize}
  	\label{pic: singular point analysis Thick Wall}
  \end{figure}
  
  We note that these are the only singular curves for the Thick Wall surface; this is also clearly visible in Figure~\ref{pic: singular point analysis Thick Wall}. On these curves \(f_z\) and \(f_{\bar z}\) blow up to infinity, so \(\|\nabla f\|^2\) also diverges to infinity, which in turn forces the first characteristic function \(h(f)=\tfrac{\Delta f}{\|\nabla f\|^2}\) to vanish along these singular curves. Consequently, these singular curves are precisely the curves that must be excluded in the choice of domain as described in Section~\ref{section: assumptions}.
  
  To determine the domain on which the surface is well defined we note that the elliptic function \(\operatorname{dn}\) takes values in the interval \([\gamma',1]\); thus we must have the inequality \(\gamma'\leq\tfrac{1}{\sqrt{2}}\sqrt{1 + \cosh\!\big(\ln\!\big[\tanh(\tfrac{\sqrt c_0(z+c_1)}{2})\coth(\tfrac{\sqrt{\bar c_0}(\bar z+\bar c_1)}{2})\big]\big)}\leq 1\), which after squaring and simplifying yields the condition \(2\gamma'^2-1<\Re\!\big[\tanh(\tfrac{\sqrt c_0(z+c_1)}{2})\coth(\tfrac{\sqrt{\bar c_0}(\bar z+\bar c_1)}{2})\big]\leq 1\); the region satisfying this inequality is shaded green in Figure~\ref{pic: singular point analysis Thick Wall}.
  
  \subsubsection{Sharp Wall}
  The Sharp Wall is given by \footnotesize\(f=\pm\tfrac{\gamma}{\sqrt{|c_0|}}\operatorname{cn}^{-1}\Big[\tfrac{1}{\sqrt{2}}\sqrt{1 + \cosh\big(\ln\big[\tanh(\tfrac{\sqrt c_0(z+c_1)}{2})\coth(\tfrac{\sqrt{\bar c_0}(\bar z+\bar c_1)}{2})\big]\big)},\gamma\Big]-C_2\)\normalsize. From the derivative \(\tfrac{d}{dy}\operatorname{cn}^{-1}(y,k)=-\tfrac{1}{\sqrt{[1-y^2][k'^2+k^2y^2]}}\) (see \cite{JD}, Page: 250, Equation: 12.4.1.1.4), we find that the point \(y=1\) is singular; this in our case gives the curve \(\tfrac{1}{\sqrt{2}}\sqrt{1 + \cosh\!\big(\ln\!\big[\tanh(\tfrac{\sqrt c_0(z+c_1)}{2})\coth(\tfrac{\sqrt{\bar c_0}(\bar z+\bar c_1)}{2})\big]\big)}=1\). A second potential singularity occurs when argument of the squareroot inside the expression for \(f\) vanishes, that is, \(1 + \cosh\!\big(\ln\!\big[\tanh(\tfrac{\sqrt c_0(z+c_1)}{2})\coth(\tfrac{\sqrt{\bar c_0}(\bar z+\bar c_1)}{2})\big]\big)=0\). Using \(\cosh[\ln(x)]=\frac{x+\frac{1}{x}}{2}\), the first condition reduces to \(\Re\!\big[\tanh(\tfrac{\sqrt c_0(z+c_1)}{2})\coth(\tfrac{\sqrt{\bar c_0}(\bar z+\bar c_1)}{2})\big]=1\) and the second to \(\Re\!\big[\tanh(\tfrac{\sqrt c_0(z+c_1)}{2})\coth(\tfrac{\sqrt{\bar c_0}(\bar z+\bar c_1)}{2})\big]=-1\). These curves are plotted in Figure: \ref{pic: singular point analysis Sharp Wall}
  \begin{figure}[htbp]
  	\centering
  	\begin{subfigure}{0.29\textwidth}
  		\centering
  		\includegraphics[width=\linewidth]{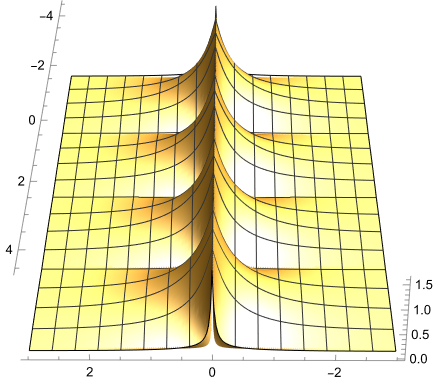}
  		\caption{\footnotesize Sharp Wall \normalsize}
  	\end{subfigure}
  	\hfill
  	\begin{subfigure}{0.25\textwidth}
  		\centering
  		\includegraphics[width=\linewidth]{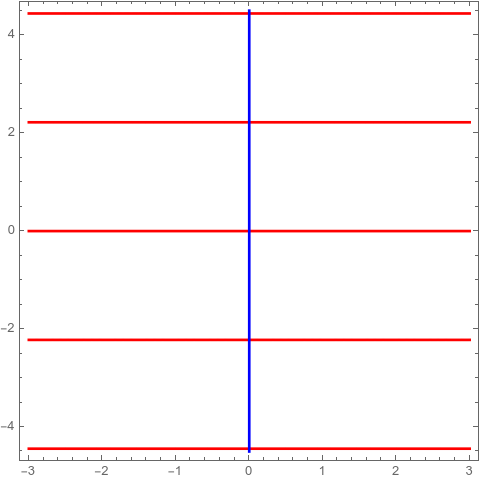}
  		\caption{\footnotesize Singular curves  \normalsize}
  	\end{subfigure}
  	\hfill
  	\begin{subfigure}{0.26\textwidth}
  		\centering
  		\includegraphics[width=\linewidth]{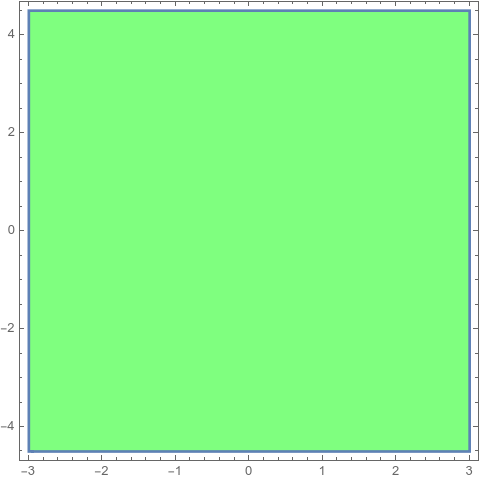}
  		\caption{\footnotesize Domain  \normalsize}
  	\end{subfigure}
  	\caption{\footnotesize The images above compare the Sharp Wall surface with a contour plot of its singularities and the domain where it is well defined. The red lines depict the singular curve \(\Re\!\big[\tanh(\tfrac{\sqrt c_0(z+c_1)}{2})\coth(\tfrac{\sqrt{\bar c_0}(\bar z+\bar c_1)}{2})\big]=1\); the blue lines represent the singular curve \(\Re\!\big[\tanh(\tfrac{\sqrt c_0(z+c_1)}{2})\coth(\tfrac{\sqrt{\bar c_0}(\bar z+\bar c_1)}{2})\big]=-1\). The region shaded in green indicates the domain on which the minimal graph surface is well defined. The parameters are chosen exactly as in Figure~\ref{pic: Sharp Wall}. \normalsize}
  	\label{pic: singular point analysis Sharp Wall}
  \end{figure}
  
  We note that these are the only singular curves for the Sharp Wall surface; this is also clearly visible in Figure~\ref{pic: singular point analysis Sharp Wall}. On these curves \(f_z\) and \(f_{\bar z}\) blow up to infinity, so \(\|\nabla f\|^2\) also diverges to infinity, which in turn forces the first characteristic function \(h(f)=\tfrac{\Delta f}{\|\nabla f\|^2}\) to vanish along these singular curves. Consequently, these singular curves are precisely the curves that must be excluded in the choice of domain as described in Section~\ref{section: assumptions}.
  
  To determine the domain on which the surface is well defined we note that the elliptic function \(\operatorname{cn}\) takes values in the interval \([-1,1]\); thus we must have the inequality \(-1\leq\tfrac{1}{\sqrt{2}}\sqrt{1 + \cosh\!\big(\ln\!\big[\tanh(\tfrac{\sqrt c_0(z+c_1)}{2})\coth(\tfrac{\sqrt{\bar c_0}(\bar z+\bar c_1)}{2})\big]\big)}\leq 1\), which after squaring and simplifying yields the condition \(-1\leq\Re\!\big[\tanh(\tfrac{\sqrt c_0(z+c_1)}{2})\coth(\tfrac{\sqrt{\bar c_0}(\bar z+\bar c_1)}{2})\big]\leq 1\); the region satisfying this inequality is infact the whole \(\mathbb{R}^2\) and hence this minimal surface is defined every where on \(\mathbb{R}^2\).
  
  \section{Future Directions}
  In this section we present several applications of the techniques developed for solving the minimal graph transformation problem.
  
  \subsection{Minimal-Maximal Graph Transformations}
  Let \(f:\Omega\longrightarrow\mathbb{R}\) be a minimal graph surface and \(g:I=f(\Omega)\longrightarrow\mathbb{R}\) a smooth function. We say that \(g\) is a \textbf{Minimal–Maximal Graph Transformation} of \(f\) if \(g\circ f\) is a maximal graph surface. Analogous to Proposition~\ref{thm: Alt. def. MGT} and Equation~\ref{eqn: Alt. def. MGT}, this condition can be expressed as a mathematical equation
  \[
  g''(f)\,\|\nabla f\|^2+[g'^{3}+g'](f)\,\Delta f=0 .
  \]
  After considering the analyticity of minimal and maximal surfaces and excluding certain singular cases, the system reduces to:
  \[
  \frac{\Delta f}{\|\nabla f\|^2}=\frac{-g''(f)}{[g'^{3}+g'](f)}=h(f), \qquad\text{and}\qquad \text{Minimal Surface Equation: }\text{Equation }\ref{eqn: MSE},
  \]
  for some smooth function \(h\). This system is identical to the one obtained for minimal graph transformations, except that here \(h=\frac{-g''}{g'^{3}+g'}\) whereas in the minimal graph transformation we had the \(h=\frac{g''}{g'^{3}-g'}\). Consequently, the techniques developed for solving the minimal graph transformation problem apply directly; the only remaining step is to compute the minimal–maximal graph transformation by solving the ODE \(\frac{-g''(s)}{[g'^{3}+g'](s)}=h(s)\). Hence, the minimal surfaces that admit a minimal–maximal graph transformation are precisely those minimal graph surfaces that admit a minimal graph transformation, and they have been exhaustively classified in this paper. Since we already possess an exhaustive list of first characteristic functions \(h\) corresponding to these surfaces, the task reduces to solving the ODE \(\frac{-g''(s)}{[g'^{3}+g'](s)}=h(s)\) for each such \(h\); the resulting \(g\) gives the minimal–maximal graph transformation, and \(g\circ f\) is the associated maximal graph surface. Indeed, one might be able to derive a general formula analogous to Equation: \ref{eqn: conv h to g} from which \(g\) can be obtained directly for any admissible \(h\).
  
  \subsection{Maximal-Minimal Graph Transformations}
  In a similar way, one can define a Maximal–Minimal Graph Transformation as follows: Let \(f:\Omega\longrightarrow\mathbb{R}\) be a maximal graph surface and \(g:I=f(\Omega)\longrightarrow\mathbb{R}\) a smooth function. We say that \(g\) is a \textbf{Maximal–Minimal Graph Transformation} of \(f\) if \(g\circ f\) is a minimal graph surface. The classification of such transformations is straightforward: they are precisely the inverses of minimal–maximal graph transformations. 
  
  \subsection{Maximal Graph Transformations}
   Analogous to a minimal graph transformation, we define the following: Let \(f:\Omega\longrightarrow\mathbb{R}\) be a maximal graph surface and \(g:I=f(\Omega)\longrightarrow\mathbb{R}\) a smooth function. We say that \(g\) is a \textbf{maximal graph transformation} of \(f\) if \(g\circ f\) is a maximal graph surface. This condition can be expressed, as in Proposition~\ref{thm: Alt. def. MGT} and Equation~\ref{eqn: Alt. def. MGT}, by the equation: 
   \[
     g''(f)\,\|\nabla f\|^2-[g'^{3}-g'](f)\,\Delta f=0.
   \]
   The maximal graph surface \(f\) is real analytic because it is a spacelike solution of an elliptic PDE satisfying certain conditions, (see, hypotheses of Theorem~\ref{thm: Nonlinear Elliptic PDE: analytic solns}, Appendix: \ref{section: Appendix analytic function}). After imposing the usual analyticity requirements for maximal surfaces, the system reduces to
   \[
   \frac{\Delta f}{\|\nabla f\|^2}=\frac{g''(f)}{[g'^{3}-g'](f)}=h(f), \qquad\text{and}\qquad \text{Maximal Surface Equation},
   \]
   for some smooth function \(h\). The maximal surface equation can be written as
   \[
     (1 - f_y^2) f_{xx} + 2 f_{xy} f_x f_y + (1 - f_x^2) f_{yy} = 0
   \]   
   This system is identical to the one obtained for minimal graph transformations, except that the minimal surface equation is replaced by the maximal surface equation. However, the relation \(\frac{\Delta f}{\|\nabla f\|^2}=h(f)\) allows us to convert the system into a complex PDE in the variables \(z,\bar z\) with the additional condition that the solution be harmonic. It may be possible to apply the weakening technique to this complex PDE, and, if so, one might thereby obtain a classification of maximal graph transformations.
   
   \section{Acknowledgement}   
   The author is very grateful to his advisor Dr. Rukmini Dey, ICTS-TIFR for teaching him minimal surfaces, and for her valuable time and insightful discussions during the writing of this paper. The author also acknowledges his co‑advisor Dr. Vishal Vasan, ICTS-TIFR for interesting discussions and valuable comments.
   
   \appendix
   
    \section{Appendix: Analytic Functions} \label{section: Appendix analytic function}  
   We begin by stating an important theorem on elliptic partial differential equations, following Theorem 1.1 in \cite{Bl}:  We consider a solution \(u \in C^\infty(\Omega, \mathbb{R})\) to the nonlinear partial differential equation
   \begin{equation}
   	\label{eqn: Nonlinear Elliptic PDE}
   	\phi(x, u(x), Du(x), D^2u(x)) = 0 \quad \text{on } \Omega
   \end{equation} 
   Here, \(\Omega \subset \mathbb{R}^n\) is an open subset of the Euclidean space of dimension \(n \ge 2\). We assume that \(\phi : \Omega \times \mathbb{R} \times \mathbb{R}^n \times S(n) \to \mathbb{R}\) is a real-analytic function, where \(S(n)\) is the set of all \(n\times n\) matrices, and suppose that the resulting partial differential equation is elliptic. Then we have the following
   \begin{theorem}
   	\label{thm: Nonlinear Elliptic PDE: analytic solns}
   	Let \(u \in C^\infty(\Omega, \mathbb{R})\) be a solution to the equation \ref{eqn: Nonlinear Elliptic PDE}, \(x_0 \in \Omega\), and assume that \(\phi\) is analytic in a neighborhood of \((x_0, u(x_0), Du(x_0), D^2u(x_0))\). Then \(u\) is real analytic near \(x_0\).
   \end{theorem}
   An analogous theorem for elliptic partial differential equations in two real variables was established earlier by Bernstein in \cite{Be}. This result is of particular importance to our work, as the minimal surface equation \ref{eqn: MSE} is a nonlinear elliptic partial differential equation that satisfies all the hypotheses of Theorem \ref{thm: Nonlinear Elliptic PDE: analytic solns}. We therefore conclude that all minimal graph surfaces are real analytic functions.
   
   The real analyticity of minimal graph surfaces can also be established independently of the aforementioned theorem. This has been demonstrated in \cite{Ni}, Page 125, \S 131; we state the result here:
   \begin{theorem}
   	\label{thm: MGS is real analytic}
   	A minimal surface \( S = \{ (x, y, z(x, y)) : (x, y) \in P \} \) is analytic. In other words: a twice continuously differentiable solution to the minimal surface equation in domain \( P \) is in fact real analytic in \( P \).
   \end{theorem}

   In this paper we have mainly used Theorem \ref{thm: MGS is real analytic}. However, in the concluding section on future directions we considered maximal graph surfaces and referred to Theorem \ref{thm: Nonlinear Elliptic PDE: analytic solns} to establish their analyticity.
   
   We now state a fundamental property  of non‑zero real analytic functions: their zero sets have Lebesgue measure zero. This result is given in \cite{Mi}, Proposition~1, and we state it as follows:
   \begin{theorem}
   	\label{thm: zero set of real analytic fns}
   	Let \( A(x) \) be a real analytic function on a connected open domain \( U \subset \mathbb{R}^d \). If \( A \) is not identically zero, then its zero set \(F(A) := \{ x \in U : A(x) = 0 \}\)	has zero measure.     	
   \end{theorem}
   
   We also used several basic properties of analytic functions, for which we refer to \cite{KH}. The specific results we require can be found in Proposition 2.2.2 on page 29 and Proposition 2.2.3 on page 30 of that text.
   
   \begin{proposition}
   	\label{thm: algebra of real analytic functions}
   	Let \( U, V \subset \mathbb{R}^m \) be open. If \( f : U \to \mathbb{R} \) and \( g : V \to \mathbb{R} \) are real analytic, then \( f + g \) and \( f \cdot g \) are real analytic on \( U \cap V \), and \( f / g \) is real analytic on \( U \cap V \cap \{ x : g(x) \neq 0 \}. \)
   \end{proposition}
   
   \begin{proposition}
   	\label{thm: calculus of real analytic functions}
   	Let \( f \) be a real analytic function defined on an open subset \( U \subseteq \mathbb{R}^m \).
   	Then \( f \) is continuous and has continuous, real analytic partial derivatives of all orders.
   	Further, the indefinite integral of \( f \) with respect to any variable is real analytic.
   \end{proposition}    
   
   \section{Appendix: Elliptic Integrals and Jacobi Elliptic Functions}\label{section: Elliptic funnctions}
   In this subsection we provide a brief discussion of elliptic integrals and Jacobi elliptic functions, which arise as special integrals in numerous problems in geometry and mechanics; for further details one may refer \cite{La}, which offers a special-function perspective on these functions, or \cite{TT}, which presents a geometric viewpoint. 
   
   \begin{definition}
   	Let \( R(x, s) \) be a rational function in two variables \( x \) and \( s \), depending non-trivially on \( s \), that is, \(\frac{\partial R}{\partial s} \neq 0.\)  Let \( \varphi(x) \) be a square-free polynomial of degree three or four in \( x \); in other words, \( \varphi(x) \) possesses three or four distinct roots. Then, an integral of the form: \(\int R\bigl(x, \sqrt{\varphi(x)}\bigr)\, dx\) is called an \textit{elliptic integral}.
   \end{definition}
   
   The well-known Legendre--Jacobi reduction theorem states that any elliptic integral can be decomposed into a linear combination of elementary functions and the three fundamental kinds of elliptic integrals, namely the elliptic integrals of the first, second, and third kinds:
   
   \textbf{First Kind}: $F(\varphi, k) = \int_{0}^{\varphi} \frac{d\theta}{(1 - k^2 \sin^2 \theta)^{1/2}}
   = \int_{0}^{\sin \varphi} \frac{dx}{\sqrt{(1 - x^2)(1 - k^2 x^2)}}$
   
   \textbf{Second Kind}: $E(\varphi, k) = \int_{0}^{\varphi} (1 - k^2 \sin^2 \theta)^{1/2} \, d\theta
   = \int_{0}^{\sin \varphi} \frac{\sqrt{1 - k^2 x^2}}{\sqrt{1 - x^2}} \, dx$
   
   \textbf{Third Kind}: $\Pi(\varphi, n, k) = \int_{0}^{\varphi} \frac{d\theta}{(1 - n \sin^2 \theta)(1 - k^2 \sin^2 \theta)^{1/2}}
   = \int_{0}^{\sin \varphi} \frac{dx}{(1 - n x^2)\sqrt{(1 - x^2)(1 - k^2 x^2)}}$ where $n$ is the characteristic, $k$ is the modulus ($0 \leq k < 1$) and $\phi$ is the amplitude.\\
   
   \noindent
   \textbf{Relation Between Elliptic Integrals and Jacobi Elliptic Functions:}	The elliptic integral of the first kind defines a relationship between $u$ and $\phi$: \(u = F(\phi, k) = \int_0^{\phi} \tfrac{d\theta}{\sqrt{1 - k^2 \sin^2\theta}}.\)	The inverse of this function is called the \textit{Jacobi amplitude function}: \(\phi = \text{am}(u, k).\) Using the amplitude, the three primary Jacobi elliptic functions are defined as:
   
   \(
   \operatorname{sn}(u, k) = \sin(\phi) \qquad \operatorname{cn}(u, k) = \cos(\phi) \qquad \operatorname{dn}(u, k) = \sqrt{1 - k^2 \sin^2(\phi)}
   \)
   
   \noindent
   These are generalizations of the trigonometric sine and cosine, depending on modulus \(k\).
   
   \noindent
   \textbf{Other Jacobi Elliptic Functions:} From the three primary functions, nine more can be defined as follows: \(\text{ns}(u,k)=\frac{1}{\text{sn}(u,k)}\), \(\;\;\;\text{nc}(u,k)=\frac{1}{\text{cn}(u,k)}\), \(\;\;\;\text{nd}(u,k)=\frac{1}{\text{dn}(u,k)}\), \(\;\;\;\text{sc}(u,k)=\frac{\text{sn}(u,k)}{\text{cn}(u,k)}\), \(\;\;\;\text{sd}(u,k)=\frac{\text{sn}(u,k)}{\text{dn}(u,k)}\), \(\;\;\;\text{cd}(u,k)=\frac{\text{cn}(u,k)}{\text{dn}(u,k)}\), \(\;\;\;\text{cs}(u,k)=\frac{1}{\text{sc}(u,k)}\), \(\;\;\;\text{ds}(u,k)=\frac{1}{\text{sd}(u,k)}\) and \(\;\;\;\text{dc}(u,k)=\frac{1}{\text{cd}(u,k)}\)\\
   
   \noindent
   \textbf{Inverse Jacobi Elliptic Functions: }The inverses of Jacobi elliptic functions can be expressed using elliptic integrals of the first kind.
   
   Inverse of $\text{sn}$: \(\text{sn}^{-1}(x, k) = F(\arcsin x, k) = \int_0^x \frac{dt}{\sqrt{(1 - t^2)(1 - k^2 t^2)}}\)	
   
   Inverse of $\text{cn}$: \(\text{cn}^{-1}(x, k) = F(\arccos x, k) = \int_x^1 \frac{dt}{\sqrt{(1 - t^2)(k'^2 + k^2 t^2)}}\) where \(k' = \sqrt{1 - k^2}\)
   
   Inverse of $\text{dn}$:	\(\text{dn}^{-1}(x, k) = F(\arccos(\tfrac{x}{k}), k) = \int_x^1 \frac{dt}{\sqrt{(1 - t^2)(t^2 - k'^2)}}\)
   
   \noindent
   \textbf{Derivatives: }Their derivatives of basic Jacobi Elliptic Functions are given by:

   Derivative of \(\operatorname{sn}\): \(\frac{d}{du}\operatorname{sn}(u, k) = \operatorname{cn}(u, k)\operatorname{dn}(u, k)\)
   \vspace{0.2cm}
   
   Derivative of \(\operatorname{cn}\): \(\frac{d}{du}\operatorname{cn}(u, k) = -\operatorname{sn}(u, k)\operatorname{dn}(u, k)\)
   \vspace{0.2cm} 
   
   Derivative of \(\operatorname{dn}\): \(\frac{d}{du}\operatorname{dn}(u, k) = -k^2 \operatorname{sn}(u, k)\operatorname{cn}(u, k)\)


\begin{thebibliography}{25}
   	
   	\bibitem{La} Derek F. Lawden, \textbf{Elliptic Functions and Applications}{, Vol. 80, Springer Science \& Business Media, 2013}
   	
   	\bibitem{JD} Alan Jeffrey and Hui Hui Dai, \textbf{Handbook of Mathematical Formulas and Integrals}{, Elsevier, 2008}
   	
   	\bibitem{Os} Robert Osserman, \textbf{A Survey of Minimal Surfaces}{, Courier Corporation, 2013}
   	
   	\bibitem{Bl} Simon Blatt, \textbf{On the Analyticity of Solutions to Non-linear Elliptic Partial Differential Equations}{, arXiv preprint arXiv:2009.08762, 2020}
   	
   	\bibitem{Be} Serge Bernstein, \textbf{Sur la nature analytique des solutions des équations aux dérivées partielles du second ordre}{, \textit{Mathematische Annalen}, Vol. 59, No. 1, pp. 20--76, Springer, 1904}
   	
   	\bibitem{Le} Jiri Lebl, \textbf{Tasty Bits of Several Complex Variables}{, Lulu.com, 2019}
   	
   	\bibitem{Ni} Johannes C. Nitsche, \textbf{Lectures on Minimal Surfaces: Vol. 1}{, Cambridge University Press, 1989}
   	
   	\bibitem{Mi} B. S. Mityagin, \textbf{The Zero Set of a Real Analytic Function}{, \textit{Mathematical Notes}, Vol. 107, No. 3--4, pp. 529--530, 2020, https://doi.org/10.1134/s0001434620030189}
   	
   	\bibitem{KH} Steven G. Krantz and Harold R. Parks, \textbf{A Primer of Real Analytic Functions}{, Springer Science \& Business Media, 2002}
   	
   	\bibitem{BF} Paul F. Byrd and Morris D. Friedman, \textbf{Handbook of Elliptic Integrals for Engineers and Physicists}{, Springer, 2013}
   	
   	\bibitem{TT} Takashi Takebe, \textbf{Elliptic Integrals and Elliptic Functions}{, Springer, 2023}.
   	
   	\bibitem{AS} Milton Abramowitz and Irene A. Stegun, \textbf{Handbook of Mathematical Functions: With Formulas, Graphs, and Mathematical Tables}, Courier Corporation, 1965.
   	
   	\bibitem{PB} P.~B{\"a}ck, \textbf{B{\"a}cklund transformations for minimal surfaces}, 2015.   	
   	
   	\bibitem{Eis} L.~P. Eisenhart, \textbf{Surfaces with Isothermal Representation of Their Lines of Curvature and Their Transformations: (Second Memoir)}, Transactions of the American Mathematical Society, vol.~11, no.~4, pp.~475--486, 1910. 
   	
   	\bibitem{Thy} A.~Thybaut, \textbf{Sur la d{\'e}formation du parabolo{\"\i}de et sur quelques probl{\`e}mes qui s'y rattachent}, Annales de l'{\'E}cole Normale Sup{\'e}rieure, 3e s{\'e}rie, vol.~14, pp.~45--74, 1897.
   	
   	\bibitem{CFT} A.~V. Corro, W.~Ferreira, and K.~Tenenblat, \textbf{Minimal Surfaces Obtained by Ribaucour Transformations}, Geometriae Dedicata, vol.~96, pp.~117--150, 2003.
   	
   	\bibitem{NIST} Frank W. J. Olver, \textbf{NIST Handbook of Mathematical Functions Hardback and CD-ROM}, Cambridge University Press, 2010.
   	
\end{thebibliography}
\end{document}